%% file: zigzag.tex
\documentclass[12pt]{amsart}
\usepackage[margin=1in]{geometry}
\usepackage{amsmath}
\usepackage{amssymb}
\usepackage{amsthm}

\usepackage{comment} 

\usepackage{thmtools} 
\usepackage{thm-restate} 

\usepackage{tabularray} 

\usepackage{lscape} 
\usepackage{mathabx}
\usepackage{graphicx}
\usepackage{algorithm}
\usepackage{algorithmic}
\usepackage{mathtools}
\usepackage[table]{xcolor}
\usepackage{tikz}
\usetikzlibrary{shapes}
\usetikzlibrary{decorations.pathreplacing}
\usetikzlibrary{calc}

\usepackage[bottom,symbol]{footmisc}
\renewcommand{\thefootnote}{\fnsymbol{footnote}}

\theoremstyle{plain}
\newtheorem{theorem}{Theorem}
\newtheorem{lemma}[theorem]{Lemma}
\newtheorem{corollary}[theorem]{Corollary}

\theoremstyle{definition}
\newtheorem{definition}[theorem]{Definition}
\newtheorem{example}[theorem]{Example}
\newtheorem{conjecture}[theorem]{Conjecture}
\newtheorem{open}[theorem]{Open Problem}

\theoremstyle{remark}
\newtheorem{remark}[theorem]{Remark}

\newcommand{\calC}{\mathcal{C}} 
\newcommand{\calF}{\mathcal{F}} 
\newcommand{\calG}{\mathcal{G}}

\newcommand{\calP}{\mathcal{P}} 
\newcommand{\calQ}{\mathcal{Q}} 
 
\newcommand{\calT}{\mathcal{T}}
\newcommand{\calZ}{\mathcal{Z}}

\newcommand{\bbR}{\mathbb{R}}

\newcommand{\bbZ}{\mathbb{Z}}

\newcommand{\bfa}{\mathbf{a}}

\newcommand{\bfd}{\mathbf{d}}
\newcommand{\bfe}{\mathbf{e}}

\newcommand{\bfu}{\mathbf{u}}

\newcommand{\bfx}{\mathbf{x}}
\newcommand{\bfz}{\mathbf{z}}

\newcommand{\row}{\rightarrow}
\newcommand{\vol}{\textup{vol}}
\newcommand{\flows}{\ensuremath{\calF_n^\bbZ}}
\newcommand{\des}{\ensuremath{\mathsf{des}}}
\newcommand{\swap}{\ensuremath{\mathsf{swap}}}
\newcommand{\sz}{\ensuremath{\mathsf{sz}}}
\newcommand{\zs}{\ensuremath{\mathsf{zs}}}
\renewcommand{\phi}{\varphi}

\newcommand{\Zig}{\ensuremath{\textup{Zig}}}
\newcommand{\cZig}{\widehat{\Zig}}

\usetikzlibrary{shapes.geometric}
\def\Square{\hspace{.01in}\tikz{
            \node[regular polygon,regular polygon sides=4,draw,inner sep=2.75pt]{}}\hspace{.01in}} 

\newcommand*\squared[1]{\tikz[baseline=(char.base)]{
            \node[shape=rectangle,draw,inner sep=2pt] (char) {\tiny #1};}}

\newcommand*\triangled[1]{\tikz[baseline=(char.base)]{
            \node[regular polygon,regular polygon sides=3, inner sep=0.1pt,draw,fill=black!5] (char) {\tiny #1};}}

\title[A Triangulation of the Flow Polytope of the Zigzag Graph]{A Triangulation of the Flow Polytope\\ of the Zigzag Graph}

\author{Rachel Brunner}
\address{Department of Mathematics \\ CUNY Graduate Center \\ 365 5th Ave \\ New York, NY~10016, U.S.A.}
\email{\tt rachel.brunner26@cuny.edu}

\author{Christopher R.\ H.\ Hanusa}
\address{Department of Mathematics \\ Queens College (CUNY) \\ 65-30 Kissena Blvd. \\ Queens, NY 11367-1597, U.S.A.}
\email{\tt chanusa@qc.cuny.edu}

\begin{document}

\begin{abstract}
We show that the dual graph of the triangulation of the flow polytope of the zigzag graph adorned with the length-reverse-length framing is a subgraph of a grid graph. Through M{\'e}sz{\'a}ros, Morales, and Striker's bijection between simplices of the triangulation, integer flows of a different, supplemental flow polytope, we provide a simple numerical characterization of the adjacency between the triangulation's simplices in terms of their corresponding integer flows. The proofs result from the development of Postnikov and Stanley's sequences of noncrossing bipartite trees as combinatorial objects we call groves. We propose two new statistics derived from this construction that we conjecture recover the $h^*$-polynomial of the flow polytope of the zigzag graph.
\end{abstract}



{\let\thefootnote\relax\footnote{Version of \today.}}

\maketitle

\section{Introduction}
\label{sec:introduction}

\subsection{Flow Polytopes}

Let $G$ be a (directed) graph with vertex set $V=[n]:=\{1,\hdots,n\}$ and edge set $E$. Designate a {\em net flow vector} $\bfa=(a_1,\hdots,a_n)\in \bbZ^n$ to specify an amount of flow $a_i$ generated or destroyed at each vertex $i\in V$. A {\em flow} on $G$ with net flow $\bfa$ is a function $\phi:E\row \bbR$ that satisfies {\em conservation of flow} at every vertex $i\in V$; that is,
\[\sum_{e\,\in\, \textup{in}(i)} \phi(e) + a_{i}= \sum_{e\,\in\, \textup{out}(i)} \phi(e),\]
where $\textup{in}(i)$ and $\textup{out}(i)$ are the sets of incoming and outgoing edges of~$i$, respectively. The {\em flow polytope} $\mathcal{F}_{G}(\bfa)$ is the set of all flows on $G$ with net flow $\bfa$. The simplest choice for $\bfa$ is the vector $\bfu=(1,0,\hdots,0,-1)\in \bbZ^n$; in this case the flow polytope $\mathcal{F}_{G}(\bfu)$ embodies all ways to send one unit of flow through $G$ from vertex $1$ to vertex $n$. 

Two polytopes $\calP\subset \bbR^n$ and $\calQ\subset \bbR^m$ are said to be {\em integrally equivalent} if there is an affine transformation $f: \bbR^n\rightarrow \bbR^m$ whose restriction to $\calP$ is a bijection to $\calQ$ that preserves the lattice. Two integrally equivalent polytopes can be thought of as essentially the same polytope---they have the same face structure, normalized volume, and Ehrhart polynomial. 

Flow polytopes have been studied in multiple contexts over the past few decades in algebra, combinatorics, geometry, and representation theory. There are connections to representation theory and Lie algebras through the study of Kostant partition functions \cite{baldoni08,mm15}. They have been used to study Tesler matrices and diagonal harmonics \cite{mmr17,lmm19}. They are related to subdivision algebras and Grothendieck polynomials \cite{ms20}. Toric varieties associated to flow polytopes connect to moduli spaces of quiver representations \cite{hill96,dj22}.

\subsection{Zigzag graphs}

This article explores flow polytopes of zigzag graphs. The {\em zigzag graph} $\Zig_{n}$ is a multigraph with vertices $[n]$ and two types of directed edges: $(i,i+1)$ for $1\leq i\leq n-1$ and $(i, i+2)$ for $1\leq i\leq n-2$. It has many nice properties---it is a planar graph, a spinal graph, and the distance graph $G(2,n)$. (See \cite{dleon23} for the latter definitions.)

This graph gets its name from the zigzag poset $\calZ_n$, defined by the covering relations $1\prec 2 \succ 3 \prec 4 \succ \cdots\, n$. The poset $\calZ_n$ is intricately related to {\em alternating permutations}---permutations $\alpha_1\alpha_2\cdots\alpha_n$ such that $\alpha_1<\alpha_2>\alpha_3<\alpha_4>\cdots~\alpha_n$. The linear extensions of the zigzag poset, when read as permutations in one-line notation, are the inverses of the alternating permutations. The reader is referred to \cite{stanley2010survey,petersen2024zigzag}. 

The poset $\calZ_{n}$ also appears naturally as the truncated dual graph of the planar drawing of $\Zig_{n+2}$ as in Figure~\ref{fig:zigzag}. Because of this, a result of Postnikov \cite[Theorem~3.11]{mms19} proves that $\mathcal{F}_{\Zig_{n+2}}(\bfu)$ is integrally equivalent to the order polytope $\mathcal{O}(\calZ_{n})$. The object $\mathcal{O}(\calZ_{n})$ is studied by Coons and Sullivant \cite{coons19} and its $h^*$-polynomial is described in terms of a $\swap$ statistic. We discuss these connections in more detail in Section~\ref{sec:hstar}.

\begin{figure}
\begin{center}
    \input{figures/zigzag.tikz}
\end{center}
\caption{The zigzag graph $\Zig_7$ (in black) and the zigzag poset $\calZ_5$ (in red).}
    \label{fig:zigzag}
\end{figure} 

Following \cite[Example~2.15]{dleon23}, we create the {\em contracted zigzag graph} $\cZig_{n}$ from $\Zig_{n+2}$ by contracting edges $(1,2)$ and $(n+1,n+2)$ and reindexing the vertices to run from $1$~to~$n$ as shown in Figure~\ref{fig:czig}. As a result, every {\em inner vertex} ($2$ through $n-1$) of $\cZig_{n}$ has in-degree and out-degree $2$.

The edges $x_i=(i,i+1)$ for $1\leq i\leq n-1$ are called {\em slack edges} and the edges $y_0=(1,2)$, $y_{n-1}=(n-1,n)$, and $y_i=(i,i+2)$ for $1\leq i\leq n-2$ are called {\em nonslack edges}. For the inner vertices, define the set of {\em incoming edges} $\textup{In}(i)=\{x_{i-1},y_{i-2}\}$ and the set of {\em outgoing edges} $\textup{Out}(i)=\{x_i,y_i\}$.

\begin{figure}
\begin{center}
    \input{figures/contracted.tikz}
\end{center}
\caption{The contracted zigzag graph $\cZig_{n}$}
    \label{fig:czig}
\end{figure}     

A path from $1$ to $n$ will be called a {\em route}. A path from $1$ to $i$ will be called a {\em prefix} and a path from $i$ to $n$ will be called a {\em suffix}.

In \cite[Lemma~2.2]{ms20} and \cite[Proposition~2.12]{dleon23} the authors prove that the flow polytopes $\mathcal{F}_{\cZig_{n}}\!(\bfu)$ and $\mathcal{F}_{\Zig_{n+2}}(\bfu)$ are integrally equivalent, so this modification does not affect the combinatorics. From now on, we focus our study only on the contracted zigzag graph. 

\subsection{Framings and DKK Triangulations}

We aim to understand the structure of the flow polytope $\calF_{\cZig_{n}}(\bfu)$ by exploring its decomposition into unimodular triangulations; an example follows in Section~\ref{sec:example}. In \cite{dkk12}, Danilov, Karzanov, and Koshevoy introduce a way to generate many different triangulations of $\calF_{G}(\bfu)$ through different framings on the graph $G$. Following the presentation of Gonz\'alez D'Le\'on et al.\ \cite{gmpty23}, a {\em framing} $F$ on $G$ is a choice of linear orders $\prec_{\textup{In}(i)}$ and $\prec_{\textup{Out}(i)}$ associated to the sets of incoming and outgoing edges of every inner vertex $i\in V(G)$. 

We visualize a framing pictorially by arranging the incoming and outgoing edges in the linear order from lowest to highest. Figure~\ref{fig:framing} shows the three different framings described below on $\cZig_5$.

Two framings have been the focus of recent study---the planar framing \cite{vonbell21,mms19} and the length framing \cite{vonbell22,vonbell21}. The planar framing is defined when $G$ is a planar graph and is drawn with a planar embedding; the order of the edges is determined by the relative spacial positions of the edges. In the length framing, edges are ordered by their lengths---longer edges come before shorter edges. 

In this paper we use the {\em length-reverse-length (LRL) framing}. The LRL framing applies different orderings to the incoming edges and the outgoing edges. On incoming edges, the ``longer'' nonslack edges are first and the ``shorter'' slack edges are second (so $y_{i-2}\prec_{\textup{In}(i)} x_{i-1}$ at vertex $i$); whereas, on the outgoing edges, the slack edges are first and the nonslack edges are second (so $x_i\prec_{\textup{Out}(i)} y_i$ at vertex $i$). See Figure~\ref{fig:framing}(c). 

\begin{figure}[t]
\begin{center}
    \input{figures/framings.tikz}
\end{center}
\caption{Three framings of $\cZig_5$: (a) the planar framing, (b) the length framing, and (c) the length-reverse-length framing.}
\label{fig:framing}
\end{figure}

\medskip
A framing induces a total order on the sets of prefixes and suffixes at all inner vertices $i$, as follows. Let $P$ and $P'$ be two different prefixes at $i$. There is some last vertex $j$ (at $i$ or before) where $P$ and $P'$ come together, along edges $e$ and $e'$, respectively. We order $P\prec P'$ if $e\prec_{\textup{In}(j)}e'$. Similarly, if $Q$ and $Q'$ are two different suffixes at $i$, there is some first vertex $k$ (at $i$ or after) where $Q$ and $Q'$ diverge, along edges $e$ and $e'$; we order $Q\prec Q'$ if $e\prec_{\textup{Out}(k)}e'$. See Figure~\ref{fig:prefixorder}.

\begin{figure}[h]
\begin{center}

\input{figures/prefixorders.tikz}

\end{center}
\caption{The total orders on prefixes (and suffixes)  at inner vertices $i$ are determined by the framing ordering of the edges at the nearest vertex where they diverge.}
\label{fig:prefixorder}
\end{figure}

We use the notation $PiQ$ to represent the route that follows prefix $P$ from vertex $1$ until~$i$ and then continues along suffix $Q$ until vertex $n$. Two routes $R=PiQ$ and $R'=P'iQ'$ that visit vertex $i$ have a conflict at $i$ if the ordering of the prefixes is the opposite of the ordering of the suffixes, for example, $P\prec P'$ and $Q'\prec Q$. We say that a collection of routes is {\em coherent} if no two of its routes has a conflict at any vertex $i$. 

A {\em maximal clique} is a maximal set of coherent routes. For a framing $F$ on a graph $G$, we denote by $\calC_G^F$ the set of all maximal cliques. Since every route is a vertex of $\calF_G(\bfu)$, we can take the convex hull of these vertices and interpret a maximal clique $C$ as a simplex $\Delta_C\subseteq\calF_G(\bfu)$. For a graph $G$ with a framing $F$, the {\em DKK~triangulation} $\calT_G^F$ of $\calF_G(\bfu)$ is the simplicial complex determined by the set of simplices 
\[\calT_G^F=\{\Delta_C \textup{ for $C\in\calC_G^F$}\}.\] 
and their subfaces. This construction gives a regular unimodular triangulation of $\calF_G(\bfu)$, as proved in \cite{dkk12}.

\begin{theorem}[\protect{\cite[Theorems~1~and~2]{dkk12}}]
    $\calT_G^F$ is the collection of the maximal cells of a regular triangulation of $\calF_G(\bfu)$.
\end{theorem}

When $G=\cZig_n$ and $F$ is the LRL framing, we simplify notation and write $\calT_n^{LRL}$ for $\calT_{\cZig_n}^{LRL}$.

We say that two simplices in a triangulation are {\em adjacent} if they share a co-dimension 1 face, or equivalently, their sets of vertices differ by exactly one vertex. Therefore, adjacent simplicies in a DKK triangulation correspond to maximal cliques that differ by exactly one route. Danilov, Karzanov, and Koshevoy explain exactly how these two differing routes are related.

\begin{theorem}[\protect{\cite[Proposition~3]{dkk12}}]
\label{thm:bump}
    Two maximal cliques $C$ and $C'$ correspond to two adjacent simplices in $\calT_G^F$ if the routes $R=PiQ$ and $R'=P'iQ'$ that they differ by (a) share a common inner vertex $i$ and (b) $P$ and $P'$ are consecutive in the total ordering of prefixes at $i$ and $Q$ and $Q'$ are consecutive in the total ordering of suffixes at $i$. 
\end{theorem}

Given a triangulation $\calT$ of the flow polytope $\calF_G(\bfu)$, we can construct the {\em dual graph} of the triangulation, $D(\calT)$, with vertex set $\{\Delta\,\vert\, \Delta \textup{ is a simplex in }\calT\}$ and edges between two simplices if they are adjacent. The main focus of this paper is to determine the structure of~$D(\calT_{n}^{\textup{LRL}})$.

\subsection{Integer flows}
\label{sec:integerflows}

For a graph $G$ with vertices $\{1,\hdots,n\}$, define its {\em shifted in-degree vector} to be 
\[\textstyle \bfd=\left(0,d_2,\hdots,d_{n-1},-\sum_{i=2}^{n-1} d_i\right),\]
where $d_i=\textsf{indeg}_G(i)-1$ for an inner vertex $i$. For $\cZig_n$, \(\bfd=\big(0,1,1,\hdots,1,-(n-2)\big)\).

M\'esz\'aros, Morales, and Striker \cite[Section~7]{mms19} show that for any framing $F$, there is a bijection between $\calT_G^F$ and the set $\calF_{G}^{\bbZ}(\bfd)$ of integer flows on $G$ with net flow vector $\bfd$, which depends on $F$ and which we discuss in Section~\ref{sec:groves}. 
We will again simplify notation and write $\calF_n^{\bbZ}$ for $\calF_{\cZig_n}^{\bbZ}\!\!(\bfd)$.

\subsection{An example} 
\label{sec:example}

Let $G=\cZig_4$. The flow polytope $\calP=\mathcal{F}_G(\bfu)$ lives in $\bbR^7$ because there are seven edges in $G$. However, the three independent constraints given by the conservation of flow equations make $\calP$ a four-dimensional polytope. As such, a simplex in a triangulation of $\calP$ is defined by its five vertices.  The polytope $\calP$ is the convex hull of the eight routes in~$G$, labelled  
\hspace{.02in}{\raisebox{.02in}{\squared{1}}}\hspace{.02in} 
through 
\hspace{.02in}{\raisebox{.02in}{\squared{8}}}\hspace{.01in} and shown in Figure~\ref{fig:zig4unitflows}. 

\begin{figure}
\begin{center}
    \input{figures/zig4unitflows.v2.tikz}
\end{center}
\caption{$\cZig_4$'s routes.}
\label{fig:zig4unitflows}
\end{figure}

The LRL framing induces a total ordering on the prefixes and suffixes at vertices 2 and~3 of $G$. At vertex $2$, the prefixes are ordered $y_0\prec x_1$ and the suffixes are ordered $x_2x_3 \prec x_2y_3 \prec y_2$. At vertex $3$, the prefixes are ordered $y_1\prec y_0x_2\prec x_1x_2$ and the suffixes are ordered $x_3 \prec y_3$. The numbering of routes \hspace{.02in}{\raisebox{.02in}{\squared{1}}}\hspace{.02in} 
through 
\hspace{.02in}{\raisebox{.02in}{\squared{8}}}\hspace{.01in} is inherited from a similar notion of a prefix order at vertex $4$.

The normalized volume of $\calP$ is 5, so every triangulation of $\calP$ consists of five simplices. We generate the DKK triangulation of~$\calP$ for the LRL framing, consisting of maximal cliques of routes specifying the vertices of each simplex; see Figure~\ref{fig:zig4example}. We label the simplices {\triangled{$\mathsf{A}$}} through {\triangled{$\mathsf{E}$}}. 

\begin{figure}
\begin{center}
    \input{figures/zig4example.tikz}
\end{center}
\caption{The combinatorics of the DKK triangulation $\calT^{\textup{LRL}}_G$ of $G=\cZig_4$. Each simplex corresponds to an integer flow $\phi\in\calF_G^{\bbZ}(0,1,2,-3)$, a maximal clique $C$ of five routes, and a grove $\Gamma$ of two noncrossing bipartite trees.}
    \label{fig:zig4example}
\end{figure}

Note that routes \hspace{.02in}\raisebox{0.01in}{\squared{2}}\hspace{.02in} and \hspace{.02in}\raisebox{0.01in}{\squared{6}}\hspace{.02in} appear as vertices in every simplex. If we project onto the appropriate two-dimensional subspace, we can visualize the adjacency relations between the simplices in $\calT_4^{\textup{LRL}}$ in Figure~\ref{fig:zig4projection}; the other three vertices of each simplex are at its corners. (The reader should note that all vertices of $\calP$ lie on its boundary; it is only because of the projection that vertex \hspace{.02in}\raisebox{0.01in}{\squared{4}}\hspace{.02in} appears to be in the interior of~$\calP$.)

\begin{figure}
\begin{center}
\input{figures/zig4projection.tikz}
\end{center}
    \caption{The adjacency structure of the five simplices of $\calT_4^{\textup{LRL}}$.}
    \label{fig:zig4projection}
\end{figure}

We can see the adjacency between Simplex {\triangled{$\mathsf{A}$}} and Simplex {\triangled{$\mathsf{B}$}} as exchanging route \hspace{.02in}\raisebox{0.01in}{\squared{8}}\hspace{.02in} ($x_1x_2x_3$) for route \hspace{.02in}\raisebox{0.01in}{\squared{4}}\hspace{.02in} ($y_0x_2y_3$). This is a manifestation of Theorem~\ref{thm:bump} because setting $i=2$, $P=x_1$, $Q=x_2x_3$, $P'=y_0$, and $Q'=x_2y_3$, we see that $P'\prec P$ and $Q\prec Q'$ as consecutive prefixes and suffixes at vertex $2$. 

\subsection{Our Main Results}

When we concern ourselves only with the graph $G=\cZig_n$ and the length-reverse-length framing, we can understand the structure of the dual graph $D(\calT_n^{\textup{LRL}})$. 

We are able to show that adjacency of simplices in this triangulation translates to a predictable numerical condition between the corresponding integer flows, which we call an elementary move. Here is how an elementary move applies to a flow $\phi$.

\begin{definition}
\label{def:flowmove}
    Define the {\em elementary moves} $m_i^+$ and $m_i^-$ from $\calF_n^\bbZ$ to $\calF_n^\bbZ$ for $2\leq i\leq n-1$ as follows. 
    
    If $\phi(x_{i})>0$ and $\phi(y_{i+1})>0$, define $m_i^+(\phi)$ to be the flow that subtracts one from $\phi(x_i)$ and $\phi(y_{i+1})$ and adds one to $\phi(x_{i+2})$ and $\phi(y_{i})$. If $\phi(x_i)=0$ or $\phi(y_{i+1})=0$, $m_i^+$ does not apply to $\phi$.

    If $\phi(x_{i+2})>0$ and $\phi(y_{i})>0$, define $m_i^-(\phi)$ to be the flow that subtracts one from $\phi(x_{i+2})$ and $\phi(y_{i})$ and adds one to $\phi(x_i)$ and $\phi(y_{i+1})$. If $\phi(x_{i+2})=0$ or $\phi(y_{i})=0$, $m_i^-$ does not apply to $\phi$.

    (When $i$ is $n-2$ or $n-1$, ignore all conditions and operations on nonsensical values such as $\phi(x_n)$, $\phi(y_n)$, or $\phi(x_{n+1})$.)
\end{definition}

These elementary moves on integer flows extend to elementary flows on simplices through M\'esz\'aros, Morales, and Striker's bijection; we will use the same notation for these moves on simplices. Consider the simplices in Section~\ref{sec:example}. We~can compute that $m_2^+\big(\raisebox{-.02in}{\,\triangled{$\mathsf{B}$}\,}\big)=\raisebox{-.02in}{\triangled{$\mathsf{D}$}}$, $m_3^+\big(\raisebox{-.02in}{\,\triangled{$\mathsf{B}$}\,}\big)=\raisebox{-.02in}{\triangled{$\mathsf{C}$}}$, and $m_3^-\big(\raisebox{-.02in}{\,\triangled{$\mathsf{B}$}\,}\big)=\raisebox{-.02in}{\triangled{$\mathsf{A}$}}$. On the other hand, $m_2^-$ does not apply to \raisebox{-.02in}{\triangled{$\mathsf{B}$}} because $\phi(y_2)=0$. It is not a coincidence that we have found all the simplices adjacent to \raisebox{-.02in}{\triangled{$\mathsf{B}$}} in the triangulation. Indeed, we will prove the following result.

\begin{restatable}{theorem}{thmElementaryMoves}
\label{thm:ElementaryMoves}
    If two simplices $\Delta$ and $\Delta'$ of $\calT_n^{LRL}$ are adjacent, then their corresponding integer flows $\phi(\Delta)$ and $\phi(\Delta')$ are adjacent; that is, they differ by one of the elementary moves.
\end{restatable}

The elementary moves play the role of $n-2$ orthogonal directions along which one simplex can be adjacent to another, which allows us to embed the dual graph $D(\calT_n^{\textup{LRL}})$ in an integer lattice. Let $\Square_n$ be the $n$-dimensional square grid graph. We will prove the following.

\begin{restatable}{theorem}{thmEmbedding}
\label{thm:embedding}
    The dual graph $D(\calT_n^{\textup{LRL}})$ is a subgraph of \hspace{.02in}$\Square_{n-2}$ via the embedding $\Delta\mapsto \bfx(\Delta)$.
\end{restatable}

In our example from Section~\ref{sec:example}---the dual graph $D(\calT_4^{\textup{LRL}})$ in Figure~\ref{fig:zig4projection} can be embedded in $\bbZ^2$ in a straightforward manner. Similarly, $D(\calT_5^{\textup{LRL}})$ can be embedded in~$\bbZ^3$. See Figure~\ref{fig:dual_graph_4_5}.

The proofs of these theorems are delayed until Section~\ref{sec:grove_operations}, in which we give explicit coordinates for the vertices. Our embedding shows that the square grid graph is the right framework for $D(\calT_n^{\textup{LRL}})$; the dual graphs $D(\calT_{\cZig_n}^{F})$ for other framings $F$ are less regular.

\begin{figure}
\begin{center}
\raisebox{0.2in}{\input{figures/embedzig4.tikz}}\qquad\qquad\qquad\qquad \input{figures/embedzig5.tikz}
\end{center}
    \caption{Embeddings of $D(\calT_4^{\textup{LRL}})$ in $\bbZ^2$ and $D(\calT_5^{\textup{LRL}})$ in $\bbZ^3$.}
    \label{fig:dual_graph_4_5}
\end{figure}

\smallskip
This is half the story. We also present a numerical method for determining which neighboring points in the integer lattice correspond to adjacent simplices in $D(\calT_n^{\textup{LRL}})$.  

\begin{definition}
\label{def:flowoffsets}
    Given an integer flow $\phi\in\calF_n^\bbZ$ and an inner vertex $i$, define the {\em sequence of offsets of $\phi$ at $i$}, $\bfz^{(i)}(\phi)=(z_{i+1},\hdots,z_{n-1})$, as follows. 
    
    If $i=n-1$, define $\bfz^{(i)}(\phi)$ to be the empty sequence and STOP. 
    
    Otherwise, initialize $z_{i+1}=\phi(y_{i+1})$. Then, as long as some entry of $\bfz^{(i)}(\phi)$ is not defined, find the largest index $j$ for which $z_{j}$ is defined, and compute subsequent terms: 
    
    If $z_j<0$, define $z_{j+1}=z_j+\phi(y_{j+1})$. 
    
    If $z_j>0$, define $z_{j+1}=\varnothing$ and, if $j<n-2$, $z_{j+2}=z_j-\phi(x_{j+2})$. 
    
    If $z_j=0$, define $z_k=0$ for all $k\geq j$. 
\end{definition}

\begin{restatable}{theorem}{thmAdjacencyCriterion}
\label{thm:AdjacencyCriterion}
Let $i$ be an inner vertex and let $\phi\in\calF_n^\bbZ$ be an integer flow where $\phi'=m_i^+(\phi)$ is defined.  The two simplices in $\calT^{\textup{LRL}}$ corresponding to $\phi$ and $\phi'$ are adjacent in $D(\calT_n^{\textup{LRL}})$ if and only if the sequence of offsets of $\phi$ at $i$ 
contains no zeroes.
\end{restatable}

The proof of Theorem~\ref{thm:AdjacencyCriterion} will be provided in Section~\ref{sec:grove_adjacency}.

\begin{example}
\label{ex:adjacentflowexample}
    Consider the integer flows $\phi_1$ and $\phi_2$ in $\calF_6^{\bbZ}$ shown in Figure~\ref{fig:adjacentflows}. The elementary move $m_2^+$ applies to both $\phi_1$ and $\phi_2$, giving $\phi_1'$ and $\phi_2'$ respectively. The sequences of offsets at $2$ are 
    \[\textup{$\bfz^{(2)}(\phi_1)=(2,\varnothing,-1)$ \qquad and \qquad $\bfz^{(2)}(\phi_2)=(1,\varnothing,0)$.}\]
    Theorem~\ref{thm:AdjacencyCriterion} shows that the simplices in $D(\calT_6^{\textup{LRL}})$ corresponding to $\phi_1$ and $\phi_1'$ are adjacent while the simplices in $D(\calT_6^{\textup{LRL}})$ corresponding to $\phi_2$ and $\phi_2'$ are not adjacent. 
\end{example}

\begin{figure}
    \input{figures/adjacentflows.tikz}
    \caption{Two integer flows $\phi_1$ and $\phi_2$ and their respective images $\phi_1'$ and~$\phi_2'$ under $m_2^+$. See Example~\ref{ex:adjacentflowexample}.}
    \label{fig:adjacentflows}
\end{figure}   

\medskip
Theorems~\ref{thm:ElementaryMoves}, \ref{thm:embedding}, and \ref{thm:AdjacencyCriterion} rely heavily on a third family of combinatorial objects---sequences of noncrossing bipartite trees we call groves---that are in bijection with the simplices and the integer flows. We have found this construct to be an essential part of this story. Indeed, the focus on groves as a combinatorial object worthy of study is a new perspective we bring to the discussion. The general theory of groves is developed and discussed in Section~\ref{sec:groves}. The connections between the families of objects are presented here.

\medskip
\begin{center}
{\(~\begin{array}{ c c c c c c c }
\begin{array}{c}
\mathrm{Simplices}  \\
\Delta\in\calT_G^F
\end{array}
& \stackrel{\textup{[DKK12]}}{\longleftrightarrow} & 
\begin{array}{c}
\mathrm{Maximum~Cliques}\\
C\in\calC_G^F
\end{array}
& \stackrel{\textup{Thm~\ref{thm:bijectionCG}}}{\longleftrightarrow} & 
\begin{array}{c}
\mathrm{Groves} \\
\Gamma\in\calG_G^F
\end{array}
& \stackrel{\textup{Thm~\ref{thm:bijectionGF}}}{\longleftrightarrow} &
\begin{array}{c}
\mathrm{Integer~flows} \\
\phi\in\calF_G^\bbZ(\bfd)
\end{array}
\end{array}\)}
\end{center}

\medskip
Section~\ref{sec:sec4} conjectures two new statistics---$\mathsf{sz}$ and $\mathsf{zs}$---for the $h^*$-polynomial of the flow polytope $\mathcal{F}_{\cZig_n}(\bfu)$ that arise from the structure of $D(\calT_n^{\textup{LRL}})$. Our examples show that $\mathsf{sz}$ and~$\mathsf{zs}$ are distinct from (but equidistributed with) statistics used to find the $h^*$-polynomial in the past---the $\mathsf{swap}$ statistic given in \cite{coons19} and the descent statistic given in \cite{ajr20,dleon23}. We conclude this article with questions for future study.


\section{Groves}
\label{sec:groves}

\subsection{Introducing groves}

We now discuss another family of combinatorial objects that are in bijection with the maximal cliques and integer flows introduced in Section~\ref{sec:introduction}. Each object is a sequence of trees, so we affectionately call them {\em groves}\footnote[2]{These {\em flow groves} have no relation to Carroll and Speyer's {\em domino tiling groves}.}. Groves, referred to as ``tuples of noncrossing bipartite trees'' in \cite{mms19}, were introduced by Postnikov in unpublished work, were first used in publication by M\'ez\'aros and Morales in \cite{mm15}, and were used to establish the bijection between $\calT_G^F$ and $\calF_G^\bbZ(\bfd)$ by M\'ez\'aros, Morales, and Striker in \cite{mms19} where they form the backbone of framed Postnikov-Stanley triangulations. 

We wish to highlight groves as fundamental combinatorial objects in their own right because they provide a balance between the complexity of maximal cliques and the simplicity of integer flows. They appear to have just enough information about the structure of an integer flow to form the basis for key proofs in \cite{mms19} and here.

\bigskip
Given a graph $G$ with vertex set $[n]$ and a framing $F$ on $G$, a {\em grove} $\Gamma=(\gamma_2,\hdots,\gamma_{n-1})$ consists of $n-2$ noncrossing bipartite trees. Each tree $\gamma_i$ has left vertices $l_1, l_2, \hdots, l_{p_i}$ and right vertices $r_1, r_2, \hdots, r_{o_i}$, where $o_i$ is the number $\lvert \textup{Out}(i)\rvert$ of outgoing edges from $i$ in $G$. The number $p_i$ of left vertices\footnote[3]{In Section~\ref{sec:CliquesToGroves} we see left vertices correspond to route prefixes, hence the variable $p_i$.} is determined by the structure of the previous trees $\gamma_2$ through $\gamma_{i-1}$ as follows. 

Every edge $e\in E(G)$ is either incident to the source vertex or associated to a vertex $r_j$ in a tree $\gamma_i$. Define $\textup{deg}(e)$ to be $1$ if $e$ is incident to the source vertex or the degree of the associated $r_j$ in $\gamma_i$. Define $p_i = \sum_{e\in In(i)} \textup{deg}(e)$ for $2\leq i\leq n-1$.

An important consequence of this construction is that the tree $\gamma_i$ depends only on the structure of the graph $G$ at vertex $i$ and {\em before}, and not on vertices {\em after} vertex $i$.

We denote by $\calG_G^F$ the set of all groves associated to a graph $G$ with a framing $F$.  We will simplify notation by writing $\calG_n^{\textup{LRL}}$ for $\calG_{\cZig_n}^{\textup{LRL}}$.

\begin{example}
Let us determine the set $\calG_{4}^{\textup{LRL}}$ of groves $\Gamma=(\gamma_2,\gamma_3)$ for $\cZig_4$ with the LRL framing.

Tree $\gamma_2$ has two left vertices because its two incoming edges $y_0$ and $x_1$ start at the source vertex so they each contribute $1$ to the sum of the degrees. Further,  $\gamma_2$ has two right vertices $r_1$ and $r_2$ corresponding to the two outgoing edges $x_2$ and $y_2$, respectively. Therefore there are exactly two possible noncrossing bipartite trees for $\gamma_2$, shown here.

\begin{figure*}[h]
\raisebox{.11in}{$\gamma_2^{(1)}=$\hspace{.1in}}\begin{tikzpicture}[scale=0.75]
\draw[-,thick] (0,0) -- (1,0) -- (0,1) -- (1,1);
\end{tikzpicture} \qquad\qquad 
\raisebox{.11in}{$\gamma_2^{(2)}=$\hspace{.1in}}\begin{tikzpicture}[scale=0.75]
\draw[-,thick] (1,0) -- (0,0) -- (1,1) -- (0,1);
\end{tikzpicture}
\end{figure*}

Tree $\gamma_3$ must have two right vertices because vertex $3$ has two outgoing edges, $x_3$ and $y_3$. The number of left vertices in $\gamma_3$ depends on $\gamma_2$. We need to determine the contribution to the sum of the degrees from $y_1$ and $x_2$. The contribution from $y_1$ is always $1$ since it starts at vertex $1$. Since $x_2$ is first in $\prec_{\textup{Out}(2)}$, the contribution from $x_2$ depends on the degree of $r_1$ in $\gamma_2$ (its lower right vertex). When $\gamma_2=\gamma_2^{(1)}$, the degree of $r_1$ is 2, so there are three left vertices in $\gamma_3$. When $\gamma_2=\gamma_2^{(2)}$, the degree of $r_1$ is 1, so there are two left vertices in $\gamma_3$. 

Enumerating the possible noncrossing bipartite trees on either 3 and 2 or 2 and 2 vertices, we determine that $\calG_{4}^{\textup{LRL}}$ consists of the following five groves.

\begin{figure*}[h]
\input{figures/fivegroves.tikz}
\vspace{-.1in}
\end{figure*}
\end{example}

\begin{remark}
    When vertex $i$ has out degree $2$, as in the case of all inner vertices of $\calG_{n}^{\textup{LRL}}$, the tree $\gamma_i$ has two right vertices $r_1$ and $r_2$; its left vertices consist of vertices adjacent to $r_1$, vertices adjacent to $r_2$, with exactly one left vertex adjacent to both.
\end{remark}

\subsection{Maximal cliques and groves}
\label{sec:CliquesToGroves}

Let $F$ be a framing on a graph $G=([n],E)$, and let~$C$ be a maximal clique made up of routes $R_1$ through $R_{|E|-n+2}$. We will construct the grove $\Gamma =\Gamma(C)=(\gamma_2,\hdots,\gamma_{n-1})$ corresponding to $C$.

At an inner vertex $i$, identify the distinct prefixes that arrive at $i$ and occur in at least one route of $C$. Create one left vertex in $\gamma_i$ for each prefix, ordered from bottom to top by the total ordering induced by the framing. Create one right vertex in $\gamma_i$ for each outgoing edge from $i$, ordered from bottom to top by the framing. The edges of $\gamma_i$ connect left vertices corresponding to prefixes to right vertices corresponding to outgoing edges whenever there is a route that contains that prefix and outgoing edge. This construction is well defined for all cliques $C$. 

This construction where left vertices correspond to prefixes and right vertices correspond to outgoing edges was presented in \cite{mms19}; they develop this construction further by recursively computing a sequence of auxiliary graphs through which they arrive at the corresponding routes in $C$. We do not need that entire construction; Theorem~7.8 of \cite{mms19} establishes the following. 

\begin{theorem}[\protect{\cite[Theorem~7.8]{mms19}}]
\label{thm:bijectionCG}
    The above rule that takes $C$ to $\Gamma(C)$ is a bijection between $\calC_G^F$ and~$\calG_G^F$.
\end{theorem}

By construction, every tree $\gamma_i$ in $\cZig_n$ has two right vertices corresponding to outgoing edges $x_i$ and $y_i$, in that order from bottom to top.

\begin{example}
    Consider the maximal clique $C$ associated to simplex~{\triangled{$\mathsf{A}$}} in Figure~\ref{fig:zig4example}. We will show that $\Gamma(C)$ is the grove $\Gamma=(\gamma_2,\gamma_3)$ shown directly below it. 

    To construct $\gamma_2$, notice that at vertex $2$ there are two distinct prefixes, $y_{0}\prec x_{1}$, and two outgoing edges, $x_2\prec y_2$. The prefix $y_0$ only occurs in route \hspace{.02in}\raisebox{0.01in}{\squared{7}}\hspace{.02in} and continues along the edge~$x_2$. This gives the bottom edge of $\gamma_2$. The prefix $x_1$ occurs in routes \;\!\raisebox{0.01in}{\squared{2}}\;\!, \;\!\raisebox{0.01in}{\squared{5}}\;\!, and \;\!\raisebox{0.01in}{\squared{8}}~\!. The first route continues along $y_2$ while the latter two routes continue along edge $x_2$. Therefore in $\gamma_2$, the upper left vertex corresponding to $x_1$ connects to both right vertices.

    Similarly, $\gamma_3$ is constructed by identifying the three prefixes that occur as routes passing through vertex $3$, $y_1\prec y_0x_2\prec x_1x_2$, and the two outgoing edges $x_3\prec y_3$. All three prefixes continue along edge $x_3$ (routes \;\!\raisebox{0.01in}{\squared{6}}\;\!, \;\!\raisebox{0.01in}{\squared{7}}\;\!, and \;\!\raisebox{0.01in}{\squared{8}}\;\!) and therefore all left vertices are adjacent to the bottom right vertex, while only prefix $x_1x_2$ continues along edge~$y_3$ in route \;\!\raisebox{0.01in}{\squared{5}}\;\!, so the top left vertex is also adjacent to the top right vertex. 
\end{example}

\subsection{Groves and integer flows}
\label{sec:GrovesToFlows}

There is also a bijection between groves $\Gamma\in \calG_G^F$ and an integer flow $\phi\in\calF_G^\bbZ(\bfd)$.

Let $\Gamma=(\gamma_2,\hdots,\gamma_{n-1})\in \calG_G^F$ be a grove. We construct an integer flow $\phi=\phi(\Gamma)\in\calF_G^\bbZ(\bfd)$ by determining the flow $\phi(e)$ on every edge of $G$. First, define $\phi(e)=0$ for every outgoing edge from vertex $1$. Every other edge $e$ is an outgoing edge from an inner vertex $i$. Find the right vertex $r$ in $\gamma_i$ that corresponds to edge $e$. Define $\phi(e)$ to be one less than the number of neighbors of $r$. See Figure~\ref{fig:groveflow}. The following is Lemma~7.9 of \cite{mms19}.

\begin{figure}
    \centering
    \input{figures/groveflowbijection.tikz}
\caption{For a graph $G$ with a framing $F$ in which $e_1\prec_{\textup{In}(i)}e_2\prec_{\textup{In}(i)}e_3$ and $e_1'\prec_{\textup{Out}(i)}e_2'\prec_{\textup{Out}(i)}e_3'\prec_{\textup{Out}(i)}e_4'$, the bijection between an integer flow $\phi$ and the grove $\Gamma$ is understood by the relationship between the flow in edges of $G$ incident with vertex~$i$ and the degrees of the vertices in tree $\gamma_i$. }
    \label{fig:groveflow}
\end{figure}   

\begin{theorem}[\protect{\cite[Lemma~7.9]{mms19}}]
\label{thm:bijectionGF}
   The above rule that takes $\Gamma$ to $\phi(\Gamma)$ is a bijection between $\calG_G^F$ and $\calF_G^\bbZ(\bfd)$.
\end{theorem}

\begin{example}
Returning to Figure~\ref{fig:zig4example}, we explore the correspondence between the grove and the integer flow for Simplex~{\triangled{$\mathsf{A}$}}.  We proceed by assigning flow to edges in one vertex at a time. First, assign flow $0$ to edges $x_0$, $y_0$, and $y_1$. For the edges leaving vertex 2, consult~$\gamma_2$. We see that $x_2$ has two neighbors and $y_2$ has one neighbor; therefore, $\phi(x_2)=1$ and $\phi(y_2)=0$. For the edges leaving vertex 3, consult $\gamma_3$. Notice that $x_3$ has three neighbors and $y_3$ has one neighbor; therefore $\phi(x_3)=2$ and $\phi(y_3)=0$. 
\end{example}

As a direct consequence of Theorems~\ref{thm:bijectionCG} and \ref{thm:bijectionGF}, given a maximal clique $C$, one can directly compute the corresponding flow $\phi=\phi(C)$.

\begin{corollary}
    Let $C$ be a maximal clique and let $\phi=\phi(C)$ be its corresponding integer flow. Then $\phi(e)$ is equal to one less than the number of distinct prefixes of $e$ that appear in a route in $C$.
\end{corollary}

\subsection{Tracking edges in groves}
\label{sec:tracking}

An important aspect of a grove is that it is not simply a sequence of trees; the edges in the trees are intricately linked. This can be seen through the edge labelings presented in \cite{mms19}; one can also see them spacially as we describe here.

Through the bijection in Theorem~\ref{thm:bijectionCG}, an edge $\eta=(l,r)$ in tree $\gamma_i$ corresponds to a pair: the left vertex $l$ corresponds to a prefix $P$ at $i$ in $G$ and the right vertex $r$ corresponds to an outgoing edge $e=(i,j)$ in $G$. The existence of $\eta$ implies there is a route $R$ with prefix~$P$ that continues along edge $e$ to vertex $j$. This, in turn, associates the edge $\eta$ to a vertex $l'$ in~$\gamma_j$, corresponding to the prefix $Pe$. (Except when $j=n$, in which case $R=Pe$.) 

This leads to a spacial way to associate the connections between edges and vertices across trees in a grove. All edges incident with the right vertex of $\gamma_i$ (corresponding to an edge $e=(i,j)$) move together and arrive in $\gamma_j$ as left vertices. (We like imagining them as moving in parallel along a ribbon.)  Each of these collections of left vertices remains in the same order as when they left $\gamma_i$; whereas, the multiple arriving collections of left vertices are assembled to match the incoming framing at vertex $j$. We call this the {\em ribbon embedding of a grove}; it is visualized in Figure~\ref{fig:grovedynamics}.

\smallskip
This means we may {\em track} an edge $\eta=(l,r)$ in tree $\gamma_i$ forward through higher indexed trees in $\Gamma$. Let $r$ be the right vertex corresponding to an edge $e=(i,j)$ in $G$. If $j<n$, there are two options for the left vertex $l'$ in tree $\gamma_j$ that corresponds to $Pe$. The first option is that $l'$ is adjacent to two or more vertices in $\gamma_j$. In this case, we will say that the tracking process is unsuccessful. The second option is that $l'$ is adjacent to exactly one vertex $r'$ in~$\gamma_j$. In this case, continue the tracking process by tracking this new edge $\eta'=(l',r')$ forward in~$\Gamma$. If $j=n$, the tracking process terminates and is successful. We will call an edge $\eta$ of $\gamma_i$ {\em trackable} if the tracking process is successful. This property will be important in Section~\ref{sec:proofs}. 

\begin{figure}
    \centering
    \input{figures/grovedynamics.tikz}
\caption{The ribbon embedding of a grove showing how edges in trees correspond to left vertices in future trees. The middle edge in $\gamma_2$ (dashed in blue) is tracked through the grove to circled left vertices and their incident edges. The corresponding flow in $\cZig_6$ is shown above.}
    \label{fig:grovedynamics}
\end{figure}   

\begin{example}
Figure~\ref{fig:grovedynamics} shows that the middle edge in $\gamma_2$ is trackable. The circled left vertices in $\gamma_3$ and $\gamma_5$ correspond to the tracked edges from $\gamma_2$ and $\gamma_3$, respectively.    
\end{example}

\begin{lemma}
\label{lem:trackable}
An edge of $\gamma_i$ corresponding to prefix $P$ and outgoing edge $e$ is trackable if and only if there is a unique route $R\in C$ with prefix $Pe$.   
\end{lemma}
\begin{proof}
The tracking process is unsuccessful if and only if in some tree $\gamma_k$, the arriving prefix is the prefix of two or more routes in $C$.
\end{proof}


\section{The structure of the dual graph}
\label{sec:proofs}

\subsection{Operations on Groves}
\label{sec:grove_operations}

Section~\ref{sec:groves} describes the relationship between maximal cliques, groves, and integer flows. We use the characterization of adjacency between cliques in Theorem~\ref{thm:bump} to determine a characterization of adjacency between groves in $\calG_n^{\textup{LRL}}$. We start by defining an operation that will recover Definition~\ref{def:flowmove} when translated to the world of integer flows.

\begin{definition}
\label{def:move}
    Define the \textit{elementary move $m_i^+:\calG_n^{\textup{LRL}}\row\calG_n^{\textup{LRL}}$} for $2\leq i\leq n-1$ as follows. 

    \smallskip
    Start with a grove $\Gamma=(\gamma_2,\hdots,\gamma_{n-1})$. 
    \begin{enumerate}
        \item Modify $\gamma_{i}$. From bottom to top, label the left vertices $l_1$ through $l_{p_i}$ and label the right vertices $r_1$ and $r_2$. Find the vertex $l_j$ that is adjacent to both $r_1$ and $r_2$. If $j=1$, $m_i^+$ does not apply to $\Gamma$. Otherwise, replace edge $l_jr_1$ by $l_{j-1}r_2$ to create the tree $\gamma_{i}'$. (Informally, we call this an ``S to Z move'' because we replace a \begin{tikzpicture}[scale=0.25]
\draw (0,0) -- (1,0) -- (0,1) -- (1,1);
\end{tikzpicture} by a \begin{tikzpicture}[scale=0.25]
\draw (1,0) -- (0,0) -- (1,1) -- (0,1);
\end{tikzpicture}.)
        \item If $i\leq n-2$, modify $\gamma_{i+1}$. If the top vertex on the left side $l_{p_{i+1}}$ is adjacent to both vertices on the right side, $m_i^+$ does not apply to $\Gamma$. Otherwise, delete $l_{p_{i+1}}$ (and its incident edge) to create the tree $\gamma_{i+1}'$.
        \item If $i\leq n-3$, modify $\gamma_{i+2}$. Add a new vertex $l_0$ at the bottom of the left side and the edge $l_0r_1$ to create the tree $\gamma_{i+2}'$.
        \item If $m_i^+$ applies to $\Gamma$, define $m_i^+(\Gamma)$ to be the grove $\Gamma'$ where $\gamma_{i}$ is replaced by $\gamma_{i}'$, $\gamma_{i+1}$ is replaced by $\gamma_{i+1}'$ and $\gamma_{i+2}$ is replaced by $\gamma_{i+2}'$ (as appropriate).
    \end{enumerate}
    This procedure is visualized in Figure~\ref{fig:elementarymove}. 
    The {\em elementary move $m_i^-$} is the inverse operation that takes $\Gamma'$ to $\Gamma$. We say that $\Gamma$ and $\Gamma'$ are {\em adjacent} if they are related by an elementary move.
\end{definition}

\begin{theorem} 
\label{thm:GrovesAdjacent}
If two simplices of $\calT_n^{\textup{LRL}}$ are adjacent, then their corresponding groves are adjacent, that is, they differ by one of the elementary moves.
\end{theorem}

\begin{proof}
Let $\Delta_C$ and $\Delta_{C'}$ be two adjacent simplices that correspond to maximal cliques $C$ and $C'$. By Theorem~\ref{thm:bump}, there are two routes $R\in C$ and $R'\in C'$ that satisfy $R=PiQ$ and $R'=P'iQ'$ for prefixes $P'\prec P$ and suffixes $Q\prec Q'$ adjacent in their total orderings at $i$. This vertex $i$ may not be unique because it may be a vertex along a maximal path including~$i$ that $R$ and $R'$ share. Choose $i$ to be the last vertex on this path. In this way, $Q$ and $Q'$ start with the edges $x_i$ and $y_i$, respectively.

The groves $\Gamma=(\gamma_2,\hdots,\gamma_{n-1})$ and $\Gamma'=(\gamma_2',\hdots,\gamma_{n-1}')$ are related in a predictable way, as described here and visualized in Figure~\ref{fig:elementarymove}.

\begin{figure}
    \centering
    \input{figures/elementarymove.tikz}
\caption{A visualization of the difference between $\Gamma$ and $\Gamma'=m_i^+(\Gamma)$.}
    \label{fig:elementarymove}
\end{figure}  

Because groves are constructed from left to right, $\gamma_j=\gamma_j'$ for all $j<i$. The subtrees of $\gamma_i$ and $\gamma_i'$ corresponding to $P$, $P'$, $x_i$, and $y_i$ are as shown in Figure~\ref{fig:szmove}; the remainder of each tree is unchanged. 

\begin{figure}[ht]
  \begin{tikzpicture}[scale=1]
    \draw[gray,  thin] (0.6,-0.15) -- (1,0);
    \draw[gray,  thin] (0.65,-0.225) -- (1,0);
    \draw[gray,  thin] (0.7,-0.3) -- (1,0);
    \draw[gray,  thin] (0.6,1.15) -- (1,1);
    \draw[gray,  thin] (0.65,1.225) -- (1,1);
    \draw[gray,  thin] (0.7,1.3) -- (1,1);
    \draw[very thick] (0,0) -- (1,0);
    \draw[blue,very thick,dash pattern=on 3pt off 1.5pt] (1,0) -- (0,1);
    \draw[very thick] (0,1) -- (1,1);
    \node[] at (-.3,1) {$P_{\phantom{,}\!}$};
    \node[] at (-.3,0) {$P'_{\phantom{,}\!}$};
    \node[] at (1.3,1) {$y_i$};
    \node[] at (1.3,0) {$x_i$};
    \node[] at (-.3,-0.4) {\scriptsize $\vdots$};
    \node[] at (-.3,1.7) {\scriptsize $\vdots$};
    \node[] at (0.5,-1.25) {\scriptsize Tree $\gamma_{i}$ of $\Gamma$};
  \end{tikzpicture} 
  \qquad\qquad%
  \begin{tikzpicture}[scale=1]
    \draw[gray,  thin] (0.6,-0.15) -- (1,0);
    \draw[gray,  thin] (0.65,-0.225) -- (1,0);
    \draw[gray,  thin] (0.7,-0.3) -- (1,0);
    \draw[gray,  thin] (0.6,1.15) -- (1,1);
    \draw[gray,  thin] (0.65,1.225) -- (1,1);
    \draw[gray,  thin] (0.7,1.3) -- (1,1);
    \draw[very thick] (0,0) -- (1,0);
    \draw[red,very thick,dash pattern=on 3pt off 1.5pt] (0,0) -- (1,1);
    \draw[very thick] (0,1) -- (1,1);
    \node[] at (-.3,1) {$P_{\phantom{,}\!}$};
    \node[] at (-.3,0) {$P'_{\phantom{,}\!}$};
    \node[] at (1.3,1) {$y_i$};
    \node[] at (1.3,0) {$x_i$};
    \node[] at (-.3,-0.4) {\scriptsize $\vdots$};
    \node[] at (-.3,1.7) {\scriptsize $\vdots$};
    \node[] at (0.5,-1.25) {\scriptsize Tree $\gamma_{i}'$ of $\Gamma'$};
  \end{tikzpicture} 
\caption{Subtrees in two adjacent groves}
\label{fig:szmove}
\end{figure}

Up to two additional pairs of trees in the groves are different, depending on the value of $i$. 

If $i\leq n-2$, $\gamma_{i+1}$ and $\gamma_{i+1}'$ are different. In tree $\gamma_{i+1}$ of $\Gamma$, there is a vertex corresponding to the prefix $Px_i$ and this prefix is the highest in the total ordering of prefixes at vertex $i+1$. (This follows because all prefixes ending in $x_i$ maintain the same order as in $\gamma_i$ and are higher than all prefixes ending in $y_i$.) Furthermore, this prefix can't be adjacent to both $x_{i+1}$ and $y_{i+1}$ in $\gamma_{i+1}$ or else its removal would delete two routes from $C$. In tree $\gamma_{i+1}'$ of $\Gamma'$, this prefix does not exist. 

If $i\leq n-3$, $\gamma_{i+2}$ and $\gamma_{i+2}'$ are different. In tree $\gamma_{i+2}'$ of $\Gamma'$, there is a vertex corresponding to the prefix $P'y_i$ and this prefix is the lowest in the total ordering of prefixes at vertex $i+2$. Furthermore, this prefix can't be adjacent to both right vertices in $\gamma_{i+2}'$ or else its removal would delete two routes from $C'$. In tree $\gamma_{i+2}$ of $\Gamma$, this prefix does not exist. 

Because of the structure imposed by the graph $\cZig_n$ and the LRL framing, a noteworthy consequence is that 
\begin{equation}
\label{eq:gamma_i+3}
\textup{$\gamma_j=\gamma_j'$ for all $j\geq i+3$.}
\end{equation}
This is because in trees $\gamma_{i+3}$ and $\gamma_{i+3}'$ (if they exist), the position of the corresponding prefixes $Px_iy_{i+1}$ in $\gamma_{i+3}$  and $P'y_ix_{i+2}$ in $\gamma_{i+3}'$ coincide in the total order of prefixes at vertex $i+3$. The prefix $Px_iy_{i+1}$ is the lowest prefix in the prefixes that end in $y_{i+1}$ and the prefix $P'y_ix_{i+2}$ is the highest prefix in the prefixes that end in $x_{i+2}$. No other prefixes differ in this tree.

The overlapping of these prefixes in tree $i+3$ means that all future trees are the same, even if the label that corresponds to the continuation of the modified prefix is different.  In other words, if $i\geq n+4$, the routes being exchanged between $C$ and $C'$ are necessarily $R=Px_iy_{i+1}S$ and $R'=P'y_ix_{i+2}S$ for the same suffix $S$ starting at vertex $i+3$.

An important observation about this configuration is that at no point can the prefix of the route $R$ arriving at a vertex $j\geq i+3$ be adjacent to both right vertices because then $C$ would differ from $C'$ by more than one route.
\end{proof}

Theorem~\ref{thm:ElementaryMoves} now follows as a direct corollary.

\thmElementaryMoves*

\begin{proof}
Theorem~\ref{thm:bijectionGF} proves that the only edges whose flows change when $m_i^+$ or $m_i^-$ is applied are the edges $x_i$, $y_{i+1}$, $y_i$, and $x_{i+2}$. 

The conditions for $m_i^+$ and $m_i^-$ for integer flows are a direct translation of the conditions for $m_i^+$ and $m_i^-$ for groves.
\end{proof}

The parallels between the elementary moves in groves and in integer flows are visualized in Figure~\ref{fig:mplusvisualized}. Note that the relative position of the circled vertex in $\gamma_5$ is the same as the relative position of the circled vertex in $\gamma_5'$, as in Equation~\eqref{eq:gamma_i+3}.

\begin{figure}
    \centering
    \input{figures/grovedynamics2.tikz}
\caption{The result of applying $m_2^+$ to the element in Figure~\ref{fig:grovedynamics}. The edge flows that have changed are highlighted and dashed in red.}
    \label{fig:mplusvisualized}
\end{figure}   

\smallskip
Let $\Gamma_0$ be the grove in $\calG_n^F$ that corresponds to the integer flow $\phi_0\in\flows$ in which all the flow is in slack variables. In other words, 
\[\phi_0=(0,1,2,\hdots,n-2,0,\hdots,0).\]

\begin{lemma}
\label{lem:seqElemMoves}
Let $\Gamma$ be a grove in $\calG_n^F$. Then there is a sequence of elementary moves $(m^{\epsilon_1}_{j_1},\hdots,m^{\epsilon_l}_{j_l})$ with $\epsilon_k\in\{+,-\}$ for $1\leq k\leq l$ such that $\Gamma=m^{\epsilon_l}_{j_l}\circ\cdots\circ m^{\epsilon_1}_{j_1}(\Gamma_0)$. 
\end{lemma}
\begin{proof}
Let $\Delta_0$ and $\Delta$ be the simplices in the triangulation $\calT$ corresponding to $\Gamma_0$ and $\Gamma$, respectively. The dual graph to the DKK triangulation $D(\calT)$ is connected, so there is a path in $D(\calT)$ from $\Delta_0$ to $\Delta$. The sequence of simplices along that path are adjacent, so their corresponding groves are related by an elementary move, and the result follows.
\end{proof}

This sequence of elementary moves does not have to be unique; however, it does leave a unique signature that will allow us to prove Theorem~\ref{thm:embedding}. 

For $2\leq k\leq n-1$, denote by $\bfe_k$ the basis vector with a $1$ in position $k$ and $0$'s elsewhere. 

\begin{definition}
\label{def:coordinates}
Let $\Delta$ be a simplex of $\calT_n^{LRL}$,  corresponding grove $\Gamma\in\calG_n^F$ and associated sequence of elementary moves $(m^{\epsilon_1}_{j_1},\hdots,m^{\epsilon_l}_{j_l})$. Define the {\em coordinate vector} of $\Delta$ to~be
\[\bfx(\Delta)=\sum_{k=1}^l \epsilon_k\bfe_{j_k}.\]
\end{definition}

\begin{lemma}
The coordinate vectors from Definition~\ref{def:coordinates} are well defined.
\end{lemma}
\begin{proof}
By Definition~\ref{def:flowmove}, applying the move $m_{k}^{\epsilon_k}$ to a flow results in a new flow that differs by $\epsilon_k\cdot(0,\hdots,-1,0,1,\hdots,0,0,\hdots,1,-1,\hdots,0)$. If two sequences of elementary moves $(m^{\epsilon_1}_{j_1},\hdots,m^{\epsilon_l}_{j_l})$ yield the same lattice point, 
then the resulting integer flows (and therefore simplices) also coincide.
\end{proof}

We are now in a position to prove Theorem~\ref{thm:embedding}.

\thmEmbedding*

\begin{proof}
We construct an embedding of $D(\calT_n^{\textup{LRL}})$ in $\Square_{n-2}$. Embed the node for each simplex~$\Delta$ at its coordinate vector $\bfx(\Delta)\in\bbZ^{n-2}$. By Theorem~\ref{thm:GrovesAdjacent}, the edges of $D$ only connect two integer lattice points that differ by $\bfe_k$ for some $2\leq k\leq n-1$, which is precisely the condition that they are edges of the square grid graph. 
\end{proof}

\subsection{Adjacency of Groves}
\label{sec:grove_adjacency}

The embedding of $D(\calT_n^{\textup{LRL}})$ in Theorem~\ref{thm:embedding} shows that the vertices of $D(\calT_n^{\textup{LRL}})$ lie in a predictable lattice structure. However, this does not determine which neighboring lattice points of $\Square_{n-2}$ are adjacent in $D(\calT_n^{\textup{LRL}})$. We now provide a characterization of this adjacency.

\begin{definition}
\label{def:groveoffsets}
    Given a grove $\Gamma=(\gamma_2,\hdots,\gamma_{n-1})$ and an index $2\leq i\leq n-1$, define the {\em sequence of offsets} of $\Gamma$ at $i$, $\bfz^{(i)}(\Gamma)=\big(z_{i+1},\hdots,z_{n-1}\big)$ by using the tracking method from Section~\ref{sec:tracking} as follows.
    
    If $i=n-1$, define $\bfz^{(i)}(\Gamma)$ to be the empty sequence. Otherwise, track the topmost edge connected to $r_1$ in~$\gamma_i$. This process arrives first at tree $\gamma_{i+1}$.   When the tracking process arrives at tree $\gamma_j$, consider the left vertex $l$ corresponding to the arriving prefix $P$:
    
    \begin{itemize}
        \item[--] If $l$ is adjacent to {\bf both} right vertices, record $z_k=0$ for all $k\geq j$ and STOP. 
        \item[--]  If $l$ is {\bf only} adjacent to right vertex $r_1$, record $z_j$ to be $-1$ times the number of edges above $lr_1$ that are adjacent to $r_1$. If $j=n-1$, STOP. Otherwise, continue the tracking process in tree $\gamma_{j+1}$ with the left vertex corresponding to prefix $Px_j$.
        \item[--] If $l$ is {\bf only} adjacent to right vertex $r_2$, record $z_j$ to be the number of edges below $lr_2$ that are adjacent to $r_2$. If $j\leq n-2$, record $z_{j+1}$ to be $\emptyset$. If $j\geq n-2$, STOP. Otherwise, continue the tracking process in tree $\gamma_{j+2}$ with the left vertex corresponding to prefix $Py_j$.
    \end{itemize}

\end{definition}

This sequence of offsets provides numerical information about the tracking process. A $z_j=\emptyset$ corresponds to the tracked edge bypassing tree $\gamma_j$. The first $j$ for which $z_j=0$ corresponds to a left vertex corresponding to the incoming prefix in tree $\gamma_j$ being adjacent to both right vertices. A non-zero integer value for $z_j$ corresponds to the (signed) distance in tree $\gamma_j$ between the left vertex corresponding to the incoming prefix and the left vertex that is connected to both right vertices.

The following lemma follows directly from the definition of trackable.

\begin{lemma}
\label{lem:nozeroes}
    The topmost edge incident to $r_1$ in $\gamma_i$ is trackable if and only if $\bfz^{(i)}(\Gamma)$ contains no zeroes.
\end{lemma}

The definition of the sequence of offsets of a flow $\phi$ is the translation of the definition of the sequence of offsets of its corresponding grove $\Gamma$.

\begin{theorem}
\label{thm:OffsetsEqual}
    Let $\Gamma\in \calG_n^{\textup{LRL}}$ and let $\phi\in\calF_n^\bbZ$ be its corresponding integer flow. For all $2\leq i\leq n-1$, The sequence of offsets $\bfz^{(i)}(\Gamma)$ in Definition~\ref{def:groveoffsets} equals the sequence of offsets $\bfz^{(i)}(\phi)$ in Definition~\ref{def:flowoffsets}.
\end{theorem}
\begin{proof}
If $i=n-1$, the sequence of offsets is the empty sequence in both definitions. When $i<n-1$, we first show that the values of $z_{i+1}$ agree and that the rest of the values agree by induction. 

We find $z_{i+1}$ from $\Gamma$ by tracking the topmost edge $lr_1$ connected to $r_1$ in $\gamma_i$. This edge arrives in $\gamma_{i+1}$ as the topmost left vertex~$l'$, which must be adjacent to $r_2'$ in $\gamma_{i+1}$. Definition~\ref{def:groveoffsets} gives $z_{i+1}$ as the number of edges in $\gamma_{i+1}$ below $l'r_2'$ that are adjacent to $r_2'$. By Theorem~\ref{thm:bijectionGF}, this equals $\phi(y_{i+1})$, which is exactly the value of $z_{i+1}$ from $\phi$ given in Definition~\ref{def:flowoffsets}. 

Now suppose that $\bfz^{(i)}(\phi)$ and $\bfz^{(i)}(\Gamma)$ agree up to the $j$-th entry and $z_j\ne 0$. We will show that the next entry (or entries) also agree. 

Suppose $z_j<0$. This means that in tree $\gamma_j$, the tracked edge is $lr_1$, which is the $(-z_j+1)$-st edge from the top of the edges adjacent to $r_1$. This, in turn, ensures that this edge becomes the $(-z_j+1)$-st left vertex $l'$ from the top in tree $\gamma_{j+1}$. There are $\phi(y_{j+1})+1$ left vertices adjacent to $r_2'$ in $\gamma_{j+2}$. So the signed distance between $l'$ and the vertex adjacent to both $r_1'$ and $r_2'$ is $z_j+\phi(y_{j+1})$ and the values of $\bfz^{(i)}(\phi)$ and $\bfz^{(i)}(\Gamma)$ agree up to the $(j+1)$-st entry.

Suppose $z_j>0$. In tree $\gamma_j$, the tracked edge is $lr_2$, which is the $(z_j+1)$-st edge from the bottom of the edges adjacent to $r_2$. This, in turn, ensures that this edge becomes the $(z_j+1)$-st left vertex $l'$ from the bottom in tree $\gamma_{j+2}$. There are $\phi(x_{j+2})+1$ left vertices adjacent to $r_1'$ in $\gamma_{j+2}$. So the signed distance between $l'$ and the vertex adjacent to both $r_1'$ and $r_2'$ is $z_j-\phi(x_{j+2})$. In both $\bfz^{(i)}(\phi)$ and $\bfz^{(i)}(\Gamma)$, $z_{j+1}=\emptyset$ (if it exists), and we also have agreement in $z_{j+2}$, so the values agree up to the $(j+2)$-st entry. 

If at any point $z_j=0$, then $\bfz^{(i)}(\Gamma)$ and $\bfz^{(i)}(\phi)$ agree for all subsequent entries. This completes the inductive step.
\end{proof}

\begin{theorem}
\label{thm:GroveAdjacency}
Let $\Gamma$ and $\Gamma'$ be two groves that satisfy $\Gamma'=m_i^+(\Gamma)$ for some~$i$. The two simplices in $\calT_n^{\textup{LRL}}$ corresponding to $\Gamma$ and $\Gamma'$ are adjacent in $D(\calT_n^{\textup{LRL}})$ if and only if the sequence of offsets of $\Gamma$ at $i$ contains no zeroes.
\end{theorem}
\begin{proof}
Suppose $\Gamma$ and $\Gamma'$ satisfy $\Gamma'=m_i^+(\Gamma)$ for some~$i$. Let $C$ and $C'$ be the maximal cliques that correspond to $\Gamma$ and $\Gamma'$, respectively. 

By Lemma~\ref{lem:nozeroes}, the sequence of offsets of $\Gamma$ at $i$ contains zeroes if and only if the topmost edge $e$ incident to $r_1$ in $\gamma_i$ (corresponding to a prefix $P$) is not trackable. By Lemma~\ref{lem:trackable}, this is true if and only if there are multiple distinct routes $R_1$ and $R_2$ that have prefix~$Pe$. Performing the elementary move $m_i^+$ removes $e$ from $\Gamma$; as a consequence, the previous statement is equivalent to $C$ and $C'$ differing by at least two routes, or, in other words, their corresponding simplices are not adjacent. 
\end{proof}

Theorem~\ref{thm:AdjacencyCriterion} now follows directly as a corollary of Theorems~\ref{thm:OffsetsEqual} and \ref{thm:GroveAdjacency} to provide a complete combinatorial characterization of $D(\calT_n^{\textup{LRL}})$.

\thmAdjacencyCriterion*


\section{Statistics old and new}
\label{sec:sec4}

\subsection{The $h^*$-vector}
\label{sec:hstar}

A popular refinement of the normalized volume of a flow polytope is its $h^*$-vector; our constructions provide two new conjectured statistics that both recover this vector. We now provide some context and refer the reader to \cite[Chap.~3, 10]{beckrobins} for more background information.

For an integral convex polytope $\calP$ of dimension $d$ in $\bbR^n$, the Ehrhart series $L_\calP(t)$ of $P$ is the generating function that encodes the number of integer lattice points in dilates $t\calP$ of $\calP$. It is well known that if $\calP$ is integral, there exists a polynomial $h^*_\calP(z)$, called the {\em $h^*$-polynomial} of $\calP$, satisfying 
\[1+\sum_{t\geq 1}L_\calP(t)z^t=\frac{h^*_\calP(z)}{(1-z)^{d+1}}.\]
The vector of coefficients of $h^*_\calP(z)$ is called the {\em $h^*$-vector} of $\calP$. Stanley \cite{stanley1980decompositions} proved that the coefficients are non-negative integers and the sum of the entries is $\vol(\calP)$.

When a polytope $\calP$ admits a regular unimodular triangulation $\calT$, we can explore $\calT$ as a simplicial complex. The simplicial complex has an associated $h$-polynomial which, in this special case, is equal to the $h^*$-polynomial of $\calP$.

\begin{theorem}[\cite{beckrobins}]
If $\calP$ is an integral polytope that admits a unimodular
triangulation~$\calT$, then the $h^*$-polynomial of $\calP$ is given by the $h$-polynomial of the triangulation $\calT$.
\end{theorem}

\begin{definition}
A {\em shelling} of a triangulation $\calT$ is an ordering $\Delta_1$, $\Delta_2$, $\hdots$ of the simplices of~$\calT$ such that for all $i\geq 2$, $\Delta_i\cap\big(\Delta_1\cup\cdots\cup \Delta_{i-1}\big)$ is a union of codimension~1 faces of $\Delta_i$; we call $s_i$ the number of these faces. A triangulation $\calT$ is {\em shellable} if a shelling of $\calT$ exists.
\end{definition}

Stanley \cite[Prop.~3.6]{stanley1979balanced} proved that if a simplicial complex is shellable, its $h$-vector can be calculated combinatorially. The coefficient $h_j$ is the number of simplices $\Delta_i$ such that $s_i=j$. 

\medskip
We conjecture two shellings of $\calT_n^{\textup{LRL}}$. Define the {\em modified lexicographic ordering} on simplices of $\calT_n^{\textup{LRL}}$ by associating to every simplex $\Delta$ its coordinate vector $\bfx(\Delta)$. First order these vectors by the sum of their coordinates and then order the vectors that have the same sum lexicographically. This induces an orientation on the edges of $D(\calT_n^{\textup{LRL}})$ to form a directed graph $\vec{D}(\calT_n^{\textup{LRL}})$.

\begin{conjecture}
\label{conj:shellability}
The DKK-triangulation $\calT_n^{\textup{LRL}}$ is shellable because both the modified lexicographic ordering and its reverse are shellings.
\end{conjecture}

Based on this construction, we can define two statistics\footnote[2]{The names $\mathsf{sz}$ and $\mathsf{zs}$ come from the concepts of S-to-Z and Z-to-S moves in a grove. See Definition~\ref{def:move}.} $\mathsf{sz}$ and $\mathsf{zs}$ on combinatorial objects that correspond to the vertices of $D(\calT)$: simplices $\Delta\in \calT_n^{\textup{LRL}}$, integer flows $\phi\in \calF_n^\bbZ$, groves $\Gamma\in \calG_n^{\textup{LRL}}$, and maximal cliques~$C\in \calC_n^{\textup{LRL}}$. 

\begin{definition}
    The statistic $\mathsf{sz}(\Delta)$ is the number of edges of $\vec{D}(\calT_n^{\textup{LRL}})$ oriented away from~$\Delta$; the statistic $\mathsf{zs}(\Delta)$ is the number of edges of $\vec{D}(\calT_n^{\textup{LRL}})$ oriented toward $\Delta$.
\end{definition}

In other words, $\mathsf{sz}$ counts combinatorial objects that are adjacent along a positive elementary move and $\mathsf{zs}$ counts combinatorial objects that are adjacent along a negative elementary move. Assuming Conjecture~\ref{conj:shellability}, these two statistics recover the $h^*$-polynomial of $\mathcal{F}_{\cZig_{n}}(\bfu)$. The following conjecture has been verified computationally for $n\leq 8$.

\begin{conjecture}
\label{conj:szzs}
    The $h^*$-polynomial of $\mathcal{F}_{\cZig_{n}}(\bfu)$ is
    \(\displaystyle \sum_{\phi\in\calF_n^\bbZ} z^{\mathsf{sz}(\phi)}=\sum_{\phi\in \calF_n^\bbZ} z^{\mathsf{zs}(\phi)}.\)
\end{conjecture}

\begin{example}
In the example from Section~\ref{sec:example}, the modified lexicographic ordering of $\calT_4^{\textup{LRL}}$ is 
\[\big(\raisebox{-.02in}{\,\triangled{$\mathsf{A}$}}, \raisebox{-.02in}{\,\triangled{$\mathsf{B}$}}, \raisebox{-.02in}{\,\triangled{$\mathsf{C}$}}, \raisebox{-.02in}{\,\triangled{$\mathsf{D}$}}, \raisebox{-.02in}{\,\triangled{$\mathsf{E}$}}\big).\] 
From Figure~\ref{fig:zig4projection} we compute 
\[\textup{$\sz\big(\raisebox{-.02in}{\,\triangled{$\mathsf{A}$}}\big)=1$,
$\sz\big(\raisebox{-.02in}{\,\triangled{$\mathsf{B}$}}\big)=2$,
$\sz\big(\raisebox{-.02in}{\,\triangled{$\mathsf{C}$}}\big)=1$,
$\sz\big(\raisebox{-.02in}{\,\triangled{$\mathsf{D}$}}\big)=1$, and
$\sz\big(\raisebox{-.02in}{\,\triangled{$\mathsf{E}$}}\big)=0$,}\] 
as well as
\[\textup{$\zs\big(\raisebox{-.02in}{\,\triangled{$\mathsf{A}$}}\big)=0$,
$\zs\big(\raisebox{-.02in}{\,\triangled{$\mathsf{B}$}}\big)=1$,
$\zs\big(\raisebox{-.02in}{\,\triangled{$\mathsf{C}$}}\big)=1$,
$\zs\big(\raisebox{-.02in}{\,\triangled{$\mathsf{D}$}}\big)=1$, and
$\zs\big(\raisebox{-.02in}{\,\triangled{$\mathsf{E}$}}\big)=2$.}\]
We see that indeed,  \[\sum_{\phi\in\calF_4^\bbZ} t^{\mathsf{sz}(\phi)}=\sum_{\phi\in\calF_4^\bbZ} t^{\mathsf{zs}(\phi)}=1+3z+z^2=h^*\big(\mathcal{F}_{\cZig_{4}}(\bfu)\big).\]
\end{example}

The $h^*$-polynomials of $\mathcal{F}_{\cZig_{n}}(\bfu)$ may be referred to as the zig-zag Eulerian polynomials (see \cite{petersen2024zigzag}) and have been computed using other statistics. In \cite{coons19}, Coons and Sullivant computed $h^*$ using a {\em swap statistic} $\mathsf{swap}$ on alternating permutations. Separately, Ayyer, Josuat-Verg\`es, and Ramassamy  \cite{ajr20} developed a {\em descent statistic} $\mathsf{des}$ on cyclic permutations to compute $h^*$ for a polytope integrally equivalent to $\calF_n^\bbZ$, which was further developed by Gonz\'alez D'Le\'on, Hanusa, Morales, and Yip \cite{dleon23}. Table~\ref{table:statistics} calculates the $\sz$ and~$\zs$ statistics for all $\phi\in\calF_5^\bbZ$.
We see that $\sz$ and $\zs$ are equidistributed with $\mathsf{swap}$ and $\mathsf{des}$ but match neither of them. Indeed, 
\[\sum_{\phi\in\calF_5^\bbZ} t^{\mathsf{sz}(\phi)}=\sum_{\phi\in\calF_5^\bbZ} t^{\mathsf{zs}(\phi)}=(1+7t+7t^2+t^3),\] 
which agrees with $h^*\big(\mathcal{F}_{\cZig_{5}}(\bfu)\big)$ as we expect from Conjecture~\ref{conj:szzs}.

\subsection{Statistic Computations}
We conclude this section by compiling a guide to walk through the bijections between these combinatorial objects and the computations of their associated statistics. The shaded row of Table~\ref{table:statistics}, corresponding to the alternating permutation $\alpha=45231$, will serve as a running example.

\begin{table}
    \centering
    \input{figures/statistics.tex}\medskip
    \caption{The sixteen alternating permutations, alternating permutation inverses, integer flows, and circular permutations for $\cZig_5$ and their associated swap, S-to-Z, Z-to-S, and descent statistics. The shaded row is discussed in the text.}
    \label{table:statistics}
\end{table}

Starting with an alternating $n$-permutation $\alpha=\alpha_1\alpha_2\cdots\alpha_n$, first compute its inverse permutation $\beta=\beta_1\beta_2\cdots\beta_n$. The $\mathsf{swap}$ statistic (see \cite[Definition~1.2]{coons19}) is computed by counting the number of indices $i$ such that $\beta_i<\beta_{i+1}-1$.  

The inverse permutation of $\alpha=45231$ is $\beta=53412$. Since $5\geq 2$, $3\geq 3$, $4\geq 0$, and~$1\geq 1$, we compute that $\swap(\alpha)=0$.

We now determine the integer flow $\phi$ that corresponds to $\alpha$ by following \cite[Figure~11]{mms19}. We start with the inverse permutation $\beta$ and construct a maximal clique $C$ on the planar framing of $\cZig_{n}$ as shown in Figure~\ref{fig:planarclique}. Label the regions of the embedding with numbers from $1$ to $n$. Start constructing $C$ with the path $P_0$ that follows the lowest edges of $\cZig_{n}$. For $1\leq i\leq n$, we will augment $C$ to include a path $P_i$ determined from the path $P_{i-1}$ by replacing the edges below the region labeled $\beta_i$ by the edges above the region labeled $\beta_i$. The set $C=\{P_0,P_1,\hdots,P_n\}$ is then a maximal clique in $\cZig_n$.

\begin{figure}
    \centering
    \input{figures/maximalclique.tikz}
    \caption{The process to recover a clique from the planar embedding of $\cZig_5$. The enclosed regions are labeled from $1$ to $5$.}
\label{fig:planarclique}
\end{figure}

In our example, the maximal clique $C$ is therefore: \[\textup{$P_0=y_0y_2y_4$, $P_1=y_0y_2x_4$, $P_2=y_0x_2x_3x_4$, $P_3=y_0x_2y_3$, $P_4=x_1x_2y_3$, and $P_5=y_1y_3$}.\] This maximal clique is then converted to the flow $\phi$ shown here by Theorem~\ref{thm:bijectionCG}.

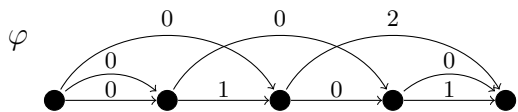
\begin{figure}[h!]
\vspace{-.15in}
$\phi$\,%
\raisebox{-.5in}{
\begin{tikzpicture}[scale=0.75]
	\filldraw[black] (0,0) circle (3pt) node[label=below:{\tiny\color{orange} $1$}](a1) {};
	\filldraw[black] (2,0) circle (3pt) node[label=below:{\tiny\color{orange} $2$}](a2) {};
	\filldraw[black] (4,0) circle (3pt) node[label=below:{\tiny\color{orange} $3$}](a3) {};
	\filldraw[black] (6,0) circle (3pt) node[label=below:{\tiny\color{orange} $4$}](a4) {};
    \filldraw[black] (8,0) circle (3pt) node[label=below:{\tiny\color{orange} $5$}](a5) {};

    \draw[->] (a1)--(a2);		
	\draw[->] (a2)--(a3);		
	\draw[->] (a3)--(a4);		
	\draw[->] (a4)--(a5);		

\draw[->, shorten <= 1.5mm, shorten >= 1.5mm] (0,0) .. controls (1,1.1) and (1.4,-0.6) .. (2,0);
\draw[->, shorten <= 1.5mm, shorten >= 1.5mm] (0,0) .. controls (1.5,2.5) and (3.4,-0.9) .. (4,0);
\draw[->, shorten <= 1.5mm, shorten >= 1.5mm] (2,0) .. controls (3.5,2.5) and (5.4,-0.9) .. (6,0);
\draw[->, shorten <= 1.5mm, shorten >= 1.5mm] (4,0) .. controls (5.5,2.5) and (7.4,-0.9) .. (8,0);
\draw[->, shorten <= 1.5mm, shorten >= 1.5mm] (6,0) .. controls (7,1.1) and (7.4,-0.4) .. (8,0);

    \node[] at (0.7,-0.25) {\tiny\color{blue} $0$};
    \node[] at (2.6,-0.25) {\tiny\color{blue} $1$};
    \node[] at (4.6,-0.25) {\tiny\color{blue} $0$};
    \node[] at (6.6,-0.25) {\tiny\color{blue} $1$};
    \node[] at (1.0,0.6) {\tiny\color{blue} $0$};
    \node[] at (1.4,1.25) {\tiny\color{blue} $0$};
    \node[] at (3.4,1.25) {\tiny\color{blue} $0$};
    \node[] at (5.4,1.25) {\tiny\color{blue} $2$};
    \node[] at (7.1,0.65) {\tiny\color{blue} $0$};
  
\end{tikzpicture}}
\end{figure}
We notice that the only elementary moves that apply to $\phi$ to give adjacent integer flows are $m_2^+$, $m_3^-$, and $m_4^+$. We conclude that $\sz(\phi)=2$ and $\zs(\phi)=1$.

We now determine the cyclic permutation $\pi$ that corresponds to the integer flow $\phi$ by following \cite[Definition~4.2]{dleon23}. We first must reindex the labels on the nonslack edges of $\cZig_n$ to run from $y_1$ through $y_n$. Then we can compute that $\mathsf{ACT}(1)=\{0\}$,  $\mathsf{INACT}(1)=\{1\}$, and $\mathsf{ACT}(j)=\{0,\hdots,j-2\}$ and $\mathsf{INACT}(j)=\{j-1,j\}$ and for $j\geq 2$. 

These conditions imply that the construction of the cyclic permutation will be extremely well behaved. The numbers 0, 1, and 2 are first placed in a circle, then the numbers $3$ through $n$ are placed in the cyclic permutation as follows. Place number $j$ in the position that is $\phi(y_j)$ positions after number $(j-1)$. 

For the integer flow $\phi$ in our running example (with re-indexed nonslack edges), $\phi(y_3)=0$, $\phi(y_4)=2$, and $\phi(y_5)=0$. Starting with $012$, the 3 is placed directly after the 2 (giving $0123$), the 4 is placed two positions after the 3 (skipping the 0 and the 1 to give $01423$), and the 5 is placed in the position right after the 4 (giving $014523$). Ignoring the leading 0, $\phi$ corresponds to the circular permutation $\pi=14523$. 

The descent statistic $\mathsf{des}$ is read by seeing how many indices $1\leq i\leq n-1$ satisfy $\pi_i>\pi_{i+1}$. In this case the only descent is $\pi_3=5>2=\pi_4$, so $\mathsf{des}(\pi)=1$.


\section{Open Problems}
\label{sec:open}

Subsequent to the preprint of this article, Open Problems~\ref{open:1} through \ref{open:4} below have been addressed by joint work of the second author along with Gonz\'alez D'Le\'on and Yip \cite{permuflows}, by further developing the concept of a grove for all graphs $G$ and framings $F$ with a netflow vector having positive entries except in the final vertex.

\smallskip
As we saw in Sections~\ref{sec:groves} and \ref{sec:proofs}, groves can be a powerful tool for proving statements about integer flows. Furthermore, Theorem~\ref{thm:bijectionGF} provides a bijection between $\calG_G^F$ and $\calF_G^\bbZ(\bfd)$ for all graphs $G$ and for all framings~$F$. 

The length-reverse-length framing of the contracted zigzag graph has the precise structure to make elementary moves on groves behave extremely well (as shown in Figure~\ref{fig:elementarymove}). This, in turn, led to a simple numerical condition to check for adjacency of the corresponding integer flows. It is natural to explore the possibilities for other framings.

\begin{open}
\label{open:1}
When the contracted zigzag graph is adorned with a different framing, are there easy-to-describe elementary moves for integer flows analogous to those in Definition~\ref{def:flowmove}? If so, are there explicit numerical conditions that guarantee adjacency of the corresponding integer flows? 
\end{open}

Of course, we need not restrict ourselves to the contracted zigzag graph.

\begin{open}
Can we determine the concept of an elementary move for a general graph $G$ with a framing $F$? (Or for particular framings of $G$?) If so, how does this allow us to determine the structure of $D(\calT_G^F)$? 
\end{open}

Building on Section~\ref{sec:hstar} with respect to the $h^*$-polynomial for $\calF_G(\bfu)$, the $\sz$ and $\zs$ statistics seem powerful and relatively easy to compute. In addition for searching for a proof of Conjecture~\ref{conj:shellability}, we wonder whether there are analogous statistics for other framings and whether one of these framed statistics agrees with an existing statistic.

\begin{open}
Prove Conjecture~\ref{conj:shellability}.
\end{open}

\begin{open}
\label{open:4}
Are $\mathsf{sz}^{F}$ and $\mathsf{zs}^{F}$ well defined for other framings $F$ on $\cZig_n$? Or for any framed graph $G$? 
\end{open}

\begin{open}
Is there a framing $F$ on $\cZig_n$ such that $\mathsf{sz}^{F}$ or $\mathsf{zs}^{F}$ agrees with either $\swap$ or $\des$?
\end{open}

We have verified that $\mathsf{sz}^{F}$ and $\mathsf{zs}^{F}$ do not agree with either $\swap$ or $\des$ for the planar and length framings on $\cZig_5$.


\section*{Acknowledgments}

We thank Carolina Benedetti, Rafael Gonz\'alez D'Le\'on, Alejandro Morales, and Martha Yip for fruitful conversations. This material is based upon work supported by the National Science Foundation under Grant Number 2150251 through the Queens Experiences in Discrete Mathematics REU at York College, CUNY and under Grant No. DMS-1929284 while the second author was in residence at the Institute for Computational and Experimental Research in Mathematics in Providence, RI, during the Computing Volumes of Flow Polytopes Collaborate$@$ICERM Program.

\bibliographystyle{alpha}
\bibliography{zigzag}

\end{document}

%% file: figures/zigzag.tikz
\begin{tikzpicture}[scale=0.7]
	\filldraw[black] (0,0) circle (5pt) node[label=below:{\small $1$}](a1) {};
	\filldraw[black] (1.75,0) circle (5pt) node[label={[label distance=0.05cm]90:\small $2$}](a2) {};
	\filldraw[black] (3.5,0) circle (5pt) node[label={[label distance=0.075cm]270:\small $3$}](a3) {};
	\filldraw[black] (5.25,0) circle (5pt) node[label={[label distance=0.075cm]90:\small $4$}](a4) {};
	\filldraw[black] (7,0) circle (5pt) node[label={[label distance=0.075cm]270:\small $5$}](a5) {};
	\filldraw[black] (8.75,0) circle (5pt) node[label={[label distance=0.05cm]90:\small $6$}](a6) {};
	\filldraw[black] (10.5,0) circle (5pt) node[label=below:{\small $7$}](a7) {};

	\filldraw[red] (1.75,-0.65) circle (3pt) node[](b1) {};
	\filldraw[red] (3.5,0.65) circle (3pt) node[](b2) {};
	\filldraw[red] (5.25,-0.65) circle (3pt) node[](b3) {};
	\filldraw[red] (7.0,0.65) circle (3pt) node[](b4) {};
	\filldraw[red] (8.75,-0.65) circle (3pt) node[](b5) {};

	\draw[-,red] (b1)--(b2)--(b3)--(b4)--(b5);		

	\draw[->] (a1)--(a2);		
	\draw[->] (a2)--(a3);		
	\draw[->] (a3)--(a4);		
	\draw[->] (a4)--(a5);		
	\draw[->] (a5)--(a6);			
	\draw[->] (a6)--(a7);			

	\draw[->]    (a1) to[out=-65,in=-115] (a3);
	\draw[->]    (a2) to[out=65,in=115] (a4);	
	\draw[->]    (a3) to[out=-65,in=-115] (a5);
	\draw[->]    (a4) to[out=65,in=115] (a6);	
	\draw[->]    (a5) to[out=-65,in=-115] (a7);
\end{tikzpicture}

%% file: figures/contracted.tikz
\begin{tikzpicture}[scale=0.6]
    \filldraw[black] (0,0) circle (5pt) node[label=below:{\small $1$}](a1) {};
	\filldraw[black] (2,0) circle (5pt) node[label=below:{\small $2$}](a2) {};
	\filldraw[black] (4,0) circle (5pt) node[label=below:{\small $3$}](a3) {};
	\filldraw[black] (6,0) circle (5pt) node[label=below:{\small $4$}](a4) {};
	\filldraw[black] (10,0) circle (5pt) node[label={[label distance=0.05cm]270:\small $n-2$}](an2) {};
	\filldraw[black] (12,0) circle (5pt) node[label={[label distance=0.05cm]270:\small $n-1$}](an1) {};
	\filldraw[black] (14,0) circle (5pt) node[label={[label distance=0.1cm]270:\small $n$}](an) {};
    \node[] at (8,0.5) {$\cdots$};
    \node[] at (1,0.16) {\scriptsize $x_1$};
    \node[] at (3,0.16) {\scriptsize $x_2$};
    \node[] at (5,0.16) {\scriptsize $x_3$};
    \node[] at (11,0.16) {\scriptsize $x_{n-2}$};
    \node[] at (13,0.16) {\scriptsize $x_{n-1}$};
    \node[] at (1,0.71) {\scriptsize $y_0$};
    \node[] at (13,0.71) {\scriptsize $y_{n-1}$};
    \node[] at (2,1.45) {\scriptsize $y_1$};
    \node[] at (4,1.45) {\scriptsize $y_2$};
    \node[] at (12,1.45) {\scriptsize $y_{n-2}$};

	\draw[->] (a1)--(a2);		
	\draw[->] (a2)--(a3);		
	\draw[->] (a3)--(a4);		
	\draw[-] (a4)--(7,0);		
	\draw[->] (9,0)--(an2);		
	\draw[->] (an2)--(an1);		
	\draw[->] (an1)--(an);		

	\draw[->]    (a1) to[out=40,in=140] (a2);
	\draw[->]    (a1) to[out=60,in=120] (a3);
	\draw[->]    (a2) to[out=60,in=120] (a4);	
	\draw[-]    (a3) to[out=60,in=145] (7,0.8);
	\draw[->]    (9,0.8) to[out=35,in=120] (an1);
	\draw[->]    (an2) to[out=60,in=120] (an);	
	\draw[->]    (an1) to[out=40,in=140] (an);	
\end{tikzpicture}

%% file: figures/framings.tikz
\begin{tikzpicture}[scale=0.7]

\begin{scope}
	\node[] at (-0.5,1.4) {(a)};
    \filldraw[black] (0,0) circle (3pt) node[label=below:{\tiny\color{orange} $1$}](a1) {};
	\filldraw[black] (2,0) circle (3pt) node[label=below:{\tiny\color{orange} $2$}](a2) {};
	\filldraw[black] (4,0) circle (3pt) node[label=below:{\tiny\color{orange} $3$}](a3) {};
	\filldraw[black] (6,0) circle (3pt) node[label=below:{\tiny\color{orange} $4$}](a4) {};
    \filldraw[black] (8,0) circle (3pt) node[label=below:{\tiny\color{orange} $5$}](a5) {};

    \draw[->] (a1)--(a2);		
	\draw[->] (a2)--(a3);		
	\draw[->] (a3)--(a4);		
	\draw[->] (a4)--(a5);		

	\draw[->]    (a1) to[out=-40,in=360-140] (a2);
	\draw[->]    (a1) to[out=60,in=120] (a3);
	\draw[->]    (a2) to[out=-60,in=360-120] (a4);	
	\draw[->]    (a3) to[out=60,in=120] (a5);	
	\draw[->]    (a4) to[out=-40,in=360-140] (a5);	

    \node[] at (1,0.2) {\tiny $x_1$};
    \node[] at (3,0.2) {\tiny $x_2$};
    \node[] at (5,0.2) {\tiny $x_3$};
    \node[] at (7,0.2) {\tiny $x_4$};
    \node[] at (1,-0.3) {\tiny $y_0$};
    \node[] at (2,1.45) {\tiny $y_1$};
    \node[] at (4,-0.95) {\tiny $y_2$};
    \node[] at (6,1.45) {\tiny $y_3$};
    \node[] at (7,-0.3) {\tiny $y_4$};
\end{scope}

\begin{scope}[xshift=300]
	\node[] at (-0.5,1.4) {(b)};
    \filldraw[black] (0,0) circle (3pt) node[label=below:{\tiny\color{orange} $1$}](a1) {};
	\filldraw[black] (2,0) circle (3pt) node[label=below:{\tiny\color{orange} $2$}](a2) {};
	\filldraw[black] (4,0) circle (3pt) node[label=below:{\tiny\color{orange} $3$}](a3) {};
	\filldraw[black] (6,0) circle (3pt) node[label=below:{\tiny\color{orange} $4$}](a4) {};
    \filldraw[black] (8,0) circle (3pt) node[label=below:{\tiny\color{orange} $5$}](a5) {};

    \draw[->] (a1)--(a2);		
	\draw[->] (a2)--(a3);		
	\draw[->] (a3)--(a4);		
	\draw[->] (a4)--(a5);		

	\draw[->]    (a1) to[out=-40,in=360-140] (a2);
	\draw[->]    (a1) to[out=-60,in=360-120] (a3);
	\draw[->]    (a2) to[out=-60,in=360-120] (a4);	
	\draw[->]    (a3) to[out=-60,in=360-120] (a5);	
	\draw[->]    (a4) to[out=-40,in=360-140] (a5);	

    \node[] at (1,0.2) {\tiny $x_1$};
    \node[] at (3,0.2) {\tiny $x_2$};
    \node[] at (5,0.2) {\tiny $x_3$};
    \node[] at (7,0.2) {\tiny $x_4$};
    \node[] at (1,-0.3) {\tiny $y_0$};
    \node[] at (2,-0.95) {\tiny $y_1$};
    \node[] at (4,-0.95) {\tiny $y_2$};
    \node[] at (6,-0.95) {\tiny $y_3$};
    \node[] at (7,-0.3) {\tiny $y_4$};
\end{scope}

\begin{scope}[xshift=150, yshift=-100]
	\node[] at (-0.5,1.4) {(c)};
    \filldraw[black] (0,0) circle (3pt) node[label=below:{\tiny\color{orange} $1$}](a1) {};
	\filldraw[black] (2,0) circle (3pt) node[label=below:{\tiny\color{orange} $2$}](a2) {};
	\filldraw[black] (4,0) circle (3pt) node[label=below:{\tiny\color{orange} $3$}](a3) {};
	\filldraw[black] (6,0) circle (3pt) node[label=below:{\tiny\color{orange} $4$}](a4) {};
    \filldraw[black] (8,0) circle (3pt) node[label=below:{\tiny\color{orange} $5$}](a5) {};

    \draw[->] (a1)--(a2);		
	\draw[->] (a2)--(a3);		
	\draw[->] (a3)--(a4);		
	\draw[->] (a4)--(a5);		

\draw[->, shorten <= 1.5mm, shorten >= 1.5mm] (0,0) .. controls (1,1.1) and (1.4,-0.6) .. (2,0);
\draw[->, shorten <= 1.5mm, shorten >= 1.5mm] (0,0) .. controls (1.5,2.5) and (3.4,-0.9) .. (4,0);
\draw[->, shorten <= 1.5mm, shorten >= 1.5mm] (2,0) .. controls (3.5,2.5) and (5.4,-0.9) .. (6,0);
\draw[->, shorten <= 1.5mm, shorten >= 1.5mm] (4,0) .. controls (5.5,2.5) and (7.4,-0.9) .. (8,0);
\draw[->, shorten <= 1.5mm, shorten >= 1.5mm] (6,0) .. controls (7,1.1) and (7.4,-0.4) .. (8,0);

    \node[] at (0.7,-0.25) {\tiny $x_1$};
    \node[] at (2.6,-0.25) {\tiny $x_2$};
    \node[] at (4.6,-0.25) {\tiny $x_3$};
    \node[] at (6.6,-0.25) {\tiny $x_4$};
    \node[] at (1.0,0.6) {\tiny $y_0$};
    \node[] at (1.4,1.25) {\tiny $y_1$};
    \node[] at (3.4,1.25) {\tiny $y_2$};
    \node[] at (5.4,1.25) {\tiny $y_3$};
    \node[] at (7.1,0.65) {\tiny $y_4$};
\end{scope}
\end{tikzpicture}  

%% file: figures/prefixorders.tikz
\begin{tikzpicture}[scale=0.9]

\definecolor{cbBlue}{HTML}{0072B2}
\definecolor{cbVermillion}{HTML}{D55E00}
\definecolor{cbGreen}{HTML}{009E73}
\definecolor{cbPurple}{HTML}{CC79A7}
\definecolor{cbOrange}{HTML}{E69F00}
\definecolor{cbSkyBlue}{HTML}{56B4E9}

\tikzset{
    edge/.style={
        ->,
        line width=1.5pt,
    },
    pathD/.style={edge, gray},
    pathE/.style={edge, cbVermillion},
    pathF/.style={edge, cbSkyBlue},
    pathG/.style={edge, cbGreen},
    pathH/.style={edge, cbOrange},
    pathI/.style={edge,cbPurple}
}

\begin{scope}

\draw[pathI,-] 
    (5,0.01) to[out=55,in=125,looseness=0.9]
        node[pos=0.65, circle, fill=black, inner sep=2pt, label={[yshift=4pt]right:{\tiny $e'$}}] {}
    (7,0.01)
    -- (8,0.01);

\draw[pathE,-] 
    (5,-0.13) to[out=-55,in=-125,looseness=0.9]
        node[pos=0.65, circle, fill=black, inner sep=2pt,
             label={[yshift=-4pt]right:{\tiny $e$}}] {}
    (7,-0.13)
    -- (8,-0.13);

\node[left, cbPurple] at (5,.1) {\tiny $P'$};
\node[left, cbVermillion] at (4.9,-0.3) {\tiny $P$};

\coordinate (v1) at (7,-.07);
\filldraw[black] (v1) circle (2.7pt);
\node[below] at (v1) {\tiny $j$};

\coordinate (v2) at (8,-.07);
\filldraw[black] (v2) circle (2.7pt);
\node[below] at (v2) {\tiny $i$};

\end{scope}

\begin{scope}

\draw[pathG,-] 
    (9,0.01) -- (10,0.01)
    to[out=55,in=125,looseness=0.9]
        node[pos=0.4, circle, fill=black, inner sep=2pt,
             label={[yshift=4pt]left:{\tiny $e'$}}] {}
    (12,0.01);

\draw[pathF,-] 
    (9,-0.13) -- (10,-0.13)
    to[out=-55,in=-125,looseness=0.9]
        node[pos=0.4, circle, fill=black, inner sep=2pt,
             label={[yshift=-4pt]left:{\tiny $e$}}] {}
    (12,-0.13);

\node[right, cbGreen] at (12,.1) {\tiny $Q'$};
\node[right, cbSkyBlue] at (12,-0.3) {\tiny $Q$};

\coordinate (v3) at (9,-.07);
\filldraw[black] (v3) circle (2.7pt);
\node[below] at (v3) {\tiny $i$};

\coordinate (v4) at (10,-.07);
\filldraw[black] (v4) circle (2.7pt);
\node[below] at (v4) {\tiny $k$};

\end{scope}

\end{tikzpicture}

%% file: figures/zig4unitflows.v2.tikz
\begin{minipage}{2.05in}
\begin{tikzpicture}[scale=0.61]
	\filldraw[black] (0,0) circle (2pt) node[](a1) {};
	\filldraw[black] (2,0) circle (2pt) node[](a2) {};
	\filldraw[black] (4,0) circle (2pt) node[](a3) {};
	\filldraw[black] (6,0) circle (2pt) node[](a4) {};


\draw[-] (0,0) .. controls (1,1.1) and (1.4,-0.6) .. (2,0);
\draw[-] (2,0) .. controls (3.5,2.5) and (5.4,-0.9) .. (6,0);

\node[shape=rectangle,draw,inner sep=2pt](name) at (7,0) {\tiny 1};
\end{tikzpicture}

\vspace{-0.5in}
\begin{tikzpicture}[scale=0.61]
	\filldraw[black] (0,0) circle (2pt) node[](a1) {};
	\filldraw[black] (2,0) circle (2pt) node[](a2) {};
	\filldraw[black] (4,0) circle (2pt) node[](a3) {};
	\filldraw[black] (6,0) circle (2pt) node[](a4) {};
\begin{scope}[node distance=0cm]
	\draw (a1.center)--(a2.center);		
\draw[-] (2,0) .. controls (3.5,2.5) and (5.4,-0.9) .. (6,0);

\end{scope}
\node[shape=rectangle,draw,inner sep=2pt](name) at (7,0) {\tiny 2};
\end{tikzpicture}

\vspace{-.5in}
\begin{tikzpicture}[scale=0.61]
	\filldraw[black] (0,0) circle (2pt) node[](a1) {};
	\filldraw[black] (2,0) circle (2pt) node[](a2) {};
	\filldraw[black] (4,0) circle (2pt) node[](a3) {};
	\filldraw[black] (6,0) circle (2pt) node[](a4) {};

\draw[-] (0,0) .. controls (1.5,2.5) and (3.4,-0.9) .. (4,0);
    \draw[-] (4,0) .. controls (5,1.1) and (5.4,-0.4) .. (6,0);

\node[shape=rectangle,draw,inner sep=2pt](name) at (7,0) {\tiny 3};
\end{tikzpicture}  

\vspace{-0.2in}
\begin{tikzpicture}[scale=0.61]
	\filldraw[black] (0,0) circle (2pt) node[](a1) {};
	\filldraw[black] (2,0) circle (2pt) node[](a2) {};
	\filldraw[black] (4,0) circle (2pt) node[](a3) {};
	\filldraw[black] (6,0) circle (2pt) node[](a4) {};

    \draw[-] (0,0) .. controls (1,1.1) and (1.4,-0.6) .. (2,0);
	\draw[-] (a2.center)--(a3.center);		
    \draw[-] (4,0) .. controls (5,1.1) and (5.4,-0.4) .. (6,0);

\node[shape=rectangle,draw,inner sep=2pt](name) at (7,0) {\tiny 4};
\end{tikzpicture}  
\end{minipage}%
\begin{minipage}{2.05in}
\vspace{0.3in}
\begin{tikzpicture}[scale=0.61]
	\filldraw[black] (0,0) circle (2pt) node[](a1) {};
	\filldraw[black] (2,0) circle (2pt) node[](a2) {};
	\filldraw[black] (4,0) circle (2pt) node[](a3) {};
	\filldraw[black] (6,0) circle (2pt) node[](a4) {};

	\draw[-] (a1.center)--(a2.center);		
	\draw[-] (a2.center)--(a3.center);		

    \draw[-] (4,0) .. controls (5,1.1) and (5.4,-0.4) .. (6,0);
\node[shape=rectangle,draw,inner sep=2pt](name) at (7,0) {\tiny 5};
\end{tikzpicture}

\vspace{-0.4in}
\begin{tikzpicture}[scale=0.61]
	\filldraw[black] (0,0) circle (2pt) node[](a1) {};
	\filldraw[black] (2,0) circle (2pt) node[](a2) {};
	\filldraw[black] (4,0) circle (2pt) node[](a3) {};
	\filldraw[black] (6,0) circle (2pt) node[](a4) {};

	\draw[-] (a3.center)--(a4.center);		

\draw[-] (0,0) .. controls (1.5,2.5) and (3.4,-0.9) .. (4,0);
\node[shape=rectangle,draw,inner sep=2pt](name) at (7,0) {\tiny 6};
\end{tikzpicture}  

\vspace{-0.2in}
\begin{tikzpicture}[scale=0.61]
	\filldraw[black] (0,0) circle (2pt) node[](a1) {};
	\filldraw[black] (2,0) circle (2pt) node[](a2) {};
	\filldraw[black] (4,0) circle (2pt) node[](a3) {};
	\filldraw[black] (6,0) circle (2pt) node[](a4) {};

	\draw[-] (a2.center)--(a3.center);		
	\draw[-] (a3.center)--(a4.center);		

\draw[-] (0,0) .. controls (1,1.1) and (1.4,-0.6) .. (2,0);
\node[shape=rectangle,draw,inner sep=2pt](name) at (7,0) {\tiny 7};
\end{tikzpicture}  

\vspace{0.05in}
\begin{tikzpicture}[scale=0.61]
	\filldraw[black] (0,0) circle (2pt) node[](a1) {};
	\filldraw[black] (2,0) circle (2pt) node[](a2) {};
	\filldraw[black] (4,0) circle (2pt) node[](a3) {};
	\filldraw[black] (6,0) circle (2pt) node[](a4) {};

	\draw[-] (a1.center)--(a2.center);		
	\draw[-] (a2.center)--(a3.center);		
	\draw[-] (a3.center)--(a4.center);		

\node[shape=rectangle,draw,inner sep=2pt](name) at (7,0) {\tiny 8};
\end{tikzpicture}  
\end{minipage}

%% file: figures/zig4example.tikz
    
\begin{minipage}{1.9in}
\begin{tikzpicture}[scale=0.55]
\begin{scope}
	\filldraw[black] (0,0) circle (3pt) node[label=below:{\tiny\color{orange} $1$}](a1) {};
	\filldraw[black] (2,0) circle (3pt) node[label=below:{\tiny\color{orange} $2$}](a2) {};
	\filldraw[black] (4,0) circle (3pt) node[label=below:{\tiny\color{orange} $3$}](a3) {};
	\filldraw[black] (6,0) circle (3pt) node[label=below:{\tiny\color{orange} $4$}](a4) {};

    \draw[->] (a1)--(a2);		
	\draw[->] (a2)--(a3);		
	\draw[->] (a3)--(a4);			

\draw[->, shorten <= 1.5mm, shorten >= 1.5mm] (0,0) .. controls (1,1.1) and (1.4,-0.6) .. (2,0);
\draw[->, shorten <= 1.5mm, shorten >= 1.5mm] (0,0) .. controls (1.5,2.5) and (3.4,-0.9) .. (4,0);
\draw[->, shorten <= 1.5mm, shorten >= 1.5mm] (2,0) .. controls (3.5,2.5) and (5.4,-0.9) .. (6,0);
\draw[->, shorten <= 1.5mm, shorten >= 1.5mm] (4,0) .. controls (5,1.1) and (5.4,-0.4) .. (6,0);

    \node[] at (0.7,-0.25) {\tiny\color{blue} $0$};
    \node[] at (2.6,-0.25) {\tiny\color{blue} $1$};
    \node[] at (4.6,-0.25) {\tiny\color{blue} $2$};
    \node[] at (1.0,0.6) {\tiny\color{blue} $0$};
    \node[] at (1.4,1.25) {\tiny\color{blue} $0$};
    \node[] at (3.4,1.25) {\tiny\color{blue} $0$};
    \node[] at (5.1,0.65) {\tiny\color{blue} $0$};
\end{scope}



    \node[] at (3,2.4) {\small Simplex \triangled{$\mathsf{A}$}};
\end{tikzpicture}  

\vspace{-.3in}
\begin{tikzpicture}[scale=0.55]
	\filldraw[black] (0,0) circle (2pt) node[](a1) {};
	\filldraw[black] (2,0) circle (2pt) node[](a2) {};
	\filldraw[black] (4,0) circle (2pt) node[](a3) {};
	\filldraw[black] (6,0) circle (2pt) node[](a4) {};
\begin{scope}[node distance=0cm]
	\draw (a1.center)--(a2.center);		
\draw[-] (2,0) .. controls (3.5,2.5) and (5.4,-0.9) .. (6,0);

\end{scope}
\node[shape=rectangle,draw,inner sep=2pt](name) at (7,0) {\tiny 2};
\end{tikzpicture}  

\vspace{-.17in}
\begin{tikzpicture}[scale=0.55]
	\filldraw[black] (0,0) circle (2pt) node[](a1) {};
	\filldraw[black] (2,0) circle (2pt) node[](a2) {};
	\filldraw[black] (4,0) circle (2pt) node[](a3) {};
	\filldraw[black] (6,0) circle (2pt) node[](a4) {};

	\draw[-] (a1.center)--(a2.center);		
	\draw[-] (a2.center)--(a3.center);		

    \draw[-] (4,0) .. controls (5,1.1) and (5.4,-0.4) .. (6,0);
\node[shape=rectangle,draw,inner sep=2pt](name) at (7,0) {\tiny 5};
\end{tikzpicture}  

\vspace{-.37in}
\begin{tikzpicture}[scale=0.55]
	\filldraw[black] (0,0) circle (2pt) node[](a1) {};
	\filldraw[black] (2,0) circle (2pt) node[](a2) {};
	\filldraw[black] (4,0) circle (2pt) node[](a3) {};
	\filldraw[black] (6,0) circle (2pt) node[](a4) {};

	\draw[-] (a3.center)--(a4.center);		

\draw[-] (0,0) .. controls (1.5,2.5) and (3.4,-0.9) .. (4,0);
\node[shape=rectangle,draw,inner sep=2pt](name) at (7,0) {\tiny 6};
\end{tikzpicture}   

\vspace{-.18in}
\begin{tikzpicture}[scale=0.55]
	\filldraw[black] (0,0) circle (2pt) node[](a1) {};
	\filldraw[black] (2,0) circle (2pt) node[](a2) {};
	\filldraw[black] (4,0) circle (2pt) node[](a3) {};
	\filldraw[black] (6,0) circle (2pt) node[](a4) {};

	\draw[-] (a2.center)--(a3.center);		
	\draw[-] (a3.center)--(a4.center);		

\draw[-] (0,0) .. controls (1,1.1) and (1.4,-0.6) .. (2,0);
\node[shape=rectangle,draw,inner sep=2pt](name) at (7,0) {\tiny 7};
\end{tikzpicture}   

\vspace{0in}
\begin{tikzpicture}[scale=0.55]
	\filldraw[black] (0,0) circle (2pt) node[](a1) {};
	\filldraw[black] (2,0) circle (2pt) node[](a2) {};
	\filldraw[black] (4,0) circle (2pt) node[](a3) {};
	\filldraw[black] (6,0) circle (2pt) node[](a4) {};

	\draw[-] (a1.center)--(a2.center);		
	\draw[-] (a2.center)--(a3.center);		
	\draw[-] (a3.center)--(a4.center);		

\node[shape=rectangle,draw,inner sep=2pt](name) at (7,0) {\tiny 8};
\end{tikzpicture}  

\vspace{.2in}
\hspace{.3in}
\begin{tikzpicture}[scale=0.7]
   
	\draw[-,thick] (0,1)--(1,1);		
	\draw[-,thick] (0,1)--(1,0);	
    \draw[-,thick] (0,0)--(1,0); 

    \draw[-,thick] (2,1.1)--(3,1);
    \draw[-,thick] (2,1.1)--(3,0.25);
    \draw[-,thick] (2,0.5)--(3,0.25);		
	\draw[-,thick] (2,-0.1)--(3,0.25);	

\end{tikzpicture}
\end{minipage}\qquad%
\begin{minipage}{1.9in}
\begin{tikzpicture}[scale=0.55]


    \node[] at (3,2.4) {\small Simplex \triangled{$\mathsf{B}$}};
\begin{scope}
	\filldraw[black] (0,0) circle (3pt) node[label=below:{\tiny\color{orange} $1$}](a1) {};
	\filldraw[black] (2,0) circle (3pt) node[label=below:{\tiny\color{orange} $2$}](a2) {};
	\filldraw[black] (4,0) circle (3pt) node[label=below:{\tiny\color{orange} $3$}](a3) {};
	\filldraw[black] (6,0) circle (3pt) node[label=below:{\tiny\color{orange} $4$}](a4) {};

    \draw[->] (a1)--(a2);		
	\draw[->] (a2)--(a3);		
	\draw[->] (a3)--(a4);			

\draw[->, shorten <= 1.5mm, shorten >= 1.5mm] (0,0) .. controls (1,1.1) and (1.4,-0.6) .. (2,0);
\draw[->, shorten <= 1.5mm, shorten >= 1.5mm] (0,0) .. controls (1.5,2.5) and (3.4,-0.9) .. (4,0);
\draw[->, shorten <= 1.5mm, shorten >= 1.5mm] (2,0) .. controls (3.5,2.5) and (5.4,-0.9) .. (6,0);
\draw[->, shorten <= 1.5mm, shorten >= 1.5mm] (4,0) .. controls (5,1.1) and (5.4,-0.4) .. (6,0);

    \node[] at (0.7,-0.25) {\tiny\color{blue} $0$};
    \node[] at (2.6,-0.25) {\tiny\color{blue} $1$};
    \node[] at (4.6,-0.25) {\tiny\color{blue} $1$};
    \node[] at (1.0,0.6) {\tiny\color{blue} $0$};
    \node[] at (1.4,1.25) {\tiny\color{blue} $0$};
    \node[] at (3.4,1.25) {\tiny\color{blue} $0$};
    \node[] at (5.1,0.65) {\tiny\color{blue} $1$};
\end{scope}
\end{tikzpicture}

\vspace{-.3in}
\begin{tikzpicture}[scale=0.55]
	\filldraw[black] (0,0) circle (2pt) node[](a1) {};
	\filldraw[black] (2,0) circle (2pt) node[](a2) {};
	\filldraw[black] (4,0) circle (2pt) node[](a3) {};
	\filldraw[black] (6,0) circle (2pt) node[](a4) {};
\begin{scope}[node distance=0cm]
	\draw (a1.center)--(a2.center);		
\draw[-] (2,0) .. controls (3.5,2.5) and (5.4,-0.9) .. (6,0);

\end{scope}
\node[shape=rectangle,draw,inner sep=2pt](name) at (7,0) {\tiny 2};
\end{tikzpicture}  

\vspace{-.18in}
\begin{tikzpicture}[scale=0.55]
	\filldraw[black] (0,0) circle (2pt) node[](a1) {};
	\filldraw[black] (2,0) circle (2pt) node[](a2) {};
	\filldraw[black] (4,0) circle (2pt) node[](a3) {};
	\filldraw[black] (6,0) circle (2pt) node[](a4) {};

    \draw[-] (0,0) .. controls (1,1.1) and (1.4,-0.6) .. (2,0);
	\draw[-] (a2.center)--(a3.center);		
    \draw[-] (4,0) .. controls (5,1.1) and (5.4,-0.4) .. (6,0);

\node[shape=rectangle,draw,inner sep=2pt](name) at (7,0) {\tiny 4};
\end{tikzpicture}  

\vspace{-.12in}
\begin{tikzpicture}[scale=0.55]
	\filldraw[black] (0,0) circle (2pt) node[](a1) {};
	\filldraw[black] (2,0) circle (2pt) node[](a2) {};
	\filldraw[black] (4,0) circle (2pt) node[](a3) {};
	\filldraw[black] (6,0) circle (2pt) node[](a4) {};

	\draw[-] (a1.center)--(a2.center);		
	\draw[-] (a2.center)--(a3.center);		

    \draw[-] (4,0) .. controls (5,1.1) and (5.4,-0.4) .. (6,0);
\node[shape=rectangle,draw,inner sep=2pt](name) at (7,0) {\tiny 5};
\end{tikzpicture}  

\vspace{-.37in}
\begin{tikzpicture}[scale=0.55]
	\filldraw[black] (0,0) circle (2pt) node[](a1) {};
	\filldraw[black] (2,0) circle (2pt) node[](a2) {};
	\filldraw[black] (4,0) circle (2pt) node[](a3) {};
	\filldraw[black] (6,0) circle (2pt) node[](a4) {};

	\draw[-] (a3.center)--(a4.center);		

\draw[-] (0,0) .. controls (1.5,2.5) and (3.4,-0.9) .. (4,0);
\node[shape=rectangle,draw,inner sep=2pt](name) at (7,0) {\tiny 6};
\end{tikzpicture}  

\vspace{-.2in}
\begin{tikzpicture}[scale=0.55]
	\filldraw[black] (0,0) circle (2pt) node[](a1) {};
	\filldraw[black] (2,0) circle (2pt) node[](a2) {};
	\filldraw[black] (4,0) circle (2pt) node[](a3) {};
	\filldraw[black] (6,0) circle (2pt) node[](a4) {};

	\draw[-] (a2.center)--(a3.center);		
	\draw[-] (a3.center)--(a4.center);		

\draw[-] (0,0) .. controls (1,1.1) and (1.4,-0.6) .. (2,0);
\node[shape=rectangle,draw,inner sep=2pt](name) at (7,0) {\tiny 7};
\end{tikzpicture}  

\vspace{.2in}
\hspace{.3in}
\begin{tikzpicture}[scale=0.7]
   
	\draw[-,thick] (0,1)--(1,1);		
	\draw[-,thick] (0,1)--(1,0);	
    \draw[-,thick] (0,0)--(1,0); 

    \draw[-,thick] (2,1.1)--(3,0.75);
    \draw[-,thick] (2,0.5)--(3,0.75);
    \draw[-,thick] (2,0.5)--(3,0.25);		
	\draw[-,thick] (2,-0.1)--(3,0.25);	

\end{tikzpicture}

\end{minipage}\qquad%
\begin{minipage}{1.9in}
\begin{tikzpicture}[scale=0.55]


    \node[] at (3,2.4) {\small Simplex \triangled{$\mathsf{C}$}};
\begin{scope}
	\filldraw[black] (0,0) circle (3pt) node[label=below:{\tiny\color{orange} $1$}](a1) {};
	\filldraw[black] (2,0) circle (3pt) node[label=below:{\tiny\color{orange} $2$}](a2) {};
	\filldraw[black] (4,0) circle (3pt) node[label=below:{\tiny\color{orange} $3$}](a3) {};
	\filldraw[black] (6,0) circle (3pt) node[label=below:{\tiny\color{orange} $4$}](a4) {};

    \draw[->] (a1)--(a2);		
	\draw[->] (a2)--(a3);		
	\draw[->] (a3)--(a4);			

\draw[->, shorten <= 1.5mm, shorten >= 1.5mm] (0,0) .. controls (1,1.1) and (1.4,-0.6) .. (2,0);
\draw[->, shorten <= 1.5mm, shorten >= 1.5mm] (0,0) .. controls (1.5,2.5) and (3.4,-0.9) .. (4,0);
\draw[->, shorten <= 1.5mm, shorten >= 1.5mm] (2,0) .. controls (3.5,2.5) and (5.4,-0.9) .. (6,0);
\draw[->, shorten <= 1.5mm, shorten >= 1.5mm] (4,0) .. controls (5,1.1) and (5.4,-0.4) .. (6,0);

    \node[] at (0.7,-0.25) {\tiny\color{blue} $0$};
    \node[] at (2.6,-0.25) {\tiny\color{blue} $1$};
    \node[] at (4.6,-0.25) {\tiny\color{blue} $0$};
    \node[] at (1.0,0.6) {\tiny\color{blue} $0$};
    \node[] at (1.4,1.25) {\tiny\color{blue} $0$};
    \node[] at (3.4,1.25) {\tiny\color{blue} $0$};
    \node[] at (5.1,0.65) {\tiny\color{blue} $2$};
\end{scope}
\end{tikzpicture}

\vspace{-.35in}
\begin{tikzpicture}[scale=0.55]
	\filldraw[black] (0,0) circle (2pt) node[](a1) {};
	\filldraw[black] (2,0) circle (2pt) node[](a2) {};
	\filldraw[black] (4,0) circle (2pt) node[](a3) {};
	\filldraw[black] (6,0) circle (2pt) node[](a4) {};
\begin{scope}[node distance=0cm]
	\draw (a1.center)--(a2.center);		
\draw[-] (2,0) .. controls (3.5,2.5) and (5.4,-0.9) .. (6,0);

\end{scope}
\node[shape=rectangle,draw,inner sep=2pt](name) at (7,0) {\tiny 2};
\end{tikzpicture}  

\vspace{-.48in}
\begin{tikzpicture}[scale=0.55]
	\filldraw[black] (0,0) circle (2pt) node[](a1) {};
	\filldraw[black] (2,0) circle (2pt) node[](a2) {};
	\filldraw[black] (4,0) circle (2pt) node[](a3) {};
	\filldraw[black] (6,0) circle (2pt) node[](a4) {};

\draw[-] (0,0) .. controls (1.5,2.5) and (3.4,-0.9) .. (4,0);
    \draw[-] (4,0) .. controls (5,1.1) and (5.4,-0.4) .. (6,0);

\node[shape=rectangle,draw,inner sep=2pt](name) at (7,0) {\tiny 3};
\end{tikzpicture}   

\vspace{-.18in}
\begin{tikzpicture}[scale=0.55]
	\filldraw[black] (0,0) circle (2pt) node[](a1) {};
	\filldraw[black] (2,0) circle (2pt) node[](a2) {};
	\filldraw[black] (4,0) circle (2pt) node[](a3) {};
	\filldraw[black] (6,0) circle (2pt) node[](a4) {};

    \draw[-] (0,0) .. controls (1,1.1) and (1.4,-0.6) .. (2,0);
	\draw[-] (a2.center)--(a3.center);		
    \draw[-] (4,0) .. controls (5,1.1) and (5.4,-0.4) .. (6,0);

\node[shape=rectangle,draw,inner sep=2pt](name) at (7,0) {\tiny 4};
\end{tikzpicture}  

\vspace{-.12in}
\begin{tikzpicture}[scale=0.55]
	\filldraw[black] (0,0) circle (2pt) node[](a1) {};
	\filldraw[black] (2,0) circle (2pt) node[](a2) {};
	\filldraw[black] (4,0) circle (2pt) node[](a3) {};
	\filldraw[black] (6,0) circle (2pt) node[](a4) {};

	\draw[-] (a1.center)--(a2.center);		
	\draw[-] (a2.center)--(a3.center);		

    \draw[-] (4,0) .. controls (5,1.1) and (5.4,-0.4) .. (6,0);
\node[shape=rectangle,draw,inner sep=2pt](name) at (7,0) {\tiny 5};
\end{tikzpicture}  

\vspace{-.38in}
\begin{tikzpicture}[scale=0.55]
	\filldraw[black] (0,0) circle (2pt) node[](a1) {};
	\filldraw[black] (2,0) circle (2pt) node[](a2) {};
	\filldraw[black] (4,0) circle (2pt) node[](a3) {};
	\filldraw[black] (6,0) circle (2pt) node[](a4) {};

	\draw[-] (a3.center)--(a4.center);		

\draw[-] (0,0) .. controls (1.5,2.5) and (3.4,-0.9) .. (4,0);
\node[shape=rectangle,draw,inner sep=2pt](name) at (7,0) {\tiny 6};
\end{tikzpicture}  

\vspace{.1in}
\hspace{.3in}
 \begin{tikzpicture}[scale=0.7]
   
	\draw[-,thick] (0,1)--(1,1);		
	\draw[-,thick] (0,1)--(1,0);	
    \draw[-,thick] (0,0)--(1,0); 

    \draw[-,thick] (2,1.1)--(3,0.75);
    \draw[-,thick] (2,0.5)--(3,0.75);
    \draw[-,thick] (2,-0.1)--(3,0.75);		
	\draw[-,thick] (2,-0.1)--(3,0);	

\end{tikzpicture}

\end{minipage}

\vspace{.2in}

\begin{minipage}{1.9in}
\begin{tikzpicture}[scale=0.55]


    \node[] at (3,2.4) {\small Simplex \triangled{$\mathsf{D}$}};
\begin{scope}
	\filldraw[black] (0,0) circle (3pt) node[label=below:{\tiny\color{orange} $1$}](a1) {};
	\filldraw[black] (2,0) circle (3pt) node[label=below:{\tiny\color{orange} $2$}](a2) {};
	\filldraw[black] (4,0) circle (3pt) node[label=below:{\tiny\color{orange} $3$}](a3) {};
	\filldraw[black] (6,0) circle (3pt) node[label=below:{\tiny\color{orange} $4$}](a4) {};

    \draw[->] (a1)--(a2);		
	\draw[->] (a2)--(a3);		
	\draw[->] (a3)--(a4);			

\draw[->, shorten <= 1.5mm, shorten >= 1.5mm] (0,0) .. controls (1,1.1) and (1.4,-0.6) .. (2,0);
\draw[->, shorten <= 1.5mm, shorten >= 1.5mm] (0,0) .. controls (1.5,2.5) and (3.4,-0.9) .. (4,0);
\draw[->, shorten <= 1.5mm, shorten >= 1.5mm] (2,0) .. controls (3.5,2.5) and (5.4,-0.9) .. (6,0);
\draw[->, shorten <= 1.5mm, shorten >= 1.5mm] (4,0) .. controls (5,1.1) and (5.4,-0.4) .. (6,0);

    \node[] at (0.7,-0.25) {\tiny\color{blue} $0$};
    \node[] at (2.6,-0.25) {\tiny\color{blue} $0$};
    \node[] at (4.6,-0.25) {\tiny\color{blue} $1$};
    \node[] at (1.0,0.6) {\tiny\color{blue} $0$};
    \node[] at (1.4,1.25) {\tiny\color{blue} $0$};
    \node[] at (3.4,1.25) {\tiny\color{blue} $1$};
    \node[] at (5.1,0.65) {\tiny\color{blue} $0$};
\end{scope}
\end{tikzpicture}

\vspace{-.3in}
\begin{tikzpicture}[scale=0.55]
	\filldraw[black] (0,0) circle (2pt) node[](a1) {};
	\filldraw[black] (2,0) circle (2pt) node[](a2) {};
	\filldraw[black] (4,0) circle (2pt) node[](a3) {};
	\filldraw[black] (6,0) circle (2pt) node[](a4) {};


\draw[-] (0,0) .. controls (1,1.1) and (1.4,-0.6) .. (2,0);
\draw[-] (2,0) .. controls (3.5,2.5) and (5.4,-0.9) .. (6,0);

\node[shape=rectangle,draw,inner sep=2pt](name) at (7,0) {\tiny 1};
\end{tikzpicture}  

\vspace{-.48in}
\begin{tikzpicture}[scale=0.55]
	\filldraw[black] (0,0) circle (2pt) node[](a1) {};
	\filldraw[black] (2,0) circle (2pt) node[](a2) {};
	\filldraw[black] (4,0) circle (2pt) node[](a3) {};
	\filldraw[black] (6,0) circle (2pt) node[](a4) {};
\begin{scope}[node distance=0cm]
	\draw (a1.center)--(a2.center);		
\draw[-] (2,0) .. controls (3.5,2.5) and (5.4,-0.9) .. (6,0);

\end{scope}
\node[shape=rectangle,draw,inner sep=2pt](name) at (7,0) {\tiny 2};
\end{tikzpicture}

\vspace{-.17in}
\begin{tikzpicture}[scale=0.55]
	\filldraw[black] (0,0) circle (2pt) node[](a1) {};
	\filldraw[black] (2,0) circle (2pt) node[](a2) {};
	\filldraw[black] (4,0) circle (2pt) node[](a3) {};
	\filldraw[black] (6,0) circle (2pt) node[](a4) {};

    \draw[-] (0,0) .. controls (1,1.1) and (1.4,-0.6) .. (2,0);
	\draw[-] (a2.center)--(a3.center);		
    \draw[-] (4,0) .. controls (5,1.1) and (5.4,-0.4) .. (6,0);

\node[shape=rectangle,draw,inner sep=2pt](name) at (7,0) {\tiny 4};
\end{tikzpicture}

\vspace{-.4in}
\begin{tikzpicture}[scale=0.55]
	\filldraw[black] (0,0) circle (2pt) node[](a1) {};
	\filldraw[black] (2,0) circle (2pt) node[](a2) {};
	\filldraw[black] (4,0) circle (2pt) node[](a3) {};
	\filldraw[black] (6,0) circle (2pt) node[](a4) {};

	\draw[-] (a3.center)--(a4.center);		

\draw[-] (0,0) .. controls (1.5,2.5) and (3.4,-0.9) .. (4,0);
\node[shape=rectangle,draw,inner sep=2pt](name) at (7,0) {\tiny 6};
\end{tikzpicture}  

\vspace{-.17in}
\begin{tikzpicture}[scale=0.55]
	\filldraw[black] (0,0) circle (2pt) node[](a1) {};
	\filldraw[black] (2,0) circle (2pt) node[](a2) {};
	\filldraw[black] (4,0) circle (2pt) node[](a3) {};
	\filldraw[black] (6,0) circle (2pt) node[](a4) {};

	\draw[-] (a2.center)--(a3.center);		
	\draw[-] (a3.center)--(a4.center);		

\draw[-] (0,0) .. controls (1,1.1) and (1.4,-0.6) .. (2,0);
\node[shape=rectangle,draw,inner sep=2pt](name) at (7,0) {\tiny 7};
\end{tikzpicture}   

\vspace{.2in}
\hspace{.3in}
\begin{tikzpicture}[scale=0.7]
   
	\draw[-,thick] (0,1)--(1,1);		
	\draw[-,thick] (0,0)--(1,1);	
    \draw[-,thick] (0,0)--(1,0); 

    \draw[-,thick] (2,1)--(3,1);
    \draw[-,thick] (2,1)--(3,0);
    \draw[-,thick] (2,0)--(3,0);		

\end{tikzpicture}

\end{minipage}\qquad%
\begin{minipage}{1.9in}
\begin{tikzpicture}[scale=0.55]


    \node[] at (3,2.4) {\small Simplex \triangled{$\mathsf{E}$}};
\begin{scope}
	\filldraw[black] (0,0) circle (3pt) node[label=below:{\tiny\color{orange} $1$}](a1) {};
	\filldraw[black] (2,0) circle (3pt) node[label=below:{\tiny\color{orange} $2$}](a2) {};
	\filldraw[black] (4,0) circle (3pt) node[label=below:{\tiny\color{orange} $3$}](a3) {};
	\filldraw[black] (6,0) circle (3pt) node[label=below:{\tiny\color{orange} $4$}](a4) {};

    \draw[->] (a1)--(a2);		
	\draw[->] (a2)--(a3);		
	\draw[->] (a3)--(a4);			

\draw[->, shorten <= 1.5mm, shorten >= 1.5mm] (0,0) .. controls (1,1.1) and (1.4,-0.6) .. (2,0);
\draw[->, shorten <= 1.5mm, shorten >= 1.5mm] (0,0) .. controls (1.5,2.5) and (3.4,-0.9) .. (4,0);
\draw[->, shorten <= 1.5mm, shorten >= 1.5mm] (2,0) .. controls (3.5,2.5) and (5.4,-0.9) .. (6,0);
\draw[->, shorten <= 1.5mm, shorten >= 1.5mm] (4,0) .. controls (5,1.1) and (5.4,-0.4) .. (6,0);

    \node[] at (0.7,-0.25) {\tiny\color{blue} $0$};
    \node[] at (2.6,-0.25) {\tiny\color{blue} $0$};
    \node[] at (4.6,-0.25) {\tiny\color{blue} $0$};
    \node[] at (1.0,0.6) {\tiny\color{blue} $0$};
    \node[] at (1.4,1.25) {\tiny\color{blue} $0$};
    \node[] at (3.4,1.25) {\tiny\color{blue} $1$};
    \node[] at (5.1,0.65) {\tiny\color{blue} $1$};
\end{scope}
\end{tikzpicture}

\vspace{-0.25in}
\begin{tikzpicture}[scale=0.55]
	\filldraw[black] (0,0) circle (2pt) node[](a1) {};
	\filldraw[black] (2,0) circle (2pt) node[](a2) {};
	\filldraw[black] (4,0) circle (2pt) node[](a3) {};
	\filldraw[black] (6,0) circle (2pt) node[](a4) {};


\draw[-] (0,0) .. controls (1,1.1) and (1.4,-0.6) .. (2,0);
\draw[-] (2,0) .. controls (3.5,2.5) and (5.4,-0.9) .. (6,0);

\node[shape=rectangle,draw,inner sep=2pt](name) at (7,0) {\tiny 1};
\end{tikzpicture}  

\vspace{-.48in}
\begin{tikzpicture}[scale=0.55]
	\filldraw[black] (0,0) circle (2pt) node[](a1) {};
	\filldraw[black] (2,0) circle (2pt) node[](a2) {};
	\filldraw[black] (4,0) circle (2pt) node[](a3) {};
	\filldraw[black] (6,0) circle (2pt) node[](a4) {};
\begin{scope}[node distance=0cm]
	\draw (a1.center)--(a2.center);		
\draw[-] (2,0) .. controls (3.5,2.5) and (5.4,-0.9) .. (6,0);

\end{scope}
\node[shape=rectangle,draw,inner sep=2pt](name) at (7,0) {\tiny 2};
\end{tikzpicture}

\vspace{-.48in}
\begin{tikzpicture}[scale=0.55]
	\filldraw[black] (0,0) circle (2pt) node[](a1) {};
	\filldraw[black] (2,0) circle (2pt) node[](a2) {};
	\filldraw[black] (4,0) circle (2pt) node[](a3) {};
	\filldraw[black] (6,0) circle (2pt) node[](a4) {};

\draw[-] (0,0) .. controls (1.5,2.5) and (3.4,-0.9) .. (4,0);
    \draw[-] (4,0) .. controls (5,1.1) and (5.4,-0.4) .. (6,0);

\node[shape=rectangle,draw,inner sep=2pt](name) at (7,0) {\tiny 3};
\end{tikzpicture}  

\vspace{-.18in}
\begin{tikzpicture}[scale=0.55]
	\filldraw[black] (0,0) circle (2pt) node[](a1) {};
	\filldraw[black] (2,0) circle (2pt) node[](a2) {};
	\filldraw[black] (4,0) circle (2pt) node[](a3) {};
	\filldraw[black] (6,0) circle (2pt) node[](a4) {};

    \draw[-] (0,0) .. controls (1,1.1) and (1.4,-0.6) .. (2,0);
	\draw[-] (a2.center)--(a3.center);		
    \draw[-] (4,0) .. controls (5,1.1) and (5.4,-0.4) .. (6,0);

\node[shape=rectangle,draw,inner sep=2pt](name) at (7,0) {\tiny 4};
\end{tikzpicture}

\vspace{-.4in}
\begin{tikzpicture}[scale=0.55]
	\filldraw[black] (0,0) circle (2pt) node[](a1) {};
	\filldraw[black] (2,0) circle (2pt) node[](a2) {};
	\filldraw[black] (4,0) circle (2pt) node[](a3) {};
	\filldraw[black] (6,0) circle (2pt) node[](a4) {};

	\draw[-] (a3.center)--(a4.center);		

\draw[-] (0,0) .. controls (1.5,2.5) and (3.4,-0.9) .. (4,0);
\node[shape=rectangle,draw,inner sep=2pt](name) at (7,0) {\tiny 6};
\end{tikzpicture}  

\vspace{.15in}
\hspace{.3in}
\begin{tikzpicture}[scale=0.7]
   
	\draw[-,thick] (0,1)--(1,1);		
	\draw[-,thick] (0,0)--(1,1);	
    \draw[-,thick] (0,0)--(1,0); 

    \draw[-,thick] (2,1)--(3,1);
    \draw[-,thick] (2,0)--(3,1);
    \draw[-,thick] (2,0)--(3,0);		

\end{tikzpicture}
\end{minipage}

%% file: figures/zig4projection.tikz
\begin{tikzpicture}[scale=1.5]
	\draw[-] (0,0)--(1,1)--(3,1)--(3,-1)--(1,-1)--(0,0);
    \draw[-,dashed] (1,1)--(1,-1);
    \draw[-,dashed] (1,1)--(3,-1);
    \draw[-,dashed] (3,1)--(1,-1);
	\draw[-,very thick,color=red] (0.6,0)--(1.4,0)--(2,0.6)--(2.6,0)--(2,-0.6)--(1.4,0);
    \node[regular polygon,regular polygon sides=4,draw,inner sep=0.1pt,fill=white](name) at (0,0) {8};
    \node[regular polygon,regular polygon sides=4,draw,inner sep=0.1pt,fill=white](name) at (1,1) {5};
    \node[regular polygon,regular polygon sides=4,draw,inner sep=0.1pt,fill=white](name) at (1,-1) {7};
    \node[regular polygon,regular polygon sides=4,draw,inner sep=0.1pt,fill=white](name) at (2,0) {4};
    \node[regular polygon,regular polygon sides=4,draw,inner sep=0.1pt,fill=white](name) at (3,1) {3};
    \node[regular polygon,regular polygon sides=4,draw,inner sep=0.1pt,fill=white](name) at (3,-1) {1};
	\node[regular polygon,regular polygon sides=3, inner sep=0.1pt,draw,fill=black!5] at (0.6,0) {\scriptsize $\mathsf{A}$};
	\node[regular polygon,regular polygon sides=3, inner sep=0.1pt,draw,fill=black!5] at (1.4,0) {\scriptsize $\mathsf{B}$};
	\node[regular polygon,regular polygon sides=3, inner sep=0.1pt,draw,fill=black!5] at (2,0.6) {\scriptsize $\mathsf{C}$};
	\node[regular polygon,regular polygon sides=3, inner sep=0.1pt,draw,fill=black!5] at (2,-0.6) {\scriptsize $\mathsf{D}$};
	\node[regular polygon,regular polygon sides=3, inner sep=0.1pt,draw,fill=black!5] at (2.6,0) {\scriptsize $\mathsf{E}$};
\end{tikzpicture}

%% file: figures/embedzig4.tikz
\begin{tikzpicture}[scale=1.8]
	\draw[-,color=black!75] (0,0)--(0,1)--(0,2)--(1,2)--(1,1)--(0,1);
    \filldraw[black] (0,0) circle (1pt);
    \filldraw[black] (0,1) circle (1pt);
    \filldraw[black] (0,2) circle (1pt);
    \filldraw[black] (1,1) circle (1pt);
    \filldraw[black] (1,2) circle (1pt);
\end{tikzpicture}

%% file: figures/embedzig5.tikz
\begin{tikzpicture}[scale=0.9]

	\draw[-,thick,color=black!50] (0,0)--(0,1)--(0,2)--(0,3);
	\draw[-,thick,color=black!50] (0+0.5,1+0.5)--(0+0.5,2+0.5)--(0+0.5,3+0.5);
    \draw[-,thick,color=black!50] (0+2*0.5,2+2*0.5)--(0+2*0.5,3+2*0.5);

    \draw[-,thick,color=black!50] (1+0.5,1+0.5)--(1+0.5,2+0.5)--(1+0.5,3+0.5)--(1+0.5,4+0.5);
	\draw[-,thick,color=black!50] (1+2*0.5,2+2*0.5)--(1+2*0.5,3+2*0.5)--(1+2*0.5,4+2*0.5);

    \draw[-,thick,color=black!50] (0,1)--(0+0.5,1+0.5);
    \draw[-,thick,color=black!50] (0,2)--(0+0.5,2+0.5)--(0+2*0.5,2+2*0.5);
    \draw[-,thick,color=black!50] (0,3)--(0+0.5,3+0.5)--(0+2*0.5,3+2*0.5);
    \draw[-,thick,color=black!50] (1+0.5,2+0.5)--(1+2*0.5,2+2*0.5);
    \draw[-,thick,color=black!50] (1+0.5,3+0.5)--(1+2*0.5,3+2*0.5);
    \draw[-,thick,color=black!50] (1+0.5,4+0.5)--(1+2*0.5,4+2*0.5);

	\draw[-,thick,color=black!50] (0+0.5,1+0.5)--(1+0.5,1+0.5);
	\draw[-,thick,color=black!50] (0+0.5,2+0.5)--(1+0.5,2+0.5);
	\draw[-,thick,color=black!50] (0+0.5,3+0.5)--(1+0.5,3+0.5);

	\draw[-,thick,color=black!50] (0+2*0.5,2+2*0.5)--(1+2*0.5,2+2*0.5);
	\draw[-,thick,color=black!50] (0+2*0.5,3+2*0.5)--(1+2*0.5,3+2*0.5);

    \filldraw[black] (0,0) circle (2pt);
    \filldraw[black] (0,1) circle (2pt);
    \filldraw[black] (0,2) circle (2pt);
    \filldraw[black] (0,3) circle (2pt);
    \filldraw[black] (0+0.5,1+0.5) circle (2pt);
    \filldraw[black] (0+0.5,2+0.5) circle (2pt);
    \filldraw[black] (0+0.5,3+0.5) circle (2pt);
    \filldraw[black] (0+2*0.5,2+2*0.5) circle (2pt);
    \filldraw[black] (0+2*0.5,3+2*0.5) circle (2pt);
    \filldraw[black] (1+0.5,1+0.5) circle (2pt);
    \filldraw[black] (1+0.5,2+0.5) circle (2pt);
    \filldraw[black] (1+0.5,3+0.5) circle (2pt);
    \filldraw[black] (1+0.5,4+0.5) circle (2pt);
    \filldraw[black] (1+2*0.5,2+2*0.5) circle (2pt);
    \filldraw[black] (1+2*0.5,3+2*0.5) circle (2pt);
    \filldraw[black] (1+2*0.5,4+2*0.5) circle (2pt);
\end{tikzpicture}

%% file: figures/adjacentflows.tikz
\begin{tabular}{ccc}

\begin{tikzpicture}[scale=0.5]
   \node[] at (-1.4,0.8) {$\phi_1$};
	\filldraw[black] (0,0) circle (3pt) node[label=below:{\tiny\color{orange} $1$}](a1) {};
	\filldraw[black] (2,0) circle (3pt) node[label=below:{\tiny\color{orange} $2$}](a2) {};
	\filldraw[black] (4,0) circle (3pt) node[label=below:{\tiny\color{orange} $3$}](a3) {};
	\filldraw[black] (6,0) circle (3pt) node[label=below:{\tiny\color{orange} $4$}](a4) {};
    \filldraw[black] (8,0) circle (3pt) node[label=below:{\tiny\color{orange} $5$}](a5) {};
    \filldraw[black] (10,0) circle (3pt) node[label=below:{\tiny\color{orange} $6$}](a6) {};

    \draw[->] (a1)--(a2);		
	\draw[->] (a2)--(a3);		
	\draw[->] (a3)--(a4);		
	\draw[->] (a4)--(a5);		
	\draw[->] (a5)--(a6);		

\draw[->, shorten <= 1.5mm, shorten >= 1.5mm] (0,0) .. controls (1,1.1) and (1.4,-0.6) .. (2,0);
\draw[->, shorten <= 1.5mm, shorten >= 1.5mm] (0,0) .. controls (1.5,2.5) and (3.4,-0.9) .. (4,0);
\draw[->, shorten <= 1.5mm, shorten >= 1.5mm] (2,0) .. controls (3.5,2.5) and (5.4,-0.9) .. (6,0);
\draw[->, shorten <= 1.5mm, shorten >= 1.5mm] (4,0) .. controls (5.5,2.5) and (7.4,-0.9) .. (8,0);
\draw[->, shorten <= 1.5mm, shorten >= 1.5mm] (6,0) .. controls (7.5,2.5) and (9.4,-0.9) .. (10,0);
\draw[->, shorten <= 1.5mm, shorten >= 1.5mm] (8,0) .. controls (9,1.1) and (9.4,-0.4) .. (10,0);
    
    \node[] at (0.7,-0.25) {\tiny\color{blue} $0$};
    \node[] at (2.6,-0.25) {\tiny\color{blue}  $1$};
    \node[] at (4.6,-0.25) {\tiny\color{blue}  $0$};
    \node[] at (6.6,-0.25) {\tiny\color{blue}  $1$};
    \node[] at (8.6,-0.25) {\tiny\color{blue}  $3$};
    \node[] at (1.0,0.6) {\tiny\color{blue}  $0$};
    \node[] at (1.4,1.25) {\tiny\color{blue}  $0$};
    \node[] at (3.4,1.25) {\tiny\color{blue}  $0$};
    \node[] at (5.4,1.25) {\tiny\color{blue}  $2$};
    \node[] at (7.4,1.25) {\tiny\color{blue}  $0$};
    \node[] at (9.1,0.65) {\tiny\color{blue}  $1$};
\end{tikzpicture} 

& &
  \begin{tikzpicture}[scale=0.5]
   \node[] at (-1.4,0.8) {$\phi_2$};
	\filldraw[black] (0,0) circle (3pt) node[label=below:{\tiny\color{orange} $1$}](a1) {};
	\filldraw[black] (2,0) circle (3pt) node[label=below:{\tiny\color{orange} $2$}](a2) {};
	\filldraw[black] (4,0) circle (3pt) node[label=below:{\tiny\color{orange} $3$}](a3) {};
	\filldraw[black] (6,0) circle (3pt) node[label=below:{\tiny\color{orange} $4$}](a4) {};
    \filldraw[black] (8,0) circle (3pt) node[label=below:{\tiny\color{orange} $5$}](a5) {};
    \filldraw[black] (10,0) circle (3pt) node[label=below:{\tiny\color{orange} $6$}](a6) {};

    \draw[->] (a1)--(a2);		
	\draw[->] (a2)--(a3);		
	\draw[->] (a3)--(a4);		
	\draw[->] (a4)--(a5);		
	\draw[->] (a5)--(a6);		

\draw[->, shorten <= 1.5mm, shorten >= 1.5mm] (0,0) .. controls (1,1.1) and (1.4,-0.6) .. (2,0);
\draw[->, shorten <= 1.5mm, shorten >= 1.5mm] (0,0) .. controls (1.5,2.5) and (3.4,-0.9) .. (4,0);
\draw[->, shorten <= 1.5mm, shorten >= 1.5mm] (2,0) .. controls (3.5,2.5) and (5.4,-0.9) .. (6,0);
\draw[->, shorten <= 1.5mm, shorten >= 1.5mm] (4,0) .. controls (5.5,2.5) and (7.4,-0.9) .. (8,0);
\draw[->, shorten <= 1.5mm, shorten >= 1.5mm] (6,0) .. controls (7.5,2.5) and (9.4,-0.9) .. (10,0);
\draw[->, shorten <= 1.5mm, shorten >= 1.5mm] (8,0) .. controls (9,1.1) and (9.4,-0.4) .. (10,0);
    
    \node[] at (0.7,-0.25) {\tiny\color{blue} $0$};
    \node[] at (2.6,-0.25) {\tiny\color{blue}  $1$};
    \node[] at (4.6,-0.25) {\tiny\color{blue}  $1$};
    \node[] at (6.6,-0.25) {\tiny\color{blue}  $1$};
    \node[] at (8.6,-0.25) {\tiny\color{blue}  $1$};
    \node[] at (1.0,0.6) {\tiny\color{blue}  $0$};
    \node[] at (1.4,1.25) {\tiny\color{blue}  $0$};
    \node[] at (3.4,1.25) {\tiny\color{blue}  $0$};
    \node[] at (5.4,1.25) {\tiny\color{blue}  $1$};
    \node[] at (7.4,1.25) {\tiny\color{blue}  $1$};
    \node[] at (9.1,0.65) {\tiny\color{blue}  $2$};
\end{tikzpicture}
\vspace{-.075in} \\
\begin{tikzpicture}[scale=0.5]
   \node[] at (-1.4,0.8) {$\phi_1'$};
	\filldraw[black] (0,0) circle (3pt) node[label=below:{\tiny\color{orange} $1$}](a1) {};
	\filldraw[black] (2,0) circle (3pt) node[label=below:{\tiny\color{orange} $2$}](a2) {};
	\filldraw[black] (4,0) circle (3pt) node[label=below:{\tiny\color{orange} $3$}](a3) {};
	\filldraw[black] (6,0) circle (3pt) node[label=below:{\tiny\color{orange} $4$}](a4) {};
    \filldraw[black] (8,0) circle (3pt) node[label=below:{\tiny\color{orange} $5$}](a5) {};
    \filldraw[black] (10,0) circle (3pt) node[label=below:{\tiny\color{orange} $6$}](a6) {};

    \draw[->] (a1)--(a2);		
	\draw[->] (a2)--(a3);		
	\draw[->] (a3)--(a4);		
	\draw[->] (a4)--(a5);		
	\draw[->] (a5)--(a6);		

\draw[->, shorten <= 1.5mm, shorten >= 1.5mm] (0,0) .. controls (1,1.1) and (1.4,-0.6) .. (2,0);
\draw[->, shorten <= 1.5mm, shorten >= 1.5mm] (0,0) .. controls (1.5,2.5) and (3.4,-0.9) .. (4,0);
\draw[->, shorten <= 1.5mm, shorten >= 1.5mm] (2,0) .. controls (3.5,2.5) and (5.4,-0.9) .. (6,0);
\draw[->, shorten <= 1.5mm, shorten >= 1.5mm] (4,0) .. controls (5.5,2.5) and (7.4,-0.9) .. (8,0);
\draw[->, shorten <= 1.5mm, shorten >= 1.5mm] (6,0) .. controls (7.5,2.5) and (9.4,-0.9) .. (10,0);
\draw[->, shorten <= 1.5mm, shorten >= 1.5mm] (8,0) .. controls (9,1.1) and (9.4,-0.4) .. (10,0);
    
    \node[] at (0.7,-0.25) {\tiny\color{blue} $0$};
    \node[] at (2.6,-0.25) {\tiny\color{blue}  $0$};
    \node[] at (4.6,-0.25) {\tiny\color{blue}  $0$};
    \node[] at (6.6,-0.25) {\tiny\color{blue}  $2$};
    \node[] at (8.6,-0.25) {\tiny\color{blue}  $3$};
    \node[] at (1.0,0.6) {\tiny\color{blue}  $0$};
    \node[] at (1.4,1.25) {\tiny\color{blue}  $0$};
    \node[] at (3.4,1.25) {\tiny\color{blue}  $1$};
    \node[] at (5.4,1.25) {\tiny\color{blue}  $1$};
    \node[] at (7.4,1.25) {\tiny\color{blue}  $0$};
    \node[] at (9.1,0.65) {\tiny\color{blue}  $1$};
\end{tikzpicture}
& &
\begin{tikzpicture}[scale=0.5]
\node[] at (-1.4,0.8) {$\phi_2'$};
	\filldraw[black] (0,0) circle (3pt) node[label=below:{\tiny\color{orange} $1$}](a1) {};
	\filldraw[black] (2,0) circle (3pt) node[label=below:{\tiny\color{orange} $2$}](a2) {};
	\filldraw[black] (4,0) circle (3pt) node[label=below:{\tiny\color{orange} $3$}](a3) {};
	\filldraw[black] (6,0) circle (3pt) node[label=below:{\tiny\color{orange} $4$}](a4) {};
    \filldraw[black] (8,0) circle (3pt) node[label=below:{\tiny\color{orange} $5$}](a5) {};
    \filldraw[black] (10,0) circle (3pt) node[label=below:{\tiny\color{orange} $6$}](a6) {};

    \draw[->] (a1)--(a2);		
	\draw[->] (a2)--(a3);		
	\draw[->] (a3)--(a4);		
	\draw[->] (a4)--(a5);		
	\draw[->] (a5)--(a6);		

\draw[->, shorten <= 1.5mm, shorten >= 1.5mm] (0,0) .. controls (1,1.1) and (1.4,-0.6) .. (2,0);
\draw[->, shorten <= 1.5mm, shorten >= 1.5mm] (0,0) .. controls (1.5,2.5) and (3.4,-0.9) .. (4,0);
\draw[->, shorten <= 1.5mm, shorten >= 1.5mm] (2,0) .. controls (3.5,2.5) and (5.4,-0.9) .. (6,0);
\draw[->, shorten <= 1.5mm, shorten >= 1.5mm] (4,0) .. controls (5.5,2.5) and (7.4,-0.9) .. (8,0);
\draw[->, shorten <= 1.5mm, shorten >= 1.5mm] (6,0) .. controls (7.5,2.5) and (9.4,-0.9) .. (10,0);
\draw[->, shorten <= 1.5mm, shorten >= 1.5mm] (8,0) .. controls (9,1.1) and (9.4,-0.4) .. (10,0);
    
    \node[] at (0.7,-0.25) {\tiny\color{blue} $0$};
    \node[] at (2.6,-0.25) {\tiny\color{blue}  $0$};
    \node[] at (4.6,-0.25) {\tiny\color{blue}  $1$};
    \node[] at (6.6,-0.25) {\tiny\color{blue}  $2$};
    \node[] at (8.6,-0.25) {\tiny\color{blue}  $1$};
    \node[] at (1.0,0.6) {\tiny\color{blue}  $0$};
    \node[] at (1.4,1.25) {\tiny\color{blue}  $0$};
    \node[] at (3.4,1.25) {\tiny\color{blue}  $1$};
    \node[] at (5.4,1.25) {\tiny\color{blue}  $0$};
    \node[] at (7.4,1.25) {\tiny\color{blue}  $1$};
    \node[] at (9.1,0.65) {\tiny\color{blue}  $2$};
\end{tikzpicture}
\end{tabular}~~

%% file: figures/fivegroves.tikz
$\raisebox{.11in}{\Big\{}\hspace{.1in}$%
\begin{tikzpicture}[scale=0.65]
   
	\draw[-,thick] (0,1)--(1,1);		
	\draw[-,thick] (0,1)--(1,0);	
    \draw[-,thick] (0,0)--(1,0); 

    \draw[-,thick] (1.8,1.1)--(2.8,1);
    \draw[-,thick] (1.8,1.1)--(2.8,0.25);
    \draw[-,thick] (1.8,0.5)--(2.8,0.25);		
	\draw[-,thick] (1.8,-0.1)--(2.8,0.25);	

\end{tikzpicture}
\hspace{.08in}
\raisebox{.05in}{,}
\hspace{.08in}
\begin{tikzpicture}[scale=0.65]
   
	\draw[-,thick] (0,1)--(1,1);		
	\draw[-,thick] (0,1)--(1,0);	
    \draw[-,thick] (0,0)--(1,0); 

    \draw[-,thick] (1.8,1.1)--(2.8,0.75);
    \draw[-,thick] (1.8,0.5)--(2.8,0.75);
    \draw[-,thick] (1.8,0.5)--(2.8,0.25);		
	\draw[-,thick] (1.8,-0.1)--(2.8,0.25);	

\end{tikzpicture}
\hspace{.08in}
\raisebox{.05in}{,}
\hspace{.08in}
 \begin{tikzpicture}[scale=0.65]
   
	\draw[-,thick] (0,1)--(1,1);		
	\draw[-,thick] (0,1)--(1,0);	
    \draw[-,thick] (0,0)--(1,0); 

    \draw[-,thick] (1.8,1.1)--(2.8,0.75);
    \draw[-,thick] (1.8,0.5)--(2.8,0.75);
    \draw[-,thick] (1.8,-0.1)--(2.8,0.75);		
	\draw[-,thick] (1.8,-0.1)--(2.8,0);	

\end{tikzpicture}
\hspace{.08in}
\raisebox{.05in}{,}
\hspace{.08in}
\raisebox{0.02in}{\begin{tikzpicture}[scale=0.65]
   
	\draw[-,thick] (0,1)--(1,1);		
	\draw[-,thick] (0,0)--(1,1);	
    \draw[-,thick] (0,0)--(1,0); 

    \draw[-,thick] (1.8,1)--(2.8,1);
    \draw[-,thick] (1.8,1)--(2.8,0);
    \draw[-,thick] (1.8,0)--(2.8,0);		

\end{tikzpicture}}
\hspace{.08in}
\raisebox{.05in}{,}
\hspace{.08in}
\raisebox{0.02in}{\begin{tikzpicture}[scale=0.65]
   
	\draw[-,thick] (0,1)--(1,1);		
	\draw[-,thick] (0,0)--(1,1);	
    \draw[-,thick] (0,0)--(1,0); 

    \draw[-,thick] (1.8,1)--(2.8,1);
    \draw[-,thick] (1.8,0)--(2.8,1);
    \draw[-,thick] (1.8,0)--(2.8,0);		

\end{tikzpicture}}
$\hspace{.1in}\raisebox{.11in}{\Big\}}$%

%% file: figures/groveflowbijection.tikz
%
%
\raisebox{0.in}{\begin{tikzpicture}[scale=1]
	\filldraw[black] (4,0) circle (3pt) node[](a3) {};
    \node[] at (4,-0.65) {\scriptsize $i$};
    \node[] at (2,-1) {\scriptsize $e_1$};
    \node[] at (2,0) {\scriptsize $e_2$};
    \node[] at (2,1) {\scriptsize $e_3$};
    \node[] at (6,-1) {\scriptsize $e_1'$};
    \node[] at (6,-0.333) {\scriptsize $e_2'$};
    \node[] at (6,0.333) {\scriptsize $e_3'$};
    \node[] at (6,1) {\scriptsize $e_4'$};
	\draw[->] (2.3,0)--(a3);		
	\draw[->] (2.3,1) to[out=0,in=120] (a3);	
	\draw[->] (2.3,-1) to[out=0,in=240] (a3);	
	\draw[->]    (a3) to[out=60,in=180] (5.7,1);
	\draw[->]    (a3) to[out=20,in=180] (5.7,0.333);
	\draw[->]    (a3) to[out=-20,in=180] (5.7,-0.333);
	\draw[->]    (a3) to[out=-60,in=180] (5.7,-1);
    \node[] at (3,1.15) {\color{red} \scriptsize $2$};
    \node[] at (3,0.25) {\color{red} \scriptsize $3$};
    \node[] at (3,-0.65) {\color{red} \scriptsize $1$};
    \node[] at (5,1.15) {\color{red} \scriptsize $1$};
    \node[] at (5,0.5) {\color{red} \scriptsize $4$};
    \node[] at (5,-0.07) {\color{red} \scriptsize $0$};
    \node[] at (5,-0.7) {\color{red} \scriptsize $3$};

    \node[] at (4,-2.3) {\scriptsize Flow on edges of $G$ };
    \node[] at (4,-2.65) {\scriptsize incident with vertex $i$};
\end{tikzpicture}
\qquad\qquad
\begin{tikzpicture}[scale=2]
    \draw[-,thick] (0,2)--(1,1.75);
    \draw[-,thick] (0,1.75)--(1,1.75);
    \draw[-,thick] (0,1.75)--(1,1.25);
    \draw[-,thick] (0,1.5)--(1,1.25);
    \draw[-,thick] (0,1.25)--(1,1.25);	    
    \draw[-,thick] (0,1)--(1,1.25);		
    \draw[-,thick] (0,0.75)--(1,1.25);		
    \draw[-,thick] (0,0.75)--(1,0.75); 
    \draw[-,thick] (0,0.75)--(1,0.25); 
    \draw[-,thick] (0,0.5)--(1,0.25); 
    \draw[-,thick] (0,0.25)--(1,0.25); 
	\draw[-,thick] (0,0.0)--(1,0.25);	
    \draw[decorate,decoration={brace,amplitude=6pt},xshift=-.05in]
(0,1.45) -- (0,2.05) node [black,midway,xshift=-0.6in] {\begin{tabular}{c}\footnotesize${\color{red}\phi(}e_3{\color{red})}+1$\vspace{-.03in}\\\footnotesize left vertices\end{tabular}};
    \draw[decorate,decoration={brace,amplitude=6pt},xshift=-.05in]
(0,.45) -- (0,1.3) node [black,midway,xshift=-0.6in] {\footnotesize \begin{tabular}{c}${\color{red}\phi(}e_2{\color{red})}+1$\\left vertices\end{tabular}};
    \draw[decorate,decoration={brace,amplitude=6pt},xshift=-.05in]
(0,-.05) -- (0,0.30) node [black,midway,xshift=-0.6in] {\footnotesize \begin{tabular}{c}${\color{red}\phi(}e_1{\color{red})}+1$\\left vertices\end{tabular}};
    \draw[decorate,decoration={brace,amplitude=6pt},xshift=.05in]
(1,1.95) -- (1,1.55) node [black,midway,xshift=0.5in] {\scriptsize \begin{tabular}{c}degree of $r_4$\\is ${\color{red}\phi(}e_4'{\color{red})}+1$\end{tabular} };
    \draw[decorate,decoration={brace,amplitude=6pt},xshift=.05in]
(1,1.45) -- (1,1.05) node [black,midway,xshift=0.5in] {\scriptsize \begin{tabular}{c}degree of $r_3$\\is ${\color{red}\phi(}e_3'{\color{red})}+1$\end{tabular} };
    \draw[decorate,decoration={brace,amplitude=6pt},xshift=.05in]
(1,0.95) -- (1,0.55) node [black,midway,xshift=0.5in] {\scriptsize \begin{tabular}{c}degree of $r_2$\\is ${\color{red}\phi(}e_2'{\color{red})}+1$\end{tabular} };
    \draw[decorate,decoration={brace,amplitude=6pt},xshift=.05in]
(1,0.45) -- (1,0.05) node [black,midway,xshift=0.5in] {\scriptsize \begin{tabular}{c} degree of $r_1$\\ is ${\color{red}\phi(}e_1'{\color{red})}+1$\end{tabular} };
    \node[] at (0.5,-0.35) {\scriptsize Tree $\gamma_i$ in grove $\Gamma$};
    \end{tikzpicture}
}

%% file: figures/grovedynamics.tikz
\vspace{-.05in}
\begin{tikzpicture}[scale=0.55]
	\filldraw[black] (0,0) circle (3pt) node[label=below:{\tiny\color{orange} $1$}](a1) {};
	\filldraw[black] (2,0) circle (3pt) node[label=below:{\tiny\color{orange} $2$}](a2) {};
	\filldraw[black] (4,0) circle (3pt) node[label=below:{\tiny\color{orange} $3$}](a3) {};
	\filldraw[black] (6,0) circle (3pt) node[label=below:{\tiny\color{orange} $4$}](a4) {};
    \filldraw[black] (8,0) circle (3pt) node[label=below:{\tiny\color{orange} $5$}](a5) {};
    \filldraw[black] (10,0) circle (3pt) node[label=below:{\tiny\color{orange} $6$}](a6) {};

    \draw[->] (a1)--(a2);		
	\draw[->] (a2)--(a3);		
	\draw[->] (a3)--(a4);		
	\draw[->] (a4)--(a5);		
	\draw[->] (a5)--(a6);		

\draw[->, shorten <= 1.5mm, shorten >= 1.5mm] (0,0) .. controls (1,1.1) and (1.4,-0.6) .. (2,0);
\draw[->, shorten <= 1.5mm, shorten >= 1.5mm] (0,0) .. controls (1.5,2.5) and (3.4,-0.9) .. (4,0);
\draw[->, shorten <= 1.5mm, shorten >= 1.5mm] (2,0) .. controls (3.5,2.5) and (5.4,-0.9) .. (6,0);
\draw[->, shorten <= 1.5mm, shorten >= 1.5mm] (4,0) .. controls (5.5,2.5) and (7.4,-0.9) .. (8,0);
\draw[->, shorten <= 1.5mm, shorten >= 1.5mm] (6,0) .. controls (7.5,2.5) and (9.4,-0.9) .. (10,0);
\draw[->, shorten <= 1.5mm, shorten >= 1.5mm] (8,0) .. controls (9,1.1) and (9.4,-0.4) .. (10,0);
    
    \node[] at (0.7,-0.25) {\tiny\color{blue} $0$};
    \node[] at (2.6,-0.25) {\tiny\color{blue}  $1$};
    \node[] at (4.6,-0.25) {\tiny\color{blue}  $0$};
    \node[] at (6.6,-0.25) {\tiny\color{blue}  $0$};
    \node[] at (8.6,-0.25) {\tiny\color{blue}  $1$};
    \node[] at (1.0,0.6) {\tiny\color{blue}  $0$};
    \node[] at (1.4,1.25) {\tiny\color{blue}  $0$};
    \node[] at (3.4,1.25) {\tiny\color{blue}  $0$};
    \node[] at (5.4,1.25) {\tiny\color{blue}  $2$};
    \node[] at (7.4,1.25) {\tiny\color{blue}  $1$};
    \node[] at (9.1,0.65) {\tiny\color{blue}  $2$};
\end{tikzpicture}  \vspace{.1in}

\begin{tikzpicture}[scale=1]

    \draw[-,thin,black!25] plot [smooth] coordinates { (-1.5,0.5) (-1.2,0.67) (0,1.5) (1.2,1.4) (2.0,0.1) (2.3,-0.1)};

    \draw[-,thin,black!25] plot [smooth] coordinates { (-1.5,0.5) (-1.2,0.57) (-0.5,0.95) (-0.2,1)};

    \draw[-,thin,black!25] plot [smooth] coordinates { (-1.5,0.5) (-1.2,0.43) (-0.5,0.05) (-0.2,0)};
    
\node at (-1.5,0.5) [circle,fill,white,inner sep=4pt]{};
\node at (-1.5,0.5) [circle,fill,black!25,inner sep=1pt]{};

    \draw[-,thin] (-0.1,0.85)--(-0.2,0.85)--(-0.2,1.15)--(-0.1,1.15);
    
    \draw[-,thin] (-0.1,-0.15)--(-0.2,-0.15)--(-0.2,0.15)--(-0.1,0.15);

	\draw[-,very thick] (0,1)--(1,1);		
	\draw[-,very thick,blue,dash pattern=on 3pt off 1.5pt] (0,1)--(1,0);	
    \draw[-,very thick] (0,0)--(1,0); 
    \node[] at (0.5,-0.8) {$\gamma_2$};

    \draw[-,thin] (1.1,1.15)--(1.2,1.15)--(1.2,0.85)--(1.1,0.85);

    \draw[-,thin] (1.1,0.25)--(1.2,0.25)--(1.2,-0.25)--(1.1,-0.25);

    \draw[-,thin] (2.4,0.4)--(2.3,0.4)--(2.3,1.2)--(2.4,1.2);

    \draw[-,thin] plot [smooth] coordinates { (1.2,0) (1.5,0.1) (2.0,0.7) (2.3,0.8)};

    \draw[-,thin] (2.4,-0.25)--(2.3,-0.25)--(2.3,0.05)--(2.4,0.05);

    \draw[-,very thick,blue,dash pattern=on 3pt off 1.5pt] (2.5,1.1)--(3.5,0.8);
    \draw[-,very thick] (2.5,0.5)--(3.5,0.8);
    \draw[-,very thick] (2.5,-0.1)--(3.5,0.8);		
	\draw[-,very thick] (2.5,-0.1)--(3.5,0.2);	
\node at (2.53,1.1) [circle,fill,blue,inner sep=1.5pt]{};
\node at (2.53,1.1) [circle,fill,white,inner sep=0.85pt]{};
    \node[] at (3,-0.8) {$\gamma_3$};

    \draw[-,thin] (3.6,1.1)--(3.7,1.1)--(3.7,0.5)--(3.6,0.5);

    \draw[-,thin] (3.6,0.35)--(3.7,0.35)--(3.7,0.05)--(3.6,0.05);

    \draw[-,thin] (4.9,0.85)--(4.8,0.85)--(4.8,1.15)--(4.9,1.15);

    \draw[-,thin] plot [smooth] coordinates { (3.7,0.2) (4.0,0.3) (4.5,0.9) (4.8,1)};

    \draw[-,thin] (4.9,-0.15)--(4.8,-0.15)--(4.8,0.15)--(4.9,0.15);

    \draw[-,thin] plot [smooth] coordinates { (1.2,1) (1.5,1.1) (2.5,1.5) (3.7,1.4) (4.5,0.2) (4.8,0)};

	\draw[-,very thick] (5,1)--(6,1);		
	\draw[-,very thick] (5,0)--(6,1);	
    \draw[-,very thick] (5,0)--(6,0); 
    \node[] at (5.5,-0.8) {$\gamma_4$};

    \draw[-,thin] (6.1,1.15)--(6.2,1.15)--(6.2,0.7)--(6.1,0.7);

    \draw[-,thin] (6.1,0.15)--(6.2,0.15)--(6.2,-0.15)--(6.1,-0.15);

    \draw[-,thin] (7.4,0.95)--(7.3,0.95)--(7.3,1.25)--(7.4,1.25);

    \draw[-,thin] plot [smooth] coordinates { (6.2,0) (6.5,0.2) (7.0,0.9) (7.3,1.1)};

    \draw[-,thin] (7.4,0.8)--(7.3,0.8)--(7.3,-0.2)--(7.4,-0.2);

    \draw[-,thin] plot [smooth] coordinates { (3.7,0.8) (4,0.9) (5,1.5) (6.2,1.4) (7,0.4) (7.3,0.3)};

	\draw[-,very thick] (7.5,1.1)--(8.5,0.75);		
	\draw[-,very thick,blue,dash pattern=on 3pt off 1.5pt] (7.5,0.7)--(8.5,0.75);	
    \draw[-,very thick] (7.5,0.3)--(8.5,0.75); 
    \draw[-,very thick] (7.5,0.3)--(8.5,0.25); 
    \draw[-,very thick] (7.5,-0.1)--(8.5,0.25); 
\node at (7.53,0.7) [circle,fill,blue,inner sep=1.5pt]{};
\node at (7.53,0.7) [circle,fill,white,inner sep=0.85pt]{};
    \node[] at (8,-0.8) {$\gamma_5$};

    \draw[-,thin] (8.6,1.1)--(8.7,1.1)--(8.7,0.5)--(8.6,0.5);
    
    \draw[-,thin] (8.6,-0.05)--(8.7,-0.05)--(8.7,0.4)--(8.6,0.4);

    \draw[-,thin,black!25] plot [smooth] coordinates { (8.7,0.175) (9,0.3) (9.5,0.7) (9.7,0.75)  (10,0.5)};

    \draw[-,thin,black!25] plot [smooth] coordinates { (8.7,0.8) (9,0.75) (9.7,0.5) (10,0.5)};
    
    \draw[-,thin,black!25] plot [smooth] coordinates { (6.2,0.925) (6.5,1) (7.5,1.5) (8.7,1.4) (9.5,0.3) (10,0.5)};

\node at (10,0.5) [circle,fill,white,inner sep=4pt]{};
\node at (10,0.5) [circle,fill,black!25,inner sep=1pt]{};

\end{tikzpicture}

%% file: figures/elementarymove.tikz
\begin{tikzpicture}[scale=0.8]
    \node[] at (0,0.5) {\scriptsize Grove $\Gamma$};

    \draw[-,thick] (2,0)--(3,0)--(2,1)--(3,1);		

    \node[] at (2.5,-0.5) {\scriptsize $\gamma_{2}$};

    \node[] at (3.75,0.5) {\scriptsize $\cdots$};
   
	\draw[-,thick] (4.5,1)--(5.5,0.75);		
	\draw[-,thick] (4.5,1.25)--(5.5,0.75);	
    \draw[-,thick] (4.5,0.5)--(5.5,0.75); 
	\draw[-,dotted,red] (4.5,0)--(5.5,0.75);	
	\draw[-,very thick,blue,dash pattern=on 3pt off 1.5pt] (4.5,0.5)--(5.5,0.25);	
	\draw[-,thick] (4.5,0)--(5.5,0.25);		
	\draw[-,thick] (4.5,-0.25)--(5.5,0.25);		

   \node[] at (5,-0.5) {\scriptsize $\gamma_{i}$};

  \draw[-,very thick, blue,dash pattern=on 3pt off 1.5pt](6.5,1.5)--(7.5,0.75);
 \draw[-,thick](6.5,1.25)--(7.5,0.75);
 \draw[-,thick] (6.5,1)--(7.5,0.75);		
	\draw[-,thick] (6.5,0.75)--(7.5,0.75);	
    \draw[-,thick] (6.5,0.5)--(7.5,0.75); 
	\draw[-,thick] (6.5,0.5)--(7.5,0.25);	
	\draw[-,thick] (6.5,0.2)--(7.5,0.25);		
	\draw[-,thick] (6.5,-0.2)--(7.5,0.25);

    \node[] at (7,-0.5) {\scriptsize $\gamma_{i+1}$};

 \draw[-,thick](8.5,1.5)--(9.5,0.75);
 \draw[-,thick] (8.5,1)--(9.5,0.75);		
	\draw[-,thick] (8.5,1.25)--(9.5,0.75);	
    \draw[-,thick] (8.5,0.5)--(9.5,0.75); 
	\draw[-,thick] (8.5,0.5)--(9.5,0.25);	
	\draw[-,thick] (8.5,0.25)--(9.5,0.25);		
	\draw[-,thick] (8.5,0)--(9.5,0.25);	
    \draw[-,dotted,red](8.5,-0.25)--(9.5,0.25);

    \node[] at (9,-0.5) {\scriptsize $\gamma_{i+2}$};

 \draw[-,thick] (10.5, 1.5 )--(11.5,0.75);
 \draw[-,thick] (10.5, 1.25 )--(11.5,0.75);
 \draw[-,thick] (10.5, 1.0 )--(11.5,0.75);	
 \draw[blue,dash pattern= on 3.5pt off 4.5pt,very thick] (10.5, 0.75 )--(11.5,0.75);
    \draw[red,dash pattern= on 3.5pt off 4.5pt,dash phase=4pt,very thick] (10.5, 0.75 )--(11.5,0.75);
 \draw[-,thick] (10.5, 0.5 )--(11.5,0.75); 
 \draw[-,thick] (10.5, 0.5 )--(11.5,0.75); 
 \draw[-,thick] (10.5, 0.5)--(11.5,0.25);	
 \draw[-,thick] (10.5, 0.25)--(11.5,0.25);		
 \draw[-,thick] (10.5,0   )--(11.5,0.25);	
 \draw[-,thick] (10.5,-0.25   )--(11.5,0.25);	

    \node[] at (11,-0.5) {\scriptsize $\gamma_{i+3}$};

    \node[] at (12.25,0.5) {\scriptsize $\cdots$};

    \draw[-,thick](13,1.5)--(14,0.75);
 \draw[blue,dash pattern= on 3.5pt off 4.5pt,very thick] (13,1.25)--(14,0.75);	
    \draw[red,dash pattern= on 3.5pt off 4.5pt,dash phase=4pt,very thick] (13,1.25)--(14,0.75);	    \draw[-,thick] (13,1)--(14,0.75);		
    \draw[-,thick] (13,0.75)--(14,0.75); 
    \draw[-,thick] (13,0.5)--(14,0.75); 
    \draw[-,thick] (13,0.25)--(14,0.75); 
	\draw[-,thick] (13,0.0)--(14,0.75);	
	\draw[-,thick] (13,0.0)--(14,0.25);		
    \draw[-,thick] (13,-0.25)--(14,0.25);

    \node[] at (13.5,-0.5) {\scriptsize $\gamma_{n-1}$};

\end{tikzpicture}

\vspace{.175in}

\begin{tikzpicture}[scale=0.8]
    \node[] at (0,0.5) {\scriptsize Grove $\Gamma'$};
    
    \draw[-,thick] (2,0)--(3,0)--(2,1)--(3,1);		

    \node[] at (2.5,-0.5) {\scriptsize $\gamma_{2}'$};

    \node[] at (3.75,0.5) {\scriptsize $\cdots$};

	\draw[-,thick] (4.5,1)--(5.5,0.75);		
	\draw[-,thick] (4.5,1.25)--(5.5,0.75);	
    \draw[-,thick] (4.5,0.5)--(5.5,0.75); 
	\draw[-,dotted,blue] (4.5,0.5)--(5.5,0.25);	
	\draw[-,very thick,red,dash pattern=on 3pt off 1.5pt] (4.5,0)--(5.5,0.75);	
	\draw[-,thick] (4.5,0)--(5.5,0.25);		
	\draw[-,thick] (4.5,-0.25)--(5.5,0.25);		

    \node[] at (5,-0.5) {\scriptsize $\gamma_{i}'$};

  \draw[-, dotted, blue](6.5,1.5)--(7.5,0.75);
 \draw[-,thick](6.5,1.25)--(7.5,0.75);
 \draw[-,thick] (6.5,1)--(7.5,0.75);		
	\draw[-,thick] (6.5,0.75)--(7.5,0.75);	
    \draw[-,thick] (6.5,0.5)--(7.5,0.75); 
	\draw[-,thick] (6.5,0.5)--(7.5,0.25);	
	\draw[-,thick] (6.5,0.2)--(7.5,0.25);		
	\draw[-,thick] (6.5,-0.2)--(7.5,0.25);		

    \node[] at (7,-0.5) {\scriptsize $\gamma_{i+1}'$};

 \draw[-,thick](8.5,1.5)--(9.5,0.75);
 \draw[-,thick] (8.5,1)--(9.5,0.75);		
	\draw[-,thick] (8.5,1.25)--(9.5,0.75);	
    \draw[-,thick] (8.5,0.5)--(9.5,0.75); 
	\draw[-,thick] (8.5,0.5)--(9.5,0.25);	
	\draw[-,thick] (8.5,0.25)--(9.5,0.25);		
	\draw[-,thick] (8.5,0)--(9.5,0.25);	
 \draw[-,very thick,red,dash pattern=on 3pt off 1.5pt](8.5,-0.25)--(9.5,0.25);

    \node[] at (9,-0.5) {\scriptsize $\gamma_{i+2}'$};

 \draw[-,thick] (10.5, 1.5 )--(11.5,0.75);
 \draw[-,thick] (10.5, 1.25 )--(11.5,0.75);
 \draw[-,thick] (10.5, 1.0 )--(11.5,0.75);	
 \draw[blue,dash pattern= on 3.5pt off 4.5pt,very thick] (10.5, 0.75 )--(11.5,0.75);
    \draw[red,dash pattern= on 3.5pt off 4.5pt,dash phase=4pt,very thick] (10.5, 0.75 )--(11.5,0.75);
 \draw[-,thick] (10.5, 0.5 )--(11.5,0.75); 
 \draw[-,thick] (10.5, 0.5 )--(11.5,0.75); 
 \draw[-,thick] (10.5, 0.5)--(11.5,0.25);	
 \draw[-,thick] (10.5, 0.25)--(11.5,0.25);		
 \draw[-,thick] (10.5,0   )--(11.5,0.25);	
 \draw[-,thick] (10.5,-0.25   )--(11.5,0.25);	

    \node[] at (11,-0.5) {\scriptsize $\gamma_{i+3}'$};

    \node[] at (12.25,0.5) {\scriptsize $\cdots$};

    \draw[-,thick](13,1.5)--(14,0.75);
 \draw[blue,dash pattern= on 3.5pt off 4.5pt,very thick] (13,1.25)--(14,0.75);	
    \draw[red,dash pattern= on 3.5pt off 4.5pt,dash phase=4pt,very thick] (13,1.25)--(14,0.75);	    \draw[-,thick] (13,1)--(14,0.75);		
    \draw[-,thick] (13,0.75)--(14,0.75); 
    \draw[-,thick] (13,0.5)--(14,0.75); 
    \draw[-,thick] (13,0.25)--(14,0.75); 
	\draw[-,thick] (13,0.0)--(14,0.75);	
	\draw[-,thick] (13,0.0)--(14,0.25);		
    \draw[-,thick] (13,-0.25)--(14,0.25);

    \node[] at (13.5,-0.5) {\scriptsize $\gamma_{n-1}'$};

\end{tikzpicture}	

%% file: figures/grovedynamics2.tikz
\vspace{-.05in}
\begin{tikzpicture}[scale=0.55]
	\filldraw[black] (0,0) circle (3pt) node[label=below:{\tiny\color{orange} $1$}](a1) {};
	\filldraw[black] (2,0) circle (3pt) node[label=below:{\tiny\color{orange} $2$}](a2) {};
	\filldraw[black] (4,0) circle (3pt) node[label=below:{\tiny\color{orange} $3$}](a3) {};
	\filldraw[black] (6,0) circle (3pt) node[label=below:{\tiny\color{orange} $4$}](a4) {};
    \filldraw[black] (8,0) circle (3pt) node[label=below:{\tiny\color{orange} $5$}](a5) {};
    \filldraw[black] (10,0) circle (3pt) node[label=below:{\tiny\color{orange} $6$}](a6) {};

    \draw[->] (a1)--(a2);		
	\draw[->] (a2)--(a3);		
	\draw[->] (a3)--(a4);		
	\draw[->] (a4)--(a5);		
	\draw[->] (a5)--(a6);		

\draw[->, shorten <= 1.5mm, shorten >= 1.5mm] (0,0) .. controls (1,1.1) and (1.4,-0.6) .. (2,0);
\draw[->, shorten <= 1.5mm, shorten >= 1.5mm] (0,0) .. controls (1.5,2.5) and (3.4,-0.9) .. (4,0);
\draw[->, shorten <= 1.5mm, shorten >= 1.5mm] (2,0) .. controls (3.5,2.5) and (5.4,-0.9) .. (6,0);
\draw[->, shorten <= 1.5mm, shorten >= 1.5mm] (4,0) .. controls (5.5,2.5) and (7.4,-0.9) .. (8,0);
\draw[->, shorten <= 1.5mm, shorten >= 1.5mm] (6,0) .. controls (7.5,2.5) and (9.4,-0.9) .. (10,0);
\draw[->, shorten <= 1.5mm, shorten >= 1.5mm] (8,0) .. controls (9,1.1) and (9.4,-0.4) .. (10,0);
    
    \node[] at (0.7,-0.25) {\tiny\color{blue} $0$};
    \node[] at (2.6,-0.25) {\tiny\color{red} $0$};
    \node[] at (4.6,-0.25) {\tiny\color{blue}  $0$};
    \node[] at (6.6,-0.25) {\tiny\color{red} $1$};
    \node[] at (8.6,-0.25) {\tiny\color{blue}  $1$};
    \node[] at (1.0,0.6) {\tiny\color{blue}  $0$};
    \node[] at (1.4,1.25) {\tiny\color{blue}  $0$};
    \node[] at (3.4,1.25) {\tiny\color{red} $1$};
    \node[] at (5.4,1.25) {\tiny\color{red} $1$};
    \node[] at (7.4,1.25) {\tiny\color{blue}  $1$};
    \node[] at (9.1,0.65) {\tiny\color{blue}  $2$};
\end{tikzpicture}  \vspace{.1in}

\begin{tikzpicture}[scale=1]

    \draw[-,thin,black!25] plot [smooth] coordinates { (-1.5,0.5) (-1.2,0.67) (0,1.5) (1.2,1.4) (2.0,0.2) (2.3,0)};

    \draw[-,thin,black!25] plot [smooth] coordinates { (-1.5,0.5) (-1.2,0.57) (-0.5,0.95) (-0.2,1)};

    \draw[-,thin,black!25] plot [smooth] coordinates { (-1.5,0.5) (-1.2,0.43) (-0.5,0.05) (-0.2,0)};
    
\node at (-1.5,0.5) [circle,fill,white,inner sep=4pt]{};
\node at (-1.5,0.5) [circle,fill,black!25,inner sep=1pt]{};

    \draw[-,thin] (-0.1,0.85)--(-0.2,0.85)--(-0.2,1.15)--(-0.1,1.15);
    
    \draw[-,thin] (-0.1,-0.15)--(-0.2,-0.15)--(-0.2,0.15)--(-0.1,0.15);

	\draw[-,very thick] (0,1)--(1,1);		
	\draw[-,very thick,red,dash pattern=on 3pt off 1.5pt] (0,0)--(1,1);	
    \draw[-,very thick] (0,0)--(1,0); 
    \node[] at (0.5,-0.8) {$\gamma_2'$};

    \draw[-,thin] (1.1,1.15)--(1.2,1.15)--(1.2,0.65)--(1.1,0.65);

    \draw[-,thin] (1.1,0.15)--(1.2,0.15)--(1.2,-0.15)--(1.1,-0.15);

    \draw[-,thin] (2.4,0.85)--(2.3,0.85)--(2.3,1.15)--(2.4,1.15);

    \draw[-,thin] plot [smooth] coordinates { (1.2,0) (1.5,0.1) (2.0,0.9) (2.3,1.0)};

    \draw[-,thin] (2.4,-0.15)--(2.3,-0.15)--(2.3,0.15)--(2.4,0.15);

	\draw[-,very thick] (2.5,1)--(3.5,1);		
	\draw[-,very thick] (2.5,0)--(3.5,1);	
    \draw[-,very thick] (2.5,0)--(3.5,0); 
    \node[] at (3,-0.8) {$\gamma_3'$};

    \draw[-,thin] (3.6,1.15)--(3.7,1.15)--(3.7,0.65)--(3.6,0.65);

    \draw[-,thin] (3.6,0.15)--(3.7,0.15)--(3.7,-0.15)--(3.6,-0.15);

    \draw[-,thin] (4.9,0.95)--(4.8,0.95)--(4.8,1.25)--(4.9,1.25);

    \draw[-,thin] plot [smooth] coordinates { (3.7,0) (4.0,0.1) (4.5,0.95) (4.8,1.1)};

    \draw[-,thin] (4.9,-0.2)--(4.8,-0.2)--(4.8,0.6)--(4.9,0.6);

    \draw[-,thin] plot [smooth] coordinates { (1.2,0.9) (1.5,1.0) (2.5,1.5) (3.7,1.4) (4.5,0.4) (4.8,0.2)};

	\draw[-,very thick] (5,1.1)--(6,0.8);		
	\draw[-,very thick] (5,0.5)--(6,0.8);	
    \draw[-,very thick] (5,0.5)--(6,0.2); 
   \draw[-,very thick,red,dash pattern=on 3pt off 1.5pt] (5,-0.1)--(6,0.2);
\node at (5.03,-0.1) [circle,fill,red,inner sep=1.5pt]{};
\node at (5.03,-0.1) [circle,fill,white,inner sep=0.85pt]{};
    \node[] at (5.5,-0.8) {$\gamma_4'$};

    \draw[-,thin] (6.1,1.0)--(6.2,1.0)--(6.2,0.6)--(6.1,0.6);

    \draw[-,thin] (6.1,0.0)--(6.2,0.0)--(6.2,0.4)--(6.1,0.4);

    \draw[-,thin] (7.4,0.6)--(7.3,0.6)--(7.3,1.2)--(7.4,1.2);

    \draw[-,thin] plot [smooth] coordinates { (6.2,0.2) (6.5,0.3) (7.0,0.8) (7.3,0.9)};

    \draw[-,thin] (7.4,0.4)--(7.3,0.4)--(7.3,-0.2)--(7.4,-0.2);

    \draw[-,thin] plot [smooth] coordinates { (3.7,0.9) (4,1.0) (5,1.5) (6.2,1.4) (7,0.3) (7.3,0.1)};

	\draw[-,very thick] (7.5,1.1)--(8.5,0.75);		
	\draw[-,very thick,red,dash pattern=on 3pt off 1.5pt] (7.5,0.7)--(8.5,0.75);	
    \draw[-,very thick] (7.5,0.3)--(8.5,0.75); 
    \draw[-,very thick] (7.5,0.3)--(8.5,0.25); 
    \draw[-,very thick] (7.5,-0.1)--(8.5,0.25); 
\node at (7.53,0.7) [circle,fill,red,inner sep=1.5pt]{};
\node at (7.53,0.7) [circle,fill,white,inner sep=0.85pt]{};
    \node[] at (8,-0.8) {$\gamma_5'$};

    \draw[-,thin] (8.6,1.1)--(8.7,1.1)--(8.7,0.5)--(8.6,0.5);
    
    \draw[-,thin] (8.6,-0.05)--(8.7,-0.05)--(8.7,0.4)--(8.6,0.4);

    \draw[-,thin,black!25] plot [smooth] coordinates { (8.7,0.175) (9,0.3) (9.5,0.7) (9.7,0.75)  (10,0.5)};

    \draw[-,thin,black!25] plot [smooth] coordinates { (8.7,0.8) (9,0.75) (9.7,0.5) (10,0.5)};
    
    \draw[-,thin,black!25] plot [smooth] coordinates { (6.2,0.8) (6.5,0.9) (7.5,1.5) (8.7,1.4) (9.5,0.3) (10,0.5)};

\node at (10,0.5) [circle,fill,white,inner sep=4pt]{};
\node at (10,0.5) [circle,fill,black!25,inner sep=1pt]{};

\end{tikzpicture}

%% file: figures/statistics.tex
\begin{tabular}{|c c c c c c c c c|}
 \hline
 Label & Alt.\ Perm. & Inv.\ Perm. & Flow & Circ.\ Perm. & $\mathsf{swap}$ & $\mathsf{sz}$ & $\mathsf{zs}$ & $\mathsf{des}$ \\ 
 \hline
 \hline
 {\triangled{$\mathsf{A}$}} &
 34251 & 53124 & 
 \raisebox{-.14in}{\begin{tikzpicture}[scale=0.45]
 \clip (-0.2,-0.6) rectangle (8.2, 1.65);
	\filldraw[black] (0,0) circle (4pt) node(a1) {};
	\filldraw[black] (2,0) circle (4pt) node(a2) {};
	\filldraw[black] (4,0) circle (4pt) node(a3) {};
	\filldraw[black] (6,0) circle (4pt) node(a4) {};
    \filldraw[black] (8,0) circle (4pt) node(a5) {};
    
    \draw[->] (a1)--(a2);		
	\draw[->] (a2)--(a3);		
	\draw[->] (a3)--(a4);		
	\draw[->] (a4)--(a5);		

\draw[->, shorten <= 1.5mm, shorten >= 1.5mm] (0,0) .. controls (1,1.1) and (1.4,-0.6) .. (2,0);
\draw[->, shorten <= 1.5mm, shorten >= 1.5mm] (0,0) .. controls (1.5,2.5) and (3.4,-0.9) .. (4,0);
\draw[->, shorten <= 1.5mm, shorten >= 1.5mm] (2,0) .. controls (3.5,2.5) and (5.4,-0.9) .. (6,0);
\draw[->, shorten <= 1.5mm, shorten >= 1.5mm] (4,0) .. controls (5.5,2.5) and (7.4,-0.9) .. (8,0);
\draw[->, shorten <= 1.5mm, shorten >= 1.5mm] (6,0) .. controls (7,1.1) and (7.4,-0.4) .. (8,0);

    \node[] at (0.7,-0.25) {\tiny $0$};
    \node[] at (2.6,-0.25) {\tiny $1$};
    \node[] at (4.6,-0.25) {\tiny $2$};
    \node[] at (6.6,-0.25) {\tiny $3$};
    \node[] at (1.0,0.6) {\tiny $0$};
    \node[] at (1.4,1.25) {\tiny $0$};
    \node[] at (3.4,1.25) {\tiny $0$};
    \node[] at (5.4,1.25) {\tiny $0$};
    \node[] at (7.1,0.65) {\tiny $0$};
        
\end{tikzpicture}}
& 12345 & 1 & 1 & 0 & 0 \\ 
\hline
 {\triangled{$\mathsf{B}$}} &
34152 & 35124 & \raisebox{-.14in}{\begin{tikzpicture}[scale=0.45]
\clip (-0.2,-0.6) rectangle (8.2, 1.65);

\filldraw[black] (0,0) circle (4pt) node (a1) {};
\filldraw[black] (2,0) circle (4pt) node (a2) {};
\filldraw[black] (4,0) circle (4pt) node (a3) {};
\filldraw[black] (6,0) circle (4pt) node (a4) {};
\filldraw[black] (8,0) circle (4pt) node (a5) {};

\draw[->] (a1)--(a2);
\draw[->] (a2)--(a3);
\draw[->] (a3)--(a4);
\draw[->] (a4)--(a5);

\draw[->, shorten <= 1.5mm, shorten >= 1.5mm]
    (0,0) .. controls (1,1.1) and (1.4,-0.6) .. (2,0);

\draw[->, shorten <= 1.5mm, shorten >= 1.5mm]
    (0,0) .. controls (1.5,2.5) and (3.4,-0.9) .. (4,0);

\draw[->, shorten <= 1.5mm, shorten >= 1.5mm]
    (2,0) .. controls (3.5,2.5) and (5.4,-0.9) .. (6,0);

\draw[->, shorten <= 1.5mm, shorten >= 1.5mm]
    (4,0) .. controls (5.5,2.5) and (7.4,-0.9) .. (8,0);

\draw[->, shorten <= 1.5mm, shorten >= 1.5mm]
    (6,0) .. controls (7,1.1) and (7.4,-0.4) .. (8,0);

     \node[] at (0.7,-0.25) {\tiny $0$};
     \node[] at (2.6,-0.25) {\tiny $1$};
     \node[] at (4.6,-0.25) {\tiny $2$};
     \node[] at (6.6,-0.25) {\tiny $2$};
     \node[] at (1.0,0.6) {\tiny $0$};
     \node[] at (1.4,1.25) {\tiny $0$};
     \node[] at (3.4,1.25) {\tiny $0$};
     \node[] at (5.4,1.25) {\tiny $0$};
     \node[] at (7.1,0.65) {\tiny $1$};



\end{tikzpicture}} & 51234 & 1 & 2 & 1 & 1 \\
 \hline
 {\triangled{$\mathsf{C}$}} &
24153 & 31524 & \raisebox{-.14in}{\begin{tikzpicture}[scale=0.45]
\clip (-0.2,-0.6) rectangle (8.2, 1.65);

\filldraw[black] (0,0) circle (4pt) node (a1) {};
\filldraw[black] (2,0) circle (4pt) node (a2) {};
\filldraw[black] (4,0) circle (4pt) node (a3) {};
\filldraw[black] (6,0) circle (4pt) node (a4) {};
\filldraw[black] (8,0) circle (4pt) node (a5) {};

\draw[->] (a1)--(a2);
\draw[->] (a2)--(a3);
\draw[->] (a3)--(a4);
\draw[->] (a4)--(a5);

\draw[->, shorten <= 1.5mm, shorten >= 1.5mm]
    (0,0) .. controls (1,1.1) and (1.4,-0.6) .. (2,0);

\draw[->, shorten <= 1.5mm, shorten >= 1.5mm]
    (0,0) .. controls (1.5,2.5) and (3.4,-0.9) .. (4,0);

\draw[->, shorten <= 1.5mm, shorten >= 1.5mm]
    (2,0) .. controls (3.5,2.5) and (5.4,-0.9) .. (6,0);

\draw[->, shorten <= 1.5mm, shorten >= 1.5mm]
    (4,0) .. controls (5.5,2.5) and (7.4,-0.9) .. (8,0);

\draw[->, shorten <= 1.5mm, shorten >= 1.5mm]
    (6,0) .. controls (7,1.1) and (7.4,-0.4) .. (8,0);

\node[] at (0.7,-0.25) {\tiny $0$};
     \node[] at (2.6,-0.25) {\tiny $1$};
     \node[] at (4.6,-0.25) {\tiny $2$};
     \node[] at (6.6,-0.25) {\tiny $1$};
     \node[] at (1.0,0.6) {\tiny $0$};
     \node[] at (1.4,1.25) {\tiny $0$};
     \node[] at (3.4,1.25) {\tiny $0$};
     \node[] at (5.4,1.25) {\tiny $0$};
     \node[] at (7.1,0.65) {\tiny $2$};



\end{tikzpicture}} & 15234 & 3 & 2 & 1 & 1 \\
\hline
 {\triangled{$\mathsf{D}$}} &
23154 & 31254 & \raisebox{-.14in}{\begin{tikzpicture}[scale=0.45]
\clip (-0.2,-0.6) rectangle (8.2, 1.65);

\filldraw[black] (0,0) circle (4pt) node (a1) {};
\filldraw[black] (2,0) circle (4pt) node (a2) {};
\filldraw[black] (4,0) circle (4pt) node (a3) {};
\filldraw[black] (6,0) circle (4pt) node (a4) {};
\filldraw[black] (8,0) circle (4pt) node (a5) {};

\draw[->] (a1)--(a2);
\draw[->] (a2)--(a3);
\draw[->] (a3)--(a4);
\draw[->] (a4)--(a5);

\draw[->, shorten <= 1.5mm, shorten >= 1.5mm]
    (0,0) .. controls (1,1.1) and (1.4,-0.6) .. (2,0);
\draw[->, shorten <= 1.5mm, shorten >= 1.5mm]
    (0,0) .. controls (1.5,2.5) and (3.4,-0.9) .. (4,0);
\draw[->, shorten <= 1.5mm, shorten >= 1.5mm]
    (2,0) .. controls (3.5,2.5) and (5.4,-0.9) .. (6,0);
\draw[->, shorten <= 1.5mm, shorten >= 1.5mm]
    (4,0) .. controls (5.5,2.5) and (7.4,-0.9) .. (8,0);
\draw[->, shorten <= 1.5mm, shorten >= 1.5mm]
    (6,0) .. controls (7,1.1) and (7.4,-0.4) .. (8,0);

\node[] at (0.7,-0.25) {\tiny $0$};
     \node[] at (2.6,-0.25) {\tiny $1$};
     \node[] at (4.6,-0.25) {\tiny $2$};
     \node[] at (6.6,-0.25) {\tiny $0$};
     \node[] at (1.0,0.6) {\tiny $0$};
     \node[] at (1.4,1.25) {\tiny $0$};
     \node[] at (3.4,1.25) {\tiny $0$};
     \node[] at (5.4,1.25) {\tiny $0$};
     \node[] at (7.1,0.65) {\tiny $3$};


\end{tikzpicture}} & 12534 & 2 & 1 & 1 & 1\\
\hline
 {\triangled{$\mathsf{E}$}} &
35241 & 53142 & \raisebox{-.14in}{\begin{tikzpicture}[scale=0.45]
\clip (-0.2,-0.6) rectangle (8.2, 1.65);

\filldraw[black] (0,0) circle (4pt) node (a1) {};
\filldraw[black] (2,0) circle (4pt) node (a2) {};
\filldraw[black] (4,0) circle (4pt) node (a3) {};
\filldraw[black] (6,0) circle (4pt) node (a4) {};
\filldraw[black] (8,0) circle (4pt) node (a5) {};

\draw[->] (a1)--(a2);
\draw[->] (a2)--(a3);
\draw[->] (a3)--(a4);
\draw[->] (a4)--(a5);

\draw[->, shorten <= 1.5mm, shorten >= 1.5mm]
    (0,0) .. controls (1,1.1) and (1.4,-0.6) .. (2,0);
\draw[->, shorten <= 1.5mm, shorten >= 1.5mm]
    (0,0) .. controls (1.5,2.5) and (3.4,-0.9) .. (4,0);
\draw[->, shorten <= 1.5mm, shorten >= 1.5mm]
    (2,0) .. controls (3.5,2.5) and (5.4,-0.9) .. (6,0);
\draw[->, shorten <= 1.5mm, shorten >= 1.5mm]
    (4,0) .. controls (5.5,2.5) and (7.4,-0.9) .. (8,0);
\draw[->, shorten <= 1.5mm, shorten >= 1.5mm]
    (6,0) .. controls (7,1.1) and (7.4,-0.4) .. (8,0);

\node[] at (0.7,-0.25) {\tiny $0$};
     \node[] at (2.6,-0.25) {\tiny $1$};
     \node[] at (4.6,-0.25) {\tiny $1$};
     \node[] at (6.6,-0.25) {\tiny $2$};
     \node[] at (1.0,0.6) {\tiny $0$};
     \node[] at (1.4,1.25) {\tiny $0$};
     \node[] at (3.4,1.25) {\tiny $0$};
     \node[] at (5.4,1.25) {\tiny $1$};
     \node[] at (7.1,0.65) {\tiny $0$};


\end{tikzpicture}} & 45123 & 1 & 2 & 1 & 1  \\
\hline
 {\triangled{$\mathsf{F}$}} &
35142 & 35142 & \raisebox{-.14in}{\begin{tikzpicture}[scale=0.45]
\clip (-0.2,-0.6) rectangle (8.2, 1.65);

\filldraw[black] (0,0) circle (4pt) node (a1) {};
\filldraw[black] (2,0) circle (4pt) node (a2) {};
\filldraw[black] (4,0) circle (4pt) node (a3) {};
\filldraw[black] (6,0) circle (4pt) node (a4) {};
\filldraw[black] (8,0) circle (4pt) node (a5) {};

\draw[->] (a1)--(a2);
\draw[->] (a2)--(a3);
\draw[->] (a3)--(a4);
\draw[->] (a4)--(a5);

\draw[->, shorten <= 1.5mm, shorten >= 1.5mm]
    (0,0) .. controls (1,1.1) and (1.4,-0.6) .. (2,0);
\draw[->, shorten <= 1.5mm, shorten >= 1.5mm]
    (0,0) .. controls (1.5,2.5) and (3.4,-0.9) .. (4,0);
\draw[->, shorten <= 1.5mm, shorten >= 1.5mm]
    (2,0) .. controls (3.5,2.5) and (5.4,-0.9) .. (6,0);
\draw[->, shorten <= 1.5mm, shorten >= 1.5mm]
    (4,0) .. controls (5.5,2.5) and (7.4,-0.9) .. (8,0);
\draw[->, shorten <= 1.5mm, shorten >= 1.5mm]
    (6,0) .. controls (7,1.1) and (7.4,-0.4) .. (8,0);

\node[] at (0.7,-0.25) {\tiny $0$};
     \node[] at (2.6,-0.25) {\tiny $1$};
     \node[] at (4.6,-0.25) {\tiny $1$};
     \node[] at (6.6,-0.25) {\tiny $1$};
     \node[] at (1.0,0.6) {\tiny $0$};
     \node[] at (1.4,1.25) {\tiny $0$};
     \node[] at (3.4,1.25) {\tiny $0$};
     \node[] at (5.4,1.25) {\tiny $1$};
     \node[] at (7.1,0.65) {\tiny $1$};


\end{tikzpicture}} & 41523 & 2 & 3 & 2 & 2 \\ 
 \hline
 {\triangled{$\mathsf{G}$}} &
24153 & 31524 & \raisebox{-.14in}{\begin{tikzpicture}[scale=0.45]
\clip (-0.2,-0.6) rectangle (8.2, 1.65);

\filldraw[black] (0,0) circle (4pt) node (a1) {};
\filldraw[black] (2,0) circle (4pt) node (a2) {};
\filldraw[black] (4,0) circle (4pt) node (a3) {};
\filldraw[black] (6,0) circle (4pt) node (a4) {};
\filldraw[black] (8,0) circle (4pt) node (a5) {};

\draw[->] (a1)--(a2);
\draw[->] (a2)--(a3);
\draw[->] (a3)--(a4);
\draw[->] (a4)--(a5);

\draw[->, shorten <= 1.5mm, shorten >= 1.5mm]
    (0,0) .. controls (1,1.1) and (1.4,-0.6) .. (2,0);
\draw[->, shorten <= 1.5mm, shorten >= 1.5mm]
    (0,0) .. controls (1.5,2.5) and (3.4,-0.9) .. (4,0);
\draw[->, shorten <= 1.5mm, shorten >= 1.5mm]
    (2,0) .. controls (3.5,2.5) and (5.4,-0.9) .. (6,0);
\draw[->, shorten <= 1.5mm, shorten >= 1.5mm]
    (4,0) .. controls (5.5,2.5) and (7.4,-0.9) .. (8,0);
\draw[->, shorten <= 1.5mm, shorten >= 1.5mm]
    (6,0) .. controls (7,1.1) and (7.4,-0.4) .. (8,0);

\node[] at (0.7,-0.25) {\tiny $0$};
     \node[] at (2.6,-0.25) {\tiny $1$};
     \node[] at (4.6,-0.25) {\tiny $1$};
     \node[] at (6.6,-0.25) {\tiny $0$};
     \node[] at (1.0,0.6) {\tiny $0$};
     \node[] at (1.4,1.25) {\tiny $0$};
     \node[] at (3.4,1.25) {\tiny $0$};
     \node[] at (5.4,1.25) {\tiny $1$};
     \node[] at (7.1,0.65) {\tiny $2$};


\end{tikzpicture}} & 41253 & 2 & 2 & 2 & 2 \\
\hline
\rowcolor{black!5}
 {\triangled{$\mathsf{H}$}} &
45231 & 53412 & \raisebox{-.14in}{\begin{tikzpicture}[scale=0.45]
\clip (-0.2,-0.6) rectangle (8.2, 1.65);

\filldraw[black] (0,0) circle (4pt) node (a1) {};
\filldraw[black] (2,0) circle (4pt) node (a2) {};
\filldraw[black] (4,0) circle (4pt) node (a3) {};
\filldraw[black] (6,0) circle (4pt) node (a4) {};
\filldraw[black] (8,0) circle (4pt) node (a5) {};

\draw[->] (a1)--(a2);
\draw[->] (a2)--(a3);
\draw[->] (a3)--(a4);
\draw[->] (a4)--(a5);

\draw[->, shorten <= 1.5mm, shorten >= 1.5mm]
    (0,0) .. controls (1,1.1) and (1.4,-0.6) .. (2,0);
\draw[->, shorten <= 1.5mm, shorten >= 1.5mm]
    (0,0) .. controls (1.5,2.5) and (3.4,-0.9) .. (4,0);
\draw[->, shorten <= 1.5mm, shorten >= 1.5mm]
    (2,0) .. controls (3.5,2.5) and (5.4,-0.9) .. (6,0);
\draw[->, shorten <= 1.5mm, shorten >= 1.5mm]
    (4,0) .. controls (5.5,2.5) and (7.4,-0.9) .. (8,0);
\draw[->, shorten <= 1.5mm, shorten >= 1.5mm]
    (6,0) .. controls (7,1.1) and (7.4,-0.4) .. (8,0);

\node[] at (0.7,-0.25) {\tiny $0$};
     \node[] at (2.6,-0.25) {\tiny $1$};
     \node[] at (4.6,-0.25) {\tiny $0$};
     \node[] at (6.6,-0.25) {\tiny $1$};
     \node[] at (1.0,0.6) {\tiny $0$};
     \node[] at (1.4,1.25) {\tiny $0$};
     \node[] at (3.4,1.25) {\tiny $0$};
     \node[] at (5.4,1.25) {\tiny $2$};
     \node[] at (7.1,0.65) {\tiny $0$};


\end{tikzpicture}}  & 14523 & 0 & 2 & 1 & 1  \\ 
\hline
 {\triangled{$\mathsf{I}$}} &
45132 & 35412 & \raisebox{-.14in}{\begin{tikzpicture}[scale=0.45]
\clip (-0.2,-0.6) rectangle (8.2, 1.65);

\filldraw[black] (0,0) circle (4pt) node (a1) {};
\filldraw[black] (2,0) circle (4pt) node (a2) {};
\filldraw[black] (4,0) circle (4pt) node (a3) {};
\filldraw[black] (6,0) circle (4pt) node (a4) {};
\filldraw[black] (8,0) circle (4pt) node (a5) {};

\draw[->] (a1)--(a2);
\draw[->] (a2)--(a3);
\draw[->] (a3)--(a4);
\draw[->] (a4)--(a5);

\draw[->, shorten <= 1.5mm, shorten >= 1.5mm]
    (0,0) .. controls (1,1.1) and (1.4,-0.6) .. (2,0);
\draw[->, shorten <= 1.5mm, shorten >= 1.5mm]
    (0,0) .. controls (1.5,2.5) and (3.4,-0.9) .. (4,0);
\draw[->, shorten <= 1.5mm, shorten >= 1.5mm]
    (2,0) .. controls (3.5,2.5) and (5.4,-0.9) .. (6,0);
\draw[->, shorten <= 1.5mm, shorten >= 1.5mm]
    (4,0) .. controls (5.5,2.5) and (7.4,-0.9) .. (8,0);
\draw[->, shorten <= 1.5mm, shorten >= 1.5mm]
    (6,0) .. controls (7,1.1) and (7.4,-0.4) .. (8,0);

\node[] at (0.7,-0.25) {\tiny $0$};
     \node[] at (2.6,-0.25) {\tiny $1$};
     \node[] at (4.6,-0.25) {\tiny $0$};
     \node[] at (6.6,-0.25) {\tiny $0$};
     \node[] at (1.0,0.6) {\tiny $0$};
     \node[] at (1.4,1.25) {\tiny $0$};
     \node[] at (3.4,1.25) {\tiny $0$};
     \node[] at (5.4,1.25) {\tiny $2$};
     \node[] at (7.1,0.65) {\tiny $1$};


\end{tikzpicture}} & 14253 & 1 & 1 & 2 & 2 \\
\hline
 {\triangled{$\mathsf{J}$}} &
24351 & 51324 & \raisebox{-.14in}{\begin{tikzpicture}[scale=0.45]
\clip (-0.2,-0.6) rectangle (8.2, 1.65);

\filldraw[black] (0,0) circle (4pt) node (a1) {};
\filldraw[black] (2,0) circle (4pt) node (a2) {};
\filldraw[black] (4,0) circle (4pt) node (a3) {};
\filldraw[black] (6,0) circle (4pt) node (a4) {};
\filldraw[black] (8,0) circle (4pt) node (a5) {};

\draw[->] (a1)--(a2);
\draw[->] (a2)--(a3);
\draw[->] (a3)--(a4);
\draw[->] (a4)--(a5);

\draw[->, shorten <= 1.5mm, shorten >= 1.5mm]
    (0,0) .. controls (1,1.1) and (1.4,-0.6) .. (2,0);
\draw[->, shorten <= 1.5mm, shorten >= 1.5mm]
    (0,0) .. controls (1.5,2.5) and (3.4,-0.9) .. (4,0);
\draw[->, shorten <= 1.5mm, shorten >= 1.5mm]
    (2,0) .. controls (3.5,2.5) and (5.4,-0.9) .. (6,0);
\draw[->, shorten <= 1.5mm, shorten >= 1.5mm]
    (4,0) .. controls (5.5,2.5) and (7.4,-0.9) .. (8,0);
\draw[->, shorten <= 1.5mm, shorten >= 1.5mm]
    (6,0) .. controls (7,1.1) and (7.4,-0.4) .. (8,0);

\node[] at (0.7,-0.25) {\tiny $0$};
     \node[] at (2.6,-0.25) {\tiny $0$};
     \node[] at (4.6,-0.25) {\tiny $1$};
     \node[] at (6.6,-0.25) {\tiny $3$};
     \node[] at (1.0,0.6) {\tiny $0$};
     \node[] at (1.4,1.25) {\tiny $0$};
     \node[] at (3.4,1.25) {\tiny $1$};
     \node[] at (5.4,1.25) {\tiny $0$};
     \node[] at (7.1,0.65) {\tiny $0$};


\end{tikzpicture}}  & 34512 & 2 & 1 & 1 & 1 \\
 
\hline
 {\triangled{$\mathsf{K}$}} &
14352 & 15324 & \raisebox{-.14in}{\begin{tikzpicture}[scale=0.45]
\clip (-0.2,-0.6) rectangle (8.2, 1.65);

\filldraw[black] (0,0) circle (4pt) node (a1) {};
\filldraw[black] (2,0) circle (4pt) node (a2) {};
\filldraw[black] (4,0) circle (4pt) node (a3) {};
\filldraw[black] (6,0) circle (4pt) node (a4) {};
\filldraw[black] (8,0) circle (4pt) node (a5) {};

\draw[->] (a1)--(a2);
\draw[->] (a2)--(a3);
\draw[->] (a3)--(a4);
\draw[->] (a4)--(a5);

\draw[->, shorten <= 1.5mm, shorten >= 1.5mm]
    (0,0) .. controls (1,1.1) and (1.4,-0.6) .. (2,0);
\draw[->, shorten <= 1.5mm, shorten >= 1.5mm]
    (0,0) .. controls (1.5,2.5) and (3.4,-0.9) .. (4,0);
\draw[->, shorten <= 1.5mm, shorten >= 1.5mm]
    (2,0) .. controls (3.5,2.5) and (5.4,-0.9) .. (6,0);
\draw[->, shorten <= 1.5mm, shorten >= 1.5mm]
    (4,0) .. controls (5.5,2.5) and (7.4,-0.9) .. (8,0);
\draw[->, shorten <= 1.5mm, shorten >= 1.5mm]
    (6,0) .. controls (7,1.1) and (7.4,-0.4) .. (8,0);

\node[] at (0.7,-0.25) {\tiny $0$};
     \node[] at (2.6,-0.25) {\tiny $0$};
     \node[] at (4.6,-0.25) {\tiny $1$};
     \node[] at (6.6,-0.25) {\tiny $2$};
     \node[] at (1.0,0.6) {\tiny $0$};
     \node[] at (1.4,1.25) {\tiny $0$};
     \node[] at (3.4,1.25) {\tiny $1$};
     \node[] at (5.4,1.25) {\tiny $0$};
     \node[] at (7.1,0.65) {\tiny $1$};


\end{tikzpicture}} & 34152 & 2 & 2 & 2 & 2 \\
\hline
 {\triangled{$\mathsf{L}$}} &
14253 & 13524 & \raisebox{-.14in}{\begin{tikzpicture}[scale=0.45]
\clip (-0.2,-0.6) rectangle (8.2, 1.65);

\filldraw[black] (0,0) circle (4pt) node (a1) {};
\filldraw[black] (2,0) circle (4pt) node (a2) {};
\filldraw[black] (4,0) circle (4pt) node (a3) {};
\filldraw[black] (6,0) circle (4pt) node (a4) {};
\filldraw[black] (8,0) circle (4pt) node (a5) {};

\draw[->] (a1)--(a2);
\draw[->] (a2)--(a3);
\draw[->] (a3)--(a4);
\draw[->] (a4)--(a5);

\draw[->, shorten <= 1.5mm, shorten >= 1.5mm]
    (0,0) .. controls (1,1.1) and (1.4,-0.6) .. (2,0);
\draw[->, shorten <= 1.5mm, shorten >= 1.5mm]
    (0,0) .. controls (1.5,2.5) and (3.4,-0.9) .. (4,0);
\draw[->, shorten <= 1.5mm, shorten >= 1.5mm]
    (2,0) .. controls (3.5,2.5) and (5.4,-0.9) .. (6,0);
\draw[->, shorten <= 1.5mm, shorten >= 1.5mm]
    (4,0) .. controls (5.5,2.5) and (7.4,-0.9) .. (8,0);
\draw[->, shorten <= 1.5mm, shorten >= 1.5mm]
    (6,0) .. controls (7,1.1) and (7.4,-0.4) .. (8,0);

\node[] at (0.7,-0.25) {\tiny $0$};
     \node[] at (2.6,-0.25) {\tiny $0$};
     \node[] at (4.6,-0.25) {\tiny $1$};
     \node[] at (6.6,-0.25) {\tiny $1$};
     \node[] at (1.0,0.6) {\tiny $0$};
     \node[] at (1.4,1.25) {\tiny $0$};
     \node[] at (3.4,1.25) {\tiny $1$};
     \node[] at (5.4,1.25) {\tiny $0$};
     \node[] at (7.1,0.65) {\tiny $2$};


\end{tikzpicture}}
& 34125 & 2 & 2 & 2 & 1 \\ 
\hline
 {\triangled{$\mathsf{M}$}} &
13254 & 13254 & \raisebox{-.14in}{\begin{tikzpicture}[scale=0.45]
\clip (-0.2,-0.6) rectangle (8.2, 1.65);

\filldraw[black] (0,0) circle (4pt) node (a1) {};
\filldraw[black] (2,0) circle (4pt) node (a2) {};
\filldraw[black] (4,0) circle (4pt) node (a3) {};
\filldraw[black] (6,0) circle (4pt) node (a4) {};
\filldraw[black] (8,0) circle (4pt) node (a5) {};

\draw[->] (a1)--(a2);
\draw[->] (a2)--(a3);
\draw[->] (a3)--(a4);
\draw[->] (a4)--(a5);

\draw[->, shorten <= 1.5mm, shorten >= 1.5mm]
    (0,0) .. controls (1,1.1) and (1.4,-0.6) .. (2,0);
\draw[->, shorten <= 1.5mm, shorten >= 1.5mm]
    (0,0) .. controls (1.5,2.5) and (3.4,-0.9) .. (4,0);
\draw[->, shorten <= 1.5mm, shorten >= 1.5mm]
    (2,0) .. controls (3.5,2.5) and (5.4,-0.9) .. (6,0);
\draw[->, shorten <= 1.5mm, shorten >= 1.5mm]
    (4,0) .. controls (5.5,2.5) and (7.4,-0.9) .. (8,0);
\draw[->, shorten <= 1.5mm, shorten >= 1.5mm]
    (6,0) .. controls (7,1.1) and (7.4,-0.4) .. (8,0);

\node[] at (0.7,-0.25) {\tiny $0$};
     \node[] at (2.6,-0.25) {\tiny $0$};
     \node[] at (4.6,-0.25) {\tiny $1$};
     \node[] at (6.6,-0.25) {\tiny $0$};
     \node[] at (1.0,0.6) {\tiny $0$};
     \node[] at (1.4,1.25) {\tiny $0$};
     \node[] at (3.4,1.25) {\tiny $1$};
     \node[] at (5.4,1.25) {\tiny $0$};
     \node[] at (7.1,0.65) {\tiny $3$};


\end{tikzpicture}} & 53412 & 2 & 1 & 1 & 2 \\
 
  \hline
 {\triangled{$\mathsf{N}$}} &
24351 & 51324 & \raisebox{-.14in}{\begin{tikzpicture}[scale=0.45]
\clip (-0.2,-0.6) rectangle (8.2, 1.65);

\filldraw[black] (0,0) circle (4pt) node (a1) {};
\filldraw[black] (2,0) circle (4pt) node (a2) {};
\filldraw[black] (4,0) circle (4pt) node (a3) {};
\filldraw[black] (6,0) circle (4pt) node (a4) {};
\filldraw[black] (8,0) circle (4pt) node (a5) {};

\draw[->] (a1)--(a2);
\draw[->] (a2)--(a3);
\draw[->] (a3)--(a4);
\draw[->] (a4)--(a5);

\draw[->, shorten <= 1.5mm, shorten >= 1.5mm]
    (0,0) .. controls (1,1.1) and (1.4,-0.6) .. (2,0);
\draw[->, shorten <= 1.5mm, shorten >= 1.5mm]
    (0,0) .. controls (1.5,2.5) and (3.4,-0.9) .. (4,0);
\draw[->, shorten <= 1.5mm, shorten >= 1.5mm]
    (2,0) .. controls (3.5,2.5) and (5.4,-0.9) .. (6,0);
\draw[->, shorten <= 1.5mm, shorten >= 1.5mm]
    (4,0) .. controls (5.5,2.5) and (7.4,-0.9) .. (8,0);
\draw[->, shorten <= 1.5mm, shorten >= 1.5mm]
    (6,0) .. controls (7,1.1) and (7.4,-0.4) .. (8,0);

\node[] at (0.7,-0.25) {\tiny $0$};
     \node[] at (2.6,-0.25) {\tiny $0$};
     \node[] at (4.6,-0.25) {\tiny $0$};
     \node[] at (6.6,-0.25) {\tiny $2$};
     \node[] at (1.0,0.6) {\tiny $0$};
     \node[] at (1.4,1.25) {\tiny $0$};
     \node[] at (3.4,1.25) {\tiny $1$};
     \node[] at (5.4,1.25) {\tiny $1$};
     \node[] at (7.1,0.65) {\tiny $0$};


\end{tikzpicture}} & 31452 & 1 & 1 & 2 & 2 \\

 \hline
 {\triangled{$\mathsf{O}$}} &
15342 & 15342 & \raisebox{-.14in}{\begin{tikzpicture}[scale=0.45]
\clip (-0.2,-0.6) rectangle (8.2, 1.65);

\filldraw[black] (0,0) circle (4pt) node (a1) {};
\filldraw[black] (2,0) circle (4pt) node (a2) {};
\filldraw[black] (4,0) circle (4pt) node (a3) {};
\filldraw[black] (6,0) circle (4pt) node (a4) {};
\filldraw[black] (8,0) circle (4pt) node (a5) {};

\draw[->] (a1)--(a2);
\draw[->] (a2)--(a3);
\draw[->] (a3)--(a4);
\draw[->] (a4)--(a5);

\draw[->, shorten <= 1.5mm, shorten >= 1.5mm]
    (0,0) .. controls (1,1.1) and (1.4,-0.6) .. (2,0);
\draw[->, shorten <= 1.5mm, shorten >= 1.5mm]
    (0,0) .. controls (1.5,2.5) and (3.4,-0.9) .. (4,0);
\draw[->, shorten <= 1.5mm, shorten >= 1.5mm]
    (2,0) .. controls (3.5,2.5) and (5.4,-0.9) .. (6,0);
\draw[->, shorten <= 1.5mm, shorten >= 1.5mm]
    (4,0) .. controls (5.5,2.5) and (7.4,-0.9) .. (8,0);
\draw[->, shorten <= 1.5mm, shorten >= 1.5mm]
    (6,0) .. controls (7,1.1) and (7.4,-0.4) .. (8,0);

\node[] at (0.7,-0.25) {\tiny $0$};
     \node[] at (2.6,-0.25) {\tiny $0$};
     \node[] at (4.6,-0.25) {\tiny $0$};
     \node[] at (6.6,-0.25) {\tiny $1$};
     \node[] at (1.0,0.6) {\tiny $0$};
     \node[] at (1.4,1.25) {\tiny $0$};
     \node[] at (3.4,1.25) {\tiny $1$};
     \node[] at (5.4,1.25) {\tiny $1$};
     \node[] at (7.1,0.65) {\tiny $1$};


\end{tikzpicture}} & 31425 & 1 & 1 & 3 & 2 \\ 
\hline
 {\triangled{$\mathsf{P}$}} &
15243 & 13542 & \raisebox{-.14in}{\begin{tikzpicture}[scale=0.45]
\clip (-0.2,-0.6) rectangle (8.2, 1.65);

\filldraw[black] (0,0) circle (4pt) node (a1) {};
\filldraw[black] (2,0) circle (4pt) node (a2) {};
\filldraw[black] (4,0) circle (4pt) node (a3) {};
\filldraw[black] (6,0) circle (4pt) node (a4) {};
\filldraw[black] (8,0) circle (4pt) node (a5) {};

\draw[->] (a1)--(a2);
\draw[->] (a2)--(a3);
\draw[->] (a3)--(a4);
\draw[->] (a4)--(a5);

\draw[->, shorten <= 1.5mm, shorten >= 1.5mm]
    (0,0) .. controls (1,1.1) and (1.4,-0.6) .. (2,0);
\draw[->, shorten <= 1.5mm, shorten >= 1.5mm]
    (0,0) .. controls (1.5,2.5) and (3.4,-0.9) .. (4,0);
\draw[->, shorten <= 1.5mm, shorten >= 1.5mm]
    (2,0) .. controls (3.5,2.5) and (5.4,-0.9) .. (6,0);
\draw[->, shorten <= 1.5mm, shorten >= 1.5mm]
    (4,0) .. controls (5.5,2.5) and (7.4,-0.9) .. (8,0);
\draw[->, shorten <= 1.5mm, shorten >= 1.5mm]
    (6,0) .. controls (7,1.1) and (7.4,-0.4) .. (8,0);

\node[] at (0.7,-0.25) {\tiny $0$};
     \node[] at (2.6,-0.25) {\tiny $0$};
     \node[] at (4.6,-0.25) {\tiny $0$};
     \node[] at (6.6,-0.25) {\tiny $0$};
     \node[] at (1.0,0.6) {\tiny $0$};
     \node[] at (1.4,1.25) {\tiny $0$};
     \node[] at (3.4,1.25) {\tiny $1$};
     \node[] at (5.4,1.25) {\tiny $1$};
     \node[] at (7.1,0.65) {\tiny $2$};


\end{tikzpicture}} & 53142 & 1 & 0 & 2 & 3 \\
 \hline
\end{tabular}

%% file: figures/maximalclique.tikz
\begin{tikzpicture}[scale=0.9]

\definecolor{cbBlue}{HTML}{0072B2}
\definecolor{cbVermillion}{HTML}{D55E00}
\definecolor{cbGreen}{HTML}{009E73}
\definecolor{cbPurple}{HTML}{CC79A7}
\definecolor{cbOrange}{HTML}{E69F00}
\definecolor{cbSkyBlue}{HTML}{56B4E9}

\tikzset{
    edge/.style={
        ->,
        line width=1.5pt,
    },
    pathD/.style={edge, gray},
    pathE/.style={edge, cbVermillion},
    pathF/.style={edge, cbOrange},
    pathG/.style={edge, cbGreen},
    pathH/.style={edge, cbSkyBlue},
    pathI/.style={edge,cbPurple}
}
7,-0.25) {\tiny $x_1$};

\begin{scope}[yshift = -300]

\coordinate (a1) at (0,0);
\coordinate (a2) at (2,0);
\coordinate (a3) at (4,0);
\coordinate (a4) at (6,0);
\coordinate (a5) at (8,0);

\draw[->,thin,black!25] (a1) to [bend left=0] (a2);
\draw[->,thin,black!25] (a2) to[bend left=0] (a3);
\draw[->,thin,black!25] (a3) to[bend left=0] (a4);
\draw[->,thin,black!25] (a4) to[bend left=0] (a5);

\draw[->,thin,black!25] (a1) to[out=55,in=125,looseness=0.9] (a3);
\draw[->,thin,black!25] (a3) to[out=55,in=125,looseness=0.9] (a5);

\draw[->,thin,black!25] (a1) to[bend right=50] (a2);
\draw[->,thin,black!25] (a2) to[out=-55,in=-125,looseness=0.9] (a4);
\draw[->,thin,black!25] (a4) to[bend right=50] (a5);


\draw[pathD,-] 
    ($(a1)+(0,-0.2)$) to[bend right=50] ($(a2)+(0,-0.2)$)
    to[out=-55,in=-125,looseness=0.9] ($(a4)+(0,-0.2)$)
    to[bend right=50] ($(a5)+(0,-0.2)$);

\draw[pathE,-] 
    ($(a1)+(0,-0.13)$) to [bend right=50] ($(a2)+(0,-0.13)$)
    to [out=-55, in =-125,looseness=0.9] ($(a4)+(0,-0.13)$)
    to [bend right=0] ($(a5)+(0,-0.13)$);

\draw[pathF,-] 
    ($(a1)+(0,-0.06)$) to [bend right=50] ($(a2)+(0,-0.06)$)
    to [bend right=0] ($(a4)+(0,-0.06)$)
    to [bend right=0] ($(a5)+(0,-0.06)$);

\draw[pathG,-] 
    ($(a1)+(0,0.01)$) to [bend right=50] ($(a2)+(0,0.01)$)
    to [bend right=0] ($(a3)+(0,0.01)$)
    to [out=55, in =125,looseness=0.9] ($(a5)+(0,0.01)$);

\draw[pathH,-] 
    ($(a1)+(0,.08)$) to [bend right=0] ($(a3)+(0,.08)$)
    to [out=55, in =125,looseness=0.9] ($(a5)+(0, .08)$);

\draw[pathI,-] 
    ($(a1)+(0,0.15)$) to[out=55,in=125,looseness=0.9] ($(a3)+(0,0.15)$)
    to[out=55,in=125,looseness=0.9] ($(a5)+(0,0.15)$);


\filldraw[black] (a1) circle (3pt);
\filldraw[black] (a2) circle (3pt);
\filldraw[black] (a3) circle (3pt);
\filldraw[black] (a4) circle (3pt);
\filldraw[black] (a5) circle (3pt);

\node[] at (1,0.25) {\tiny $x_1$};
\node[] at (3,0.25) {\tiny $x_2$};
\node[] at (5,0.25) {\tiny $x_3$};
\node[] at (7,0.25) {\tiny $x_4$};

\node[] at (1,-0.85) {\tiny $y_0$};
\node[] at (1,-0.2) {\tiny\color{cbSkyBlue} $1$};

\node[] at (2,1.25) {\tiny $y_1$};
\node[] at (2,0.55) {\tiny\color{cbPurple} $2$};

\node[] at (4,-1.25) {\tiny $y_2$};
\node[] at (4,-0.45) {\tiny\color{cbOrange} $3$};

\node[] at (6,1.25) {\tiny $y_3$};
\node[] at (6,0.50) {\tiny\color{cbGreen} $4$};

\node[] at (7,-0.85) {\tiny $y_4$};
\node[] at (7,-0.3) {\tiny\color{cbVermillion} $5$};

\node[left, gray] at (-0.2,-0.75) {\tiny $P_0$};
\node[left, cbVermillion] at (-0.2,-0.45) {\tiny $P_1$};
\node[left, cbOrange] at (-0.2,-0.15) {\tiny $P_2$};
\node[left, cbGreen] at (-0.2, 0.15) {\tiny $P_3$};
\node[left, cbSkyBlue] at (-0.2, 0.45) {\tiny $P_4$};
\node[left, cbPurple] at (-0.2, 0.75) {\tiny $P_5$};
    
\end{scope}
\end{tikzpicture}